\documentclass[10pt]{article}
\newif\ifarxiv
\arxivtrue
\usepackage[letterpaper, margin=1in]{geometry}
\usepackage[singlespacing]{setspace}
\usepackage{caption}
\usepackage{titlesec}
\titleformat*{\section}{\Large\bfseries}
\titleformat*{\subsection}{\large\bfseries}

\usepackage{amsmath,amssymb,graphicx,amscd,amsthm,dsfont,bm,mathtools}
\graphicspath{ {figures/} }
\usepackage[table]{xcolor}
\usepackage{enumitem}
\usepackage{diagbox}
\usepackage{hyperref}
\usepackage{array}
\usepackage[capitalise,noabbrev,nosort]{cleveref}
\usepackage[authoryear]{natbib}

\newcommand{\suppref}[1]{%
    \ifarxiv
        \cref{#1}%
    \else
        the Supplement%
    \fi
}
\newcommand{\supprefcap}[1]{%
    \ifarxiv
        \cref{#1}%
    \else
        The Supplement%
    \fi
}
\newcommand{\mainlabel}[1]{%
    \ifarxiv
        #1%
    \else
        main-#1%
    \fi
}
\newcommand{\op}[1]{{\operatorname{#1}}}

\newcommand{\m}[1]{\mathbf{#1}}
\newcommand{\mh}[1]{\widehat{\mathbf{#1}}}

\newcommand{\mc}[1]{\mathcal{#1}}
\newcommand{\hc}[1]{\widehat{\mathcal{#1}}}

\newcommand{\Rs}[2]{\m{R}(#1, #2)}
\newcommand{\RU}{\m{R}_{\m{U}}}

\newcommand{\dd}{\mathrm{d}}
\newcommand{\T}{^\mathsf{T}}

\DeclareMathOperator*{\argmin}{arg\,min}
\newcommand{\R}{\mathbb{R}}
\newcommand{\C}{\mathbb{C}}

\newcommand{\all}{\boldsymbol{\cdot}}

\newcommand{\rank}{\operatorname{rank}}

\newcommand{\plr}[1]{\left(#1\right)}
\newcommand{\blr}[1]{\left[#1\right]}
\newcommand{\clr}[1]{\left\{#1\right\}}
\newcommand{\alr}[1]{\left|#1\right|}
\newcommand{\dpar}[1]{[\![ #1 ]\!]}
\newcommand{\Norm}[2]{\left\| #1 \right\|_{#2}}
\newcommand{\norm}[2]{\| #1 \|_{#2}}
\newcommand{\tnorm}[1]{\norm{#1}{2}}

\newcommand{\fnorm}[1]{\norm{#1}{\op{F}}}

\newcommand{\tinorm}[1]{\norm{#1}{2,\infty}}

\newcommand{\bbmat}{\begin{bmatrix}}
\newcommand{\ebmat}{\end{bmatrix}}

\newcommand{\inv}{^{-1}}

\newcommand{\I}[1]{\mathds{1}{\left\{#1\right\}}}
\newcommand{\Is}[1]{\mathds{1}{\{#1\}}}

\newcommand{\1}{\mathbf{1}}
\newcommand{\ir}{\mathrm{i}} 
\newcommand{\spt}{\widehat{t}} 
\newcommand{\svec}{{\m{s}}}

\newcommand{\svectilde}{\tilde{\m{s}}}
\newcommand{\tstat}{{T_{\svec,n}}}
\newcommand{\tstata}{{T_{\svec,n,0}}}
\newcommand{\tplug}{{\widetilde{T}_{\svectilde,n}}}
\newcommand{\tfrob}{{T_{\operatorname{Frob}}}}
\newcommand{\tfplug}{{\widetilde{T}_{\operatorname{Frob}}}}
\newcommand{\tpluguc}{{\widetilde{T}_{\hat{\m{s}},n}^{\operatorname{uc}}}}
\newcommand{\abias}{A_{\svec,n}}
\newcommand{\abplug}{\widetilde{A}_{\svectilde,n}}

\newcommand{\gumbelrv}{{G}}
\newcommand{\sgn}{\operatorname{{sgn}}}
\newcommand{\loc}[1]{\operatorname{RD}_{n}(#1)}
\newcommand{\mul}{\ell}
\newcommand{\rlog}{{\zeta_{r,n}}}
\newcommand{\acdf}{{F_{n,0}}}
\newcommand{\xp}{{X^{\prime}}}

 \newcommand{\nmc}{N_{\operatorname{MC}}}

\newcommand{\NN}{{\mathcal{N}}}

\newcommand{\ber}[1]{\operatorname{Bernoulli}\!\left(#1\right)}

\newcommand{\pois}[1]{\operatorname{Poisson}\!\left(#1\right)}

\newcommand{\eqd}{\overset{\operatorname{d}}{=}}

\renewcommand{\P}{{\mathbb{P}}}
 \DeclareMathOperator{\E}{\mathbb{E}}      

\newcommand{\hypn}{\operatorname{H}_{0}}
\newcommand{\hypa}{\operatorname{H}_{\operatorname{A}}}
\newcommand{\vs}{\text{versus}}

 \newcommand{\diag}{\operatorname{diag}}

\newtheorem{theorem}{Theorem}
\newtheorem{lemma}[theorem]{Lemma}
\newtheorem{corollary}[theorem]{Corollary}
\newtheorem{proposition}[theorem]{Proposition}

\theoremstyle{definition}
\newtheorem{definition}{Definition}
\newtheorem{assumption}{Assumption}
\newtheorem{remark}{Remark}

\crefname{theorem}{Theorem}{Theorems}
\crefname{lemma}{Lemma}{Lemmas}
\crefname{corollary}{Corollary}{Corollaries}
\crefname{proposition}{Proposition}{Propositions}
\crefname{definition}{Definition}{Definitions}
\crefname{assumption}{Assumption}{Assumptions}
\crefname{remark}{Remark}{Remarks}
\crefname{comment}{Comment}{Comments}

\setlist{itemsep=0.5mm, leftmargin=5mm}
\hypersetup{
	colorlinks = true,
    linkcolor = blue,
	citecolor = cyan,
    urlcolor = blue,
	}

\title{Extreme value theory for singular subspace estimation in the matrix denoising model}

\author{Junhyung Chang\footnote{Department of Statistics, University of Wisconsin--Madison. Email: \textsf{jchang243@wisc.edu}} \and Joshua Cape\footnote{Department of Statistics, University of Wisconsin--Madison. Email: \textsf{jrcape@wisc.edu}}}

\date{\today}

\begin{document}

\maketitle

\begin{abstract}
    This paper studies fine-grained singular subspace estimation in the matrix denoising model where a deterministic low-rank signal matrix is additively perturbed by a stochastic matrix of Gaussian noise.
We establish that the maximum Euclidean row norm (i.e., the two-to-infinity norm) of the aligned difference between the leading sample and population singular vectors approaches the Gumbel distribution in the large-matrix limit, under suitable signal-to-noise conditions and after appropriate centering and scaling.
We apply our novel asymptotic distributional theory to test hypotheses of low-rank signal structure encoded in the leading singular vectors and their corresponding principal subspace.
We provide de-biased estimators for the corresponding nuisance signal singular values and show that our proposed plug-in test statistic has desirable properties.
Notably, compared to using the Frobenius norm subspace distance, our test statistic based on the two-to-infinity norm empirically has higher power to detect structured alternatives that differ from the null in only a few matrix entries or rows.
Our main results are obtained by a novel synthesis of and technical analysis involving row-wise matrix perturbation analysis, extreme value theory, saddle point approximation methods, and random matrix theory.
Our contributions complement the existing literature for matrix denoising focused on minimaxity, mean squared error analysis, unitarily invariant distances between subspaces, component-wise asymptotic distributional theory, and row-wise uniform error bounds.
Numerical simulations illustrate our main results and demonstrate the robustness properties of our testing procedure to non-Gaussian noise distributions.
\end{abstract}

\section{Introduction}
\label{body-section:introduction}

This paper considers noisy observable data matrices $\mh{M} \in \R^{n \times m}$ of the form
\begin{align}
    \label{body-eq:matrix-denoising}
    \mh{M}
    \coloneqq
    \m{M}
    +
    \m{E}
\end{align}
where $\m{M} \in \R^{n \times m}$ is an unobserved deterministic signal matrix with $r \coloneqq \rank(\m{M}) \ll \min\{n, m\}$, and where $\m{E} \in \R^{n \times m}$ is an unobserved stochastic noise matrix having mean-zero Gaussian entries.
Our goal is to infer properties of the signal matrix $\m{M}$ whose singular value decomposition (SVD) is written as $\m{USV}\T$.
Here, $\m{U}$ and $\m{V}$ denote matrices whose orthonormal columns are the leading left and right singular vectors of $\m{M}$, respectively, while the diagonal matrix $\m{S} = \diag(s_{1}, \dots, s_{r})$ contains the non-zero singular values of $\m{M}$, where $s_{1} \ge \dots \ge s_{r} > 0$.
We write the corresponding top-$r$ SVD of $\mh{M}$ as $\mh{U}\mh{S}\mh{V}\T$ provided it is well defined.

This paper focuses on $\m{U}$ and $\mh{U}$, though the results and machinery herein can be similarly adapted to $\m{V}$ and $\mh{V}$.
In particular, this paper studies fine-grained estimation of the left singular vectors, $\m{U}$, using the observable noisy counterpart $\mh{U}$.
To do so, we consider the difference $\mh{U}\RU - \m{U}$, where $\RU$ denotes an $r \times r$ orthogonal matrix that accounts for the rotational ambiguity inherent in subspace estimation.
Concretely, we take $\RU \coloneqq \sgn(\mh{U}\T \m{U})$, where the sign of a square matrix, $\sgn(\cdot)$, is defined as the orthogonal factor of its polar decomposition, which minimizes the Frobenius norm of $\mh{U}\m{O} - \m{U}$ over the set of $r \times r$ orthogonal matrices $\m{O}$ \citep{schonemann_generalized_1966}.

Instead of considering the more classical Frobenius norm or spectral norm, both of which are insufficient for fine-grained analysis, this paper undertakes a detailed and novel investigation of
\begin{equation}
    \label{body-eq:2-inf-norm-property}
    \tinorm{\mh{U} \RU - \m{U}}
\end{equation}
where $\tinorm{\cdot}$ denotes the two-to-infinity norm of a matrix \citep{cape_two--infinity_2019,chen_spectral_2021}.
In particular, the two-to-infinity norm, also known as the maximum Euclidean row norm, is given in terms of the $\ell_{2}$ and $\ell_{\infty}$ vector norms by
\begin{equation}
    \label{body-eq:tinorm-statement}
    \tinorm{\m{A}}
    \coloneqq
    \sup_{\|\m{x}\|_{2} = 1}
    \|\m{Ax}\|_{\infty}
    \equiv
    \max_{1 \le i \le n}
    \tnorm{\m{A}_{i,\all}}
\end{equation}
where $\m{A}_{i,\all}$ denotes the $i$-th row vector of the matrix $\m{A}$.
This norm has received considerable attention within statistics and data science in recent years, as detailed in \cref{body-section:related_work}.
Although numerous two-to-infinity norm bounds and row-specific estimation results have been obtained in the literature on matrix denoising and related problems, such as principal component analysis, factor analysis, matrix completion, and random graph inference, no prior works develop the distributional properties of the two-to-infinity norm or its rigorous use for statistical hypothesis testing.
Given the inherent extremal nature of the two-to-infinity norm, our analysis necessitates novel technical analysis encompassing spectral methods, function approximation, and extreme value theory.

\subsection{Contributions}

The main purpose of this paper is to establish the asymptotic distributional properties of $\tinorm{\mh{U} \RU - \m{U}}$ and to subsequently develop downstream inferential applications.
To do so, we employ a combination of row-wise eigenvector perturbation analysis, random matrix theory, saddle point density approximation methods, and extreme value theory but in a non-traditional high-dimensional setting involving a maximum of the form $\max_{1 \le i \le n} X_{i,n,m}$, where each random variable $X_{i,n,m}$ changes with $n$ and $m \equiv m(n)$ as $n, m \to \infty$.

As will be explained later, our distributional theory enables methodology for addressing hypothesis testing problems that involve signal subspace structure, loosely speaking of the form
$
\hypn:
\m{U} = \m{U}_{0}
$
versus
$
\hypa:
\m{U} = \m{U}_{1}
$.
Our maximum-type test statistic involving $\tinorm{\mh{U} \RU - \m{U}}$ is sensitive to alternatives $\m{U}_{1}$ that differ from the null $\m{U}_{0}$ in only a few entries or rows, unlike summation-type test statistics utilizing the Frobenius norm difference between orthogonal projection matrices. 

A key technical contribution of this paper is that we permit the smallest non-zero signal singular value $s_{r}$ to have general multiplicity $\mul$ where $1 \leq \mul \leq r$.
Our analysis characterizes the influence of $\mul$ on the extreme value asymptotics for subspace estimation using the two-to-infinity norm.

Our main contributions are summarized as follows.
\begin{enumerate}
    \item We establish novel extreme value distributional theory for fine-grained singular subspace estimation in the Gaussian matrix denoising model by proving that $\tinorm{\mh{U} \RU - \m{U}}$ converges in distribution to a standard Gumbel random variable in the large-matrix limit after suitable centering and scaling.
    See \cref{body-proposition:saddle-point-tail-equivalence,body-theorem:Gumbel-convergence}.
    
    \item We derive a non-asymptotic Kolmogorov--Smirnov distance error bound between the cumulative distribution function of $\tinorm{\mh{U} \RU - \m{U}}$, after suitable centering and scaling, and that of the limiting Gumbel distribution.
    See \cref{body-theorem:CDF-bounds}.
    
    \item Viewing singular values as nuisance parameters in the subspace estimation framework, we derive de-biased estimators of the signal singular values to establish the convergence in distribution of a data-driven plug-in test statistic.
    See \cref{body-proposition:debiased-singular-values,body-theorem:Gumbel-convergence-plugin}.
    
    \item Our theoretical developments enable consistent testing for singular subspace structure in the signal matrix for which we provide a power analysis under simple alternatives.
    See \cref{body-theorem:power-analysis}.
    
    \item Our extreme value-based test empirically has higher power to detect structured alternatives that differ from the null in only a few matrix entries or rows, compared to an existing asymptotically Gaussian test statistic based on the Frobenius norm and differences between orthogonal projections.
    See \cref{body-section:power-against-local-alternatives}.
    
    \item Simulation studies illustrate our main results and empirically demonstrate the robustness of our testing framework to departures from stated modeling assumptions.
    See \cref{body-section:departures-from-Gaussianity} and the Supplement.
\end{enumerate}

Although not the main focus of this paper, we also briefly discuss extending our extreme value analysis to non-Gaussian noise settings.
See \cref{body-proposition:non-gaussian-noise-rank-1}.

\subsection{Related work}
\label{body-section:related_work}

The problem of \emph{matrix denoising} commonly refers to estimating an unobserved deterministic low-rank signal matrix from its observed noisy counterpart, often modeled as in \cref{body-eq:matrix-denoising}.
Historically, interest in estimating the low-rank signal matrix has led to the development of \emph{singular value shrinkage} methods, with the goal of transforming sample singular values to correct for the inflation or bias caused by the noise matrix.
Notable examples include low-rank thresholding \citep{cai_singular_2010,gavish_optimal_2014,donoho_minimax_2014,chatterjee_matrix_2015} and non-linear shrinkage  \citep{shabalin_reconstruction_2013,nadakuditi_optshrink_2014,gavish_optimal_2017,donoho_optimal_2018,gavish_matrix_2023,su_data-driven_2025}.
More recently, considerable attention has been given to the related problem of \emph{subspace estimation} which pertains to estimating the principal singular subspaces of the signal matrix.
The study of subspace estimation originally focused on error bounds under the Frobenius or spectral norm \citep{davis_rotation_1970,wedin_perturbation_1972,cai_rate-optimal_2018,ding_high_2020,cai_subspace_2021,orourke_matrices_2023} and has advanced to a more delicate understanding of row-wise and entry-wise error analysis, via fine-grained eigenvector perturbation theory \citep{fan_ell_infty_2018,cape_signal-plus-noise_2019,cape_two--infinity_2019,abbe_entrywise_2020,damle_uniform_2020,agterberg_entrywise_2021,wang_analysis_2024,tang_eigenvector_2025}. 

Beyond the derivation of error bounds, the study of asymptotic distributional theory and the goal of developing inference procedures for principal subspace estimation have both emerged as highly active research topics \citep{cape_signal-plus-noise_2019,xia_normal_2021,bao_singular_2021,xia_statistical_2021,cheng_tackling_2021,agterberg_entrywise_2021,fan_asymptotic_2022,fan_simple_2022,xie_entrywise_2024,yan_inference_2024,athreya_limit_2016,tang_semiparametric_2017,tang_limit_2018,li_two-sample_2018,du_hypothesis_2023,yan_entrywise_2024}.
At a high level, the study of asymptotic distributional theory for subspace estimation in the matrix denoising model can be divided into two categories.
On the one hand, there is distributional theory for quantities related to the subspace spanned by a subset of singular vectors or eigenvectors, such as the norm of the difference of orthogonal projection matrices 
\citep{xia_normal_2021,bao_singular_2021,tang_nonparametric_2017,li_two-sample_2018}. 
On the other hand, there is distributional theory pertaining to individual rows or entries of concatenated matrices of singular vectors or eigenvectors \citep{cape_signal-plus-noise_2019,xie_entrywise_2024,fan_simple_2022,yan_inference_2024,athreya_limit_2016,tang_semiparametric_2017,tang_limit_2018,agterberg_entrywise_2021,rubin-delanchy_statistical_2022,du_hypothesis_2023,yan_entrywise_2024}.
At the intersection of these two categories lies the distributional theory of linear forms of singular vectors or eigenvectors, enabling the study of individual components as well as linear combinations thereof \citep{fan_asymptotic_2022,cheng_tackling_2021,xia_statistical_2021,agterberg_distributional_2024}.
Both categories pertain to this paper and are thus discussed in more detail below.

Within the first category, several existing works study the distributional theory of singular subspace estimation in (typically unitarily invariant) matrix norms \citep{xia_normal_2021,bao_singular_2021,tang_nonparametric_2017,li_two-sample_2018}.
\citet{xia_normal_2021} proved a non-asymptotic normal approximation for a statistic involving the Frobenius norm distance between subspaces via a representation formula for differences of orthogonal projections under i.i.d. Gaussian noise.
\citet{bao_singular_2021} showed that a statistic also related to the Frobenius norm projection distance is asymptotically normal under i.i.d. Gaussian noise.
Other related works include \citet{tang_semiparametric_2017,li_two-sample_2018} which developed two-sample testing methodology based on the limiting distribution of Procrustes-type subspace distances in the Frobenius norm.

In the second category, which focuses on the asymptotic normality of individual rows or entries of matrices of singular vectors or eigenvectors, \citet{cape_signal-plus-noise_2019} established row-wise normality of the Procrustes type subspace difference $\mh{U}\RU - \m{U}$.
\citet{xie_entrywise_2024} further developed non-asymptotic Berry--Esseen-type bounds for the rows of the Procrustes-type subspace differences $\mh{U}\RU - \m{U}$ in a weak signal regime.
\citet{agterberg_entrywise_2021} established the limiting distribution of rows and entries under heteroskedastic and possibly dependent noise.
In the related setting of latent space network analysis, \citet{athreya_limit_2016,tang_limit_2018,rubin-delanchy_statistical_2022} derived the limiting distribution of individual rows or entries of spectral graph embeddings involving the adjacency matrix and graph Laplacian.
\citet{fan_simple_2022,du_hypothesis_2023} used the limiting distribution of individual rows of the adjacency or Laplacian embedding to conduct node-wise latent position network hypothesis tests.

As for the distributional theory of linear forms of singular vectors or eigenvectors, at the intersection of the aforementioned two categories,
\citet{cheng_tackling_2021} showed the asymptotic normality of inner products between a fixed vector and an individual observed eigenvector under asymmetry, heteroskedastic noise, and small eigengaps.
\citet{agterberg_distributional_2024} established a Berry--Esseen-type bound for a de-biased inner product between a fixed vector and an individual observed eigenvector under a small-eigengap condition in the context of both matrix denoising and principal component analysis.
\citet{fan_asymptotic_2022} derived the asymptotic normality of linear and bilinear forms involving individual signal eigenvectors and observed eigenvectors of symmetric matrices.
\citet{xia_statistical_2021} extended the result of \citet{xia_normal_2021} to enable inference for linear forms in the trace regression setting.

A closely related topic is the distributional theory of eigenvectors of sample covariance matrices under the spiked covariance model \citep{yan_inference_2024,bao_statistical_2020,koltchinskii_asymptotics_2016,koltchinskii_efficient_2020,bloemendal_principal_2014}. 
Notably, \citet{yan_inference_2024} established a non-asymptotic Gaussian approximation for the rows of the principal subspace matrix, as well as individual entries in the covariance matrix estimator under heteroskedastic noise and missing data.
En route, the authors also obtain performance guarantees for matrix denoising using the output of their estimation algorithm.

Finally, we note that the Gumbel distribution (the limiting distribution of our test statistic) has appeared elsewhere in the literature on extreme value-based testing procedures though for different problem settings \citep{donoho_higher_2004,hu_using_2021,fan_simple-rc_2022}.
Briefly, \citet{donoho_higher_2004} developed a higher criticism test statistic for the multiple testing problem, where the test statistic is obtained by taking the maximum over an empirical distribution of $p$-values.
\citet{hu_using_2021} proposed a goodness-of-fit test for the stochastic blockmodel by examining the maximum entry-wise deviation of the centered and rescaled adjacency matrix of a network.
\citet{fan_simple-rc_2022} considered testing whether a group of network nodes share similar membership profiles by apply the test in \citet{fan_simple_2022} to pairs of nodes in a group and then computing the maximum thereof.

\subsection{Notation}
\label{body-section:notation}

The symbols $\coloneqq$ and $\eqqcolon$ are used to assign definitions.
The upper-case blackboard bold letters $\R$, $\C$, and $\mathbb{N}$ are reserved for the set of real numbers, complex numbers, and natural numbers, respectively. 
Given $N \in \mathbb{N}$, let $\dpar{N}$ denote the set $\{1, 2, \dots,N\}$.
Given a matrix $\m{A} \in \R^{n \times m}$ having rank $r$, let $s_{1}(\m{A}) \ge s_{2}(\m{A}) \ge \dots \ge s_{r}(\m{A}) > 0$ denote its non-zero singular values.
The $i$-th row of $\m{A}$ is denoted by $\m{A}_{i,\all}$ whereas the $i$-th column is denoted by $\m{A}_{\all,i}$, both of which are treated as column vectors.
In particular, $\m{A}\T_{i,\all}$ denotes the $i$-th row of $\m{A}$ as a row vector. 
The $(i,j)$-th entry of $\m{A}$ is denoted by $A_{i,j}$.
We write $\1_{n}$ to denote the vector of dimension $n$ with entries all equal to one.

The (univariate) normal or Gaussian distribution with implicit mean and variance is denoted by $\NN(\cdot, \cdot)$, whereas $\NN_{d}(\cdot, \cdot)$ denotes the multivariate setting with dimension $d \ge 2$.
Given a sequence of random variables $(X_{n})_{n \ge 1}$ and a target random variable $X$, we write $X_{n} \rightsquigarrow X$ when $(X_{n})_{n \ge 1}$ convergences in distribution to $X$.
For two random variables $X$ and $Y$ that are equal in distribution, we write $X \eqd Y$.
Given a real-valued random variable $X$, let $F_{X}$ denote its cumulative distribution function, and let $f_{X}$ denote its probability density function provided $F_{X}$ is differentiable almost everywhere.
The moment generating function (MGF) of $X$ is given by $M_{X}(t) \coloneqq \E[\exp(t X)]$, provided that the expectation exists for all real-valued $t$ in an open neighborhood of zero.
The cumulant generating function (CGF) of $Z$ is given by $K_{Z}(t) \coloneqq \log(M_{Z}(t))$.
The characteristic function (CF) of $Z$ is given by $\varphi_{Z}(t)\coloneqq \E[\exp(\ir t Z)]$, where $\ir \coloneqq \sqrt{-1}$.

Given two real-valued sequences $(a_{n})_{n \ge 1}$ and $(b_{n})_{n \ge 1}$, we write $a_{n} \lesssim b_{n}$ if there exists a constant $C > 0$ and a natural number $n_{0}$ depending only on $C$ such that $|a_{n}| \leq C |b_{n}|$ for all $n \geq n_{0}$.
We write $a_{n} \gtrsim b_{n}$ if $b_{n} \lesssim a_{n}$.
We also write $a_{n} = O(b_{n})$ if $a_{n} \lesssim b_{n}$ and $a_{n} = \Omega(b_{n})$ if $a_{n} \gtrsim b_{n}$.
We write $a_{n} \ll b_{n}$ if for all $\varepsilon > 0$ there exists a natural number $n_{0}$ only depending on $\varepsilon$ such that $|a_{n}| \leq \varepsilon |b_{n}|$ when $n \geq n_{0}$.
The notation $a_{n}\gg b_{n}$ indicates that $b_{n} \ll a_{n}$, and $a_{n} = o(b_{n})$ if $a_{n} \ll b_{n}$.
We write $a_{n} \simeq b_{n}$ if $a_{n} \lesssim b_{n}$ and $a_{n} \gtrsim b_{n}$.

\subsection{Organization}

The remainder of this paper is organized as follows.
\cref{body-section:preliminaries} specifies the matrix denoising model assumptions and presents a first-order approximation of singular vector perturbations, as a starting point for the subsequent distributional theory. 
\cref{body-section:Main-results} presents the main results of this paper, namely the asymptotic distributional characterization of test statistics based on the two-to-infinity norm, both when the population signal singular values are assumed known and when they first need to be estimated.
These test statistics are in turn leveraged for hypothesis testing involving the left signal singular vectors.
This section concludes with a brief treatment of non-Gaussian noise in a rank-one signal matrix setting with delocalized singular vectors.
\cref{body-section:Numerical-simulations} provides numerical examples illustrating the main results and their finite-sample properties across a variety of simulation settings, including comparisons to a different, less fine-grained test statistic based instead on the Frobenius norm, and explores the robustness properties of the test statistic under non-Gaussian noise.
\cref{body-section:discussion} provides additional discussion and concluding remarks. 
The Supplement contains proofs of the main results, supporting technical derivations, and additional experiments investigating the robustness of the main results to departures from stated assumptions.

\section{Preliminaries}
\label{body-section:preliminaries}

\subsection{Model assumptions}
\label{body-section:model-asm}

This paper considers the matrix denoising model in \cref{body-eq:matrix-denoising} under the following specifications.

\begin{definition}[Signal matrix quantities]
\label{body-defn:signal-matrix-quantities}
The following quantities are related to the signal matrix $\m{M}$, and are frequently used throughout the paper.
    \begin{enumerate}[label=(\roman*)]
        \item Let $\mul$ denote the multiplicity of the smallest non-zero signal singular value $s_{r}$.
        
        \item Let $\mu$ denote the coherence parameter
        \begin{equation*}
            \mu
            \coloneqq
            \max
            \clr{
            \frac{n}{r}\tinorm{\m{U}}^{2}, 
            \frac{m}{r}\tinorm{\m{V}}^{2}
            },
            \qquad
            1
            \le
            \mu
            \le
            \frac{\max\{n, m\}}{r}
            .
        \end{equation*}
        
        \item Define $\rlog \coloneqq r \log n + \log^{2} n$.
    \end{enumerate}
    
\end{definition}

\begin{assumption}[Row-wise Gaussian noise]
    \label{body-assumption:noise}
    The $n \times m$ noise matrix $\m{E}$ has independent and identically distributed rows satisfying $\m{E}_{i,\all} \sim \NN_{m}(\m{0}, \m{D})$ as follows.
    \begin{itemize}
        \item If $\mul = 1$, then assume $\m{D} = \diag(\sigma_{1}^{2}, \dots, \sigma_{m}^{2})$, where 
        $
        0
        <
        \sigma_{m}^{2}
        \leq
        \dots
        \leq
        \sigma_{1}^{2}
        \leq
        \sigma^{2}
        < \infty
        $
        for some $\sigma^{2}$.
        
        \item If $\mul \geq 2$, then assume $\m{D} = \sigma^{2} \m{I}_{m}$ for some
        $
        0
        <
        \sigma^{2}
        <
        \infty
        $.
    \end{itemize}
\end{assumption}

\begin{assumption}[Matrix aspect ratio]
    \label{body-assumption:matrix-size}
    The matrix dimensions $n$ and $m \equiv m(n)$ satisfy $n / m \simeq 1$ as $n \to \infty$.
\end{assumption}

\begin{assumption}[Signal singular vectors]
    \label{body-assumption:delocalization}
    The coherence parameter satisfies $\mu \ll n / (r  \rlog)$, where $\mu$ and $\rlog$ are defined in \cref{body-defn:signal-matrix-quantities}.
\end{assumption}

\begin{assumption}[Signal singular values]
    \label{body-assumption:snr}
    The rank of the signal matrix, denoted by $r$, and the multiplicity of the smallest non-zero signal singular value $s_{r}$, denoted by $\mul$, are fixed.
    Further, 
    $s_{r} \gg \sqrt{ \mu \rlog } (\sigma \sqrt{n})$ holds, where $\mu$ and $\rlog$ are defined in \cref{body-defn:signal-matrix-quantities}.
\end{assumption}

\cref{body-assumption:noise} posits Gaussian noise which is by far the most common and widely studied setting considered in the matrix denoising literature.
The ordering of the noise variances is made for notational convenience and holds without loss of generality, since it can always be achieved by applying an appropriate permutation.
Here, the influence of the multiplicity $\mul$ on the permissible structure of the covariance matrix $\m{D}$ reflects a trade-off between flexibility and specificity when modeling the signal and noise.
We emphasize that there are no restrictions on the multiplicities of the other non-zero signal singular values $s_{1}, \dots, s_{r-\mul}$ when $r \geq \mul + 1$.

\cref{body-assumption:matrix-size} posits that the matrix dimensions grow at a comparable rate.
This condition is often encountered in the existing literature and is used here to enable studying the concentration properties of linear and quadratic forms involving $\m{E}$ and $\m{E}\T$.

\cref{body-assumption:delocalization} posits a mild upper bound on the permitted localization of mass or energy in the individual entries and rows of the matrices of left and right signal singular vectors.
This assumption allows the entries of the signal singular vectors to be highly non-uniform in magnitude, preventing only near-localization edge cases.
Assumptions of this type (moreover, often stronger assumptions requiring delocalization) are ubiquitous in the modern literature on entry-wise and row-wise matrix perturbation analysis, extending well beyond the matrix denoising problem studied in this paper.
Here, the condition on $\mu$ is crucial in our development of asymptotic distributional theory, rather than for preliminary singular subspace perturbation analysis.

\cref{body-assumption:snr} posits that the smallest non-zero signal singular value diverges sufficiently quickly, interpretable as a form of population-level signal strength.
Conditions of this type, too, are pervasive in the literature on spectral methods.
Concretely, when $\mul, r, \mu \simeq 1$,
then $s_{r}$ is required to grow only slightly faster than the noise level rate $\sigma \sqrt{n}$ which reflects the operator norm concentration of $\m{E}$ when $n$ and $m$ are comparable.

\subsection{Row-wise perturbation analysis and first-order approximation}
\label{body-section:row-wise-analysis-of-first-order-approx}

The row-specific concentration and distributional properties of $\mh{U}\RU - \m{U}$ are prerequisite for obtaining the limiting distribution of $\tinorm{\mh{U}\RU - \m{U}}$.
To achieve our stated goals, we employ \cref{body-lemma:first-order-approximation}, a consequence of \citet[Proposition 2]{yan_entrywise_2024}, which establishes uniform row-wise error bounds for approximating the leading left singular vectors of the signal matrix.
The proof of \cref{body-lemma:first-order-approximation} is provided in \suppref{supp-pf:first-order-approximation}.

\begin{lemma}[First-order approximation of the left principal subspace]
    \label{body-lemma:first-order-approximation}
    Under \cref{body-assumption:noise,body-assumption:matrix-size,body-assumption:snr}, the bounds
    \begin{align}
        \tinorm{\mh{U}\RU - \m{U}}
        \lesssim
        \frac{\sigma\sqrt{r + \log n}}
        {s_{r}} 
        +
        \frac{\sigma^{2}\sqrt{\mu r n}}
        {s_{r}^{2}}
        \label{body-eq:perturbation-bound}
    \end{align}
    and
    \begin{align}
        \tinorm{\mh{U}\RU - \m{U} - \m{EVS}\inv}
        \lesssim
        \frac{\sigma\sqrt{r + \log n}}{s_{r}}
        \plr{
            \frac{\sigma\sqrt{n}}{s_{r}}
            +
            \sqrt{
            \frac{\mu r}{n}
            }
        }
        +
        \frac{\sigma^{2}\sqrt{\mu r n}}
        {s_{r}^{2}}
        \label{body-eq:first-order-approximation}
    \end{align}
    each hold with probability at least $1 - O(n^{-9})$.
\end{lemma}

Importantly, \cref{body-lemma:first-order-approximation} establishes that, under certain conditions, $\m{U} + \m{EVS}\inv$ is a better uniform row-wise approximation of $\mh{U}\RU$ than is $\m{U}$ alone.
Moreover, these perturbation bounds justify treating $\m{EVS}\inv$ as a fine-grained approximation of $\mh{U}\RU - \m{U}$, in pursuit of establishing the limiting distribution of $\tinorm{\mh{U}\RU - \m{U}}$.

Given that we are considering the maximum Euclidean row norm, $\tinorm{\cdot}$, and since the rows of $\m{EVS}\inv$ are independent and identically distributed as a consequence of \cref{body-assumption:noise}, we are led to investigate the largest order statistic
\begin{align}
    \label{body-eq:Xi-quadratic-form}
    \max_{i \in \dpar{n}}
    X_{i}
    ,
    \qquad
    \text{ where }
    X_{i}
    \equiv
    X_{i,n,m}
    \coloneqq
    \tnorm{\m{E}_{i, \all}\T\m{VS}\inv}
    .
\end{align} 
Here, $\m{E}_{i, \all}\T\m{VS}\inv \sim \NN_{r}(\m{0}, \m{S}\inv\m{V}\T\m{D}\m{V}\m{S}\inv)$ for each $i$, so it follows that $X_{i}^{2}$ is a sum of $r$ (possibly non-identical and non-independent) squared zero-mean Gaussian random variables.
Though admitting non-trivial structure, the moment generating function of quadratic forms of Gaussian vectors such as $X_{i}^{2}$ is understood \citep{mathai_quadratic_1992}, and this is used heavily in the proofs of our main results via the following technical lemma.
The proof of \cref{body-lemma:MGF} is provided in \suppref{supp-pf:MGF}.
    
\begin{lemma}[Moment generating function of $X_{i}^{2}$]
    \label{body-lemma:MGF}
    Define the matrix $\m{P} \coloneqq 2\m{S}\inv\m{V}\T\m{D}\m{V}\m{S}\inv \in \R^{r \times r}$.
    The following properties hold.
        \begin{enumerate}
            \item The matrix $\m{P}$ has eigenvalues $\lambda_{1} \geq \dots \geq \lambda_{r} > 0$ where $\lambda_{1} \in [2\sigma_{m}^{2} / s_{r}^{2}, 2\sigma_{1}^{2} / s_{r}^{2}]$.
            In particular, if $\m{D} = \sigma^{2}\m{I}_{m}$, then 
        \begin{align*}
            (\lambda_{1},
            \dots,
            \lambda_{r})
            =
            \left(
            \frac{2\sigma^{2}}{s_{r}^{2}}, 
            \dots,
            \frac{2\sigma^{2}}{s_{1}^{2}}
            \right)
            .
        \end{align*}
            \item Under \cref{body-assumption:noise}, the moment generating function of $X_{i}^{2}$ is given by
        \begin{align*}
            M_{X_{i}^{2}}(t)
            =
            \det(\m{I}_{r} - t\m{P})^{-1/2}
            =
            \prod_{j=1}^{r}
            (1-\lambda_{j} t)^{-1/2},
            \qquad
            t < 1/\lambda_{1}
            .
        \end{align*}
        \end{enumerate}
\end{lemma}

Notably, under \cref{body-assumption:noise}, it holds that $\lambda_{1} = \cdots = \lambda_{\mul}$.
In addition, the eigenvalues $(\lambda_{j})_{j \in \dpar{r}}$ are functions of the (diverging) signal singular values $(s_{j})_{j \in \dpar{r}}$ which change with $n$ and $m$.
Put differently, each random variable $X_{i}$ in $\max_{i \in \dpar{n}} X_{i}$ changes with $n$ and $m$, which in turn affects the scaling of the maximum, thus requiring careful analysis.

To avoid possible degenerate or edge-type cases in \cref{body-section:Main-results} when $r > 1$ and $\mul < r$, we impose an additional mild assumption on the eigenstructure of $\m{P}$.
In particular, \cref{body-assumption:gap-lam1-lam2} prevents $\lambda_{\mul + 1} / \lambda_{\mul} \to 1$ as $n \to \infty$, which can be described as $\lambda_{\mul}$ being asymptotically indistinguishable from $\lambda_{\mul + 1}$.
\cref{body-assumption:gap-lam1-lam2} holds for many popular choices of $\m{M}$ and $\m{D}$.
For example, if $\m{D} = \sigma^{2}\m{I}_{m}$, then this condition holds for any signal matrix $\m{M}$ satisfying $s_{r} / s_{r - \mul} < \sqrt{c_{0}}$ for some constant $0 < c_{0} < 1$ when $n$ is sufficiently large.

\begin{assumption}[Distinguishable multiplicity]
    \label{body-assumption:gap-lam1-lam2}
    If $r > 1$ and $\mul < r$, assume there exist constants $n_{0} \in \mathbb{N}$ and $c_{0} \in (0,1)$ such that $\lambda_{\mul + 1} / \lambda_{\mul} < c_{0}$ holds for all $n \geq n_{0}$.
\end{assumption}

Before proceeding to the next section, we reiterate that $\lambda_{1} = \lambda_{\mul}$ holds no matter the value of $\mul$ in \cref{body-assumption:noise}.
Thus, for simplicity of presentation and to avoid potential confusion regarding the downstream role of $\mul$ in what follows, we write $\lambda_{1}$ rather than $\lambda_{\mul}$.

\section{Main results: distributional theory and hypothesis testing}
\label{body-section:Main-results}

\subsection{Extreme value analysis for two-to-infinity norm subspace estimation}
\label{body-section:gumbel-convergence}

For ease of notation, in this section we let $X$ denote an independent and identically distributed copy of $X_{1}$ in \cref{body-eq:Xi-quadratic-form}.
In order to obtain normalizing sequences in $n$ for $X$, this section identifies a random variable (more precisely, a sequence of random variables indexed by $n$) that is tail equivalent to $X$ \citep[Section 1.5]{resnick_extreme_2007}, namely a random variable whose distribution shares upper-tail properties with $X$ and whose normalizing sequences are comparatively more straightforward to compute.

\begin{definition}[Tail equivalence]
    \label{body-defn:tail-equivalence}
    Two cumulative distribution functions $S$ and $T$ are said to be \emph{tail equivalent} if they have the same right endpoint $x_{0} \coloneqq \sup\{x: S(x) < 1\} = \sup\{x: T(x) < 1\}$ and satisfy the relationship
    \begin{align*}
        \lim_{x \to x_{0}}
        \frac{1-S(x)}{1-T(x)}
        =
        A \in (0, \infty)
        .
    \end{align*}
\end{definition}

With this concept in mind, let $H$ denote a random variable following the generalized gamma distribution with scale parameter $\sqrt{\lambda_{1}}$ and shape parameters $\mul$ and $2$, i.e., $H \sim \operatorname{GG}(\sqrt{\lambda_{1}},\mul,2)$.
In particular, $H$ has probability density function $f_{H}(x) = \frac{2}{\lambda_{1}^{\mul/2}\Gamma(\mul/2)}x^{\mul-1}\exp(-x^{2}/\lambda_{1})$ for $x \ge 0$.
For reference, this is equivalent to a $\sqrt{\lambda_{1}/2}$-scaled chi distribution with $\mul$ degrees of freedom, corresponding to the square root of a sum of $\mul$ squared i.i.d. Gaussian random variables with mean zero and variance $\lambda_{1}/2$.
Important special cases of $H$ are $\mul = 1$, which reduces to the half-normal distribution with scale parameter $\sqrt{\lambda_{1}/2}$, and $\mul = 2$, which reduces to the Rayleigh distribution with scale parameter $\sqrt{\lambda_{1}/2}$.

\cref{body-proposition:saddle-point-tail-equivalence} establishes the tail equivalence of $X$ and $H$, a key technical finding in this paper.
The proof of \cref{body-proposition:saddle-point-tail-equivalence} is provided in \suppref{supp-section:proof-of-saddle-point-approximation} and makes crucial use of saddle point density approximations \citep{daniels_saddlepoint_1954}, discussed in detail in \suppref{supp-section:deriving-saddle-point}.

\begin{proposition}
[Tail equivalence of $X$ and $H$]
\label{body-proposition:saddle-point-tail-equivalence}
    Under \cref{body-assumption:noise,body-assumption:snr,body-assumption:gap-lam1-lam2}, for all $n, m \geq r$, it holds that
    \begin{align*}
        \lim_{x\to\infty}
        \frac{1-F_{H}(x)}{1-F_{X}(x)}
        =
        \abias
        \coloneqq
        \begin{cases}
            1
            &
            \text{if }\;
            r = 1 
            \text{ or }\;
            \mul = r
            , \\
            \prod_{j=\mul+1}^{r}\left(1-\frac{\lambda_{j}}{\lambda_{1}}\right)^{1/2}
            &
            \text{if }\;
            r > 1
            \text{ and }\;
            \mul < r
            .
        \end{cases}
    \end{align*}
    Necessarily, $\abias \in (0,1]$.
\end{proposition}

In words, \cref{body-proposition:saddle-point-tail-equivalence} establishes that the right tail of the random variable $X$, which is equal in distribution to the Euclidean norm of the Gaussian random vector $\m{E}_{1,\all}\T \m{V} \m{S}\inv$, is governed by the eigenvalues $(\lambda_{j})_{j \in \dpar{r}}$, which intricately depend on the interplay between spectral properties of the signal and noise matrices.
Here, the term $\abias$ quantifies the discrepancy between the upper tail of $X$ and of the reference $H$.
This non-negative term is bounded away from zero due to \cref{body-assumption:gap-lam1-lam2}, though more tightly spaced values of $(\lambda_{j})_{j \in \dpar{r}}$ lead to smaller values of $\abias$, indicating a heavier tail of $X$ relative to the generalized gamma tail of $H$.

In order to establish \cref{body-theorem:Gumbel-convergence}, one of our main results, the remaining key ingredient is to apply a fundamental result in extreme value theory \citep[Proposition 1.19]{resnick_extreme_2007} which establishes a correspondence between normalizing sequences for tail equivalent random variables.
This enables obtaining normalizing sequences for $X$ using easier-to-compute normalizing sequences for $H$.
Finally, using \cref{body-lemma:first-order-approximation}, we obtain the limiting distribution of $\tinorm{\mh{U}\RU - \m{U}}$ in \cref{body-theorem:Gumbel-convergence} which is proved \suppref{supp-section:proof-of-Gumbel-convergence}.

\begin{theorem}[Extreme value asymptotics for singular subspace estimation]
\label{body-theorem:Gumbel-convergence}
    Let $\gumbelrv$ denote a standard Gumbel random variable. 
    Under \cref{body-assumption:noise,body-assumption:matrix-size,body-assumption:delocalization,body-assumption:snr,body-assumption:gap-lam1-lam2}, it holds that
    \begin{align*}
       \tstat
       \coloneqq
       a_{n}^{-1}
       \plr{\tinorm{\mh{U}\RU - \m{U}} - b_{n}}
       +
       \log(\abias)
       \rightsquigarrow
       \gumbelrv
    \end{align*}
    as $n \to \infty$, where $(a_{n})_{n \ge 1}$ and $(b_{n})_{n \ge 1}$ are normalizing sequences given by 
    \begin{align*}
        a_{n}
        \coloneqq
        \frac{\sqrt{\lambda_{1}}}{2\sqrt{\log n}},
        \quad
        b_{n}
        \coloneqq
        \sqrt{\lambda_{1}\log n}
        +
        \frac{\sqrt{\lambda_{1}} (\mul-2)\log\log n}{4\sqrt{\log n}}
        -
        \frac{\sqrt{\lambda_{1}}\log \Gamma(\mul/2)}{2\sqrt{\log n}}
    \end{align*}
    and where $\Gamma(\cdot)$ denotes the gamma function.
\end{theorem}

In \cref{body-theorem:Gumbel-convergence}, $\log(\abias) \in (-\infty, 0]$ is a necessary bias term that simultaneously reflects properties of the signal singular values and of the upper tail comparison between $X$ and $H$ discussed above.
Further, the normalizing sequences $(a_{n})_{n \ge 1}$ and $(b_{n})_{n \ge 1}$ themselves depend on the diverging signal singular values, which accounts for the scaling of the distribution of $H$ as $n \to \infty$.

\begin{remark}[Scaling property of the generalized gamma distribution]
\label{body-remark:scaled-chi-properties} 
    The distribution of $H$ implicitly depends on the number of rows $n$ via the parameter $\sqrt{\lambda_{1}}$ which is a function of the non-zero singular values of $\m{M}$ (see \cref{body-lemma:MGF}).
    For instance, if $\m{D} = \sigma^{2}\m{I}_{m}$, then $\lambda_{1} = 2\sigma^{2} / s_{r}^{2}$ and so $\sqrt{\lambda_{1}} = o(1)$ by \cref{body-assumption:snr}.
    Nevertheless, degeneracy is avoided here by the scaling property of the generalized gamma distribution.
    In particular, for any $\rho > 0$ it holds that $H_{\rho} \eqd \rho H_{1}$, where $H_{\rho} \sim \operatorname{GG}(\rho,\mul,2)$.
    Consequently, given scalars $\alpha_{1,n} > 0$ and $\beta_{1,n}$,
    \begin{align}
        \alpha_{1,n}\inv
        \plr{\max_{i \in \dpar{n}} H_{1,i} - \beta_{1,n}}
        \eqd 
        (\rho \alpha_{1,n})\inv
        \plr{\max_{i \in \dpar{n}} H_{\rho,i} - \rho\beta_{1,n}}
        ,
        \label{body-eq:generalized-gamma-sequences-scaling}
    \end{align}
    where $H_{\rho,1}, \dots, H_{\rho,n}$ denote i.i.d. copies of $H_{\rho}$.
    Thus, if $(\alpha_{1,n})_{n \ge 1}$ and $(\beta_{1,n})_{n \ge 1}$ are real-valued sequences such that the left-hand side of \cref{body-eq:generalized-gamma-sequences-scaling} converges in distribution to an extreme value distribution as $n \to \infty$, then $(\rho\alpha_{1,n})_{n \ge 1}$ and $(\rho\beta_{1,n})_{n \ge 1}$ are corresponding normalizing sequences for $H_{\rho}$ such that the right-hand side of \cref{body-eq:generalized-gamma-sequences-scaling} converges in distribution to the same extreme value distribution, irrespective of whether $\rho$ changes with $n$.
    Consequently, the normalizing sequences for the maximum of $H$ can be obtained by first using classical techniques in extreme value theory to compute the normalizing sequences for the maximum of $H_{1}$ and then rescaling by $\sqrt{\lambda_{1}}$. 
\end{remark}

Below, \cref{body-theorem:CDF-bounds} bounds the Kolmogorov--Smirnov distance between the cumulative distribution function of the statistic $\tstat$ and that of the standard Gumbel distribution, yielding a non-asymptotic distributional guarantee that strengthens \cref{body-theorem:Gumbel-convergence}.
The proof of \cref{body-theorem:CDF-bounds} is provided in \suppref{supp-pf:CDF-bounds}.

\begin{theorem}[Extreme value CDF convergence rate for singular subspace estimation]
\label{body-theorem:CDF-bounds}
    Let $F_{n}$ denote the cumulative distribution function of $\tstat$ in \cref{body-theorem:Gumbel-convergence}, and let $F_{\gumbelrv}(x) = \exp(-\exp(-x))$ denote the cumulative distribution function of a standard Gumbel random variable $\gumbelrv$.
    Under \cref{body-assumption:noise,body-assumption:matrix-size,body-assumption:delocalization,body-assumption:snr,body-assumption:gap-lam1-lam2}, there exists $n_{0} \in \mathbb{N}$ and a constant $C > 0$ such that if $n \geq n_{0}$, then
    \begin{align*}
        \sup_{x \in \R} |F_{n}(x) - F_{\gumbelrv}(x)|
        \leq
        C
        \clr{
            {\sqrt{\rlog}}
            \plr{
                \frac{\sigma\sqrt{n}}
                {s_{r}}
                +
                \sqrt{
                    \frac{\mu r}{n}
                }
            }
            +
            \frac{\sigma\sqrt{\mu r n \log n}}
            {s_{r}}
            +
            \frac{\mul \log\log n}{\log n}
        }.
    \end{align*}
\end{theorem}

The upper bound appearing in \cref{body-theorem:CDF-bounds} consists of three additive terms.
The first two terms arise from the first-order approximation error established in \cref{body-lemma:first-order-approximation} using perturbation analysis and do not involve $\mul$.
In contrast, the third term reflects the intrinsic extreme value convergence rate per the study of \cref{body-eq:Xi-quadratic-form} and does involve $\mul$.
Both terms vanish asymptotically as $n \to \infty$ under \cref{body-assumption:snr}, with the overall rate depending on the magnitude of $s_{r}$.
Concretely, letting $\sigma, \mul, r, \mu \simeq 1$ for the sake of discussion, if $n^{1/2} \log n \ll s_{r} \lesssim n^{1/2} \log^{2} n / (\log\log n)$, which can be interpreted as corresponding to a weak signal strength regime, then the first two terms due to the perturbation analysis dominates the overall rate.
Conversely, if $s_{r} \gg n^{1/2} \log^{2} n / (\log\log n)$, which can be interpreted as corresponding to a moderate or strong signal strength regime, then the third term due to the underlying extreme value convergence dominates the overall rate.

\subsection{Plug-in extreme value test statistic using de-biased sample singular values}
\label{body-section:debiased-singular-values}

In what follows, we pursue hypothesis testing and inference for $\m{U}$, and doing so leads us to treat the unknown population signal singular values as nuisance parameters.
This section obtains shrinkage-type de-biased estimators $(\widetilde{s}_{j})_{j \in \dpar{r}}$ for the signal singular values using the leading sample singular values of $\mh{M}$, thus yielding a data-driven version of \cref{body-theorem:Gumbel-convergence} that begets statistical hypothesis testing.
To enable explicit calculations, and as is often done in the existing literature, we consider some additional simplifying model assumptions as described below.

\begin{assumption}[Model simplification for de-biased plug-in estimation]
    \label{body-assumption:plug-in-asms}
    The following items hold.
    \begin{enumerate}[label=(\roman*)]
        \item (Entry-wise Gaussian noise)
        The entries of the noise matrix $\m{E}$ are i.i.d. $\NN(0,\sigma^{2})$ with $0 < \sigma^{2} < \infty$.
        \label{body-item:E-iid-entries}
        
        \item (Limiting matrix aspect ratio)
        The matrix dimensions $n$ and $m \equiv m(n)$ satisfy $\min\{n / m, m / n\} \to c \in (0, 1]$ as $n \to \infty$.
        \label{body-item:matrix-aspect-ratio}
        
        \item (Signal singular value separation)
        If $r \geq 2$, there exist constants $n_{0}, c_{0} > 1$ such that $\min_{j \in \dpar{r-1}}\{s_{j}/s_{j+1}\} \geq c_{0}$ for all $n \geq n_{0}$. 
        \label{body-item:sval-separation}
    \end{enumerate}
\end{assumption}

\cref{body-item:E-iid-entries} posits that the entries of $\m{E}$ are i.i.d. $\NN(0, \sigma^{2})$ random variables to enable computing the almost sure limits of $(\widehat{s}_{j})_{j \in \dpar{r}}$ using the techniques in \citet{benaych-georges_singular_2012}. 
\cref{body-item:matrix-aspect-ratio} posits that the matrices $\mh{M}$, $\m{M}$, and $\m{E}$ have a specified limiting aspect ratio, to facilitate direct calculations.
\cref{body-item:sval-separation} is a singular value gap condition between the non-zero signal singular values.
This separation ensures that the fluctuations of $(\widehat{s}_{j})_{j\in\dpar{r}}$ around their deterministic locations are uniformly of order $n^{-1/2}$, which is used to invoke \citet[Theorem 1]{fan_asymptotic_2022} in the proof of \cref{body-proposition:debiased-singular-values}.
Further, the multiplicity of $s_{r}$ is fixed with $\mul = 1$ as a consequence of \cref{body-item:sval-separation}.

\cref{body-proposition:debiased-singular-values} establishes de-biased singular value estimators and their rates of convergence for our model setting and is proved in \suppref{supp-section:debiased-singular-values}.

\begin{proposition}[De-biased singular value estimators]
    \label{body-proposition:debiased-singular-values}
    Invoke \cref{body-assumption:snr,body-assumption:plug-in-asms}, and let $N \coloneqq \max\{n,m\}$.
    For each $j \in \dpar{r}$, the estimator
    \begin{equation*}
        \widetilde{s}_{j}
        \coloneqq
        \frac{1}{\sqrt{2}}
        \left(
        \widehat{s}_{j}^{2} - (1 + c)\sigma^{2} N
        +
        \sqrt{[
        \widehat{s}_{j}^{2}
        -
        (1 + c)\sigma^{2} N]^{2} - 4 c\sigma^{4} N^{2}}
        \right)^{1/2}
    \end{equation*}
    satisfies
    \begin{equation*}
        \frac{\widetilde{s}_{j}}{s_{j}}
        =
        1
        + 
        O
        \left(
            \frac{\sigma^{4} n^{2}}
                {s_{j}^{4}}
                +
                \frac{\sigma\log^{1/2} n}
                {s_{j}}
         \right)
    \end{equation*}
    with probability at least $1 - O(n^{-6})$.
\end{proposition}

\cref{body-proposition:debiased-singular-values} is the key ingredient for establishing \cref{body-theorem:Gumbel-convergence-plugin} which is a plug-in version of \cref{body-theorem:Gumbel-convergence} that accounts for the signal singular values.
\cref{body-theorem:Gumbel-convergence-plugin} is proved in \suppref{supp-pf:Gumbel-convergence-plugin}.

\begin{theorem}[Data-driven plug-in version of \cref{body-theorem:Gumbel-convergence}]
    \label{body-theorem:Gumbel-convergence-plugin}
    Let $(\widetilde{\lambda}_{j})_{j \in \dpar{r}}$ denote the eigenvalues of $\sigma^{2} \widetilde{\m{S}}^{-2}$ where $\widetilde{\m{S}} = \diag(\widetilde{s}_{1}, \dots, \widetilde{s}_{r})$ per \cref{body-proposition:debiased-singular-values}. 
    Define
    \begin{align*}
       \tplug
       \coloneqq
       \widetilde{a}_{n}^{-1}
       \plr{
       \tinorm{\mh{U}\RU - \m{U}} -\widetilde{b}_{n}}
       +
       \log(\abplug)
    \end{align*}
    where
    $
    \widetilde{a}_{n}
    \coloneqq
    \dfrac{{\widetilde{\lambda}_{1}}^{1/2}}{2\sqrt{\log n}}
    $
    and
    $
    \widetilde{b}_{n}
    \coloneqq
    \widetilde{\lambda}_{1}^{1/2}\sqrt{\log n}
    -
    \dfrac{{\widetilde{\lambda}_{1}}^{1/2}\log\log n}{4\sqrt{\log n}}
    -
    \dfrac{{\widetilde{\lambda}_{1}}^{1/2}\log \Gamma(1/2)}{2\sqrt{\log n}}
    $
    and
    \begin{align*}
        \abplug
        \coloneqq
        \begin{cases}
            1
            &
            \text{if } r = 1, 
            \\
            \prod_{j=2}^{r} \left(1 -\frac{\widetilde{\lambda}_{j}}{\widetilde{\lambda}_{1}}\right)^{1/2}
            &
            \text{if } r > 1.
        \end{cases}
    \end{align*}
    Under \cref{body-assumption:delocalization,body-assumption:snr,body-assumption:gap-lam1-lam2,body-assumption:plug-in-asms}, it holds that $\tplug \rightsquigarrow \gumbelrv$ as $n \to \infty$.
\end{theorem}

\subsection{Hypothesis testing for singular subspaces}
\label{body-section:hypothesis-testing-framework}

The asymptotic distributional theory for singular subspace estimation in \cref{body-section:gumbel-convergence,body-section:debiased-singular-values} enables testing for singular subspace structure. Conceptually, we are interested in testing between null and alternative hypotheses of the form 
\begin{align}
    \label{body-eq:T1}
    \hypn:
    \m{U} = \m{U}_{0}
    \quad
    \text{versus}
    \quad
    \hypa:
    \m{U} = \m{U}_{1}
    .
\end{align}
More precisely, we consider sequences of hypotheses indexed by $n$, namely $\hypn^{(n)}: \m{U} = \m{U}_{0,n}$ versus $\hypa^{(n)}: \m{U} = \m{U}_{1,n}$ for $(\m{U}_{0,n})_{n \ge 1}$ and $(\m{U}_{1,n})_{n \ge 1}$, where in practice $n$ is large but fixed and for simplicity we write $\mh{U}$ to denote $\mh{U}_{n}$.
To accommodate \cref{body-assumption:delocalization}, the parameter space of interest is taken to be
\begin{equation}
    \label{body-eq:hypothesis-parameter-space}
    \Theta
    \equiv
    \Theta(r, n)
    \coloneqq
    \clr{
        \m{U} \in \R^{n \times r}:
        \m{U}\T\m{U} = \m{I}_{r}, 
        \;
        \frac{r}{n}
        \leq
        \tinorm{\m{U}}^{2}
        \leq
        \frac{1}{\rlog \log\log n}
    }.
\end{equation}
\cref{body-theorem:Gumbel-convergence,body-theorem:Gumbel-convergence-plugin} imply asymptotic Type I error control for the hypothesis testing problem posed in \cref{body-eq:T1}.

\begin{corollary}[Asymptotic Type I error probability]
\label{body-corollary:type-I-error}
    Consider testing the (sequences of) hypotheses in \cref{body-eq:T1} with significance level $\alpha \in (0,1)$.
    Under \cref{body-assumption:noise,body-assumption:matrix-size,body-assumption:delocalization,body-assumption:snr,body-assumption:gap-lam1-lam2}, the test (sequence) ${\phi}_{\alpha,n}(\mh{U}) = \Is{\tstat\geq F_{\gumbelrv}\inv({1-\alpha})}$ has asymptotic size $\alpha$, i.e.,
    $
        \E_{\hypn}[
            {\phi}_{\alpha,n}(\mh{U})
        ]
        \to
        \alpha
    $
    as $n \to \infty$.
    Similarly, under \cref{body-assumption:delocalization,body-assumption:snr,body-assumption:gap-lam1-lam2,body-assumption:plug-in-asms}, the test (sequence) $\widetilde{\phi}_{\alpha,n}(\mh{U}) = \Is{\tplug\geq F_G\inv({1-\alpha})}$ has asymptotic size $\alpha$, i.e.,
    $
        \E_{\hypn}[
            \widetilde{\phi}_{\alpha,n}(\mh{U})
        ]
        \to
        \alpha
    $
    as $n \to \infty$.
\end{corollary}

Next, given two matrices $\m{A}$ and $\m{B}$ of the same dimensions and each full column rank, let $\Rs{\m{A}}{\m{B}} \coloneqq \sgn(\m{A}\T\m{B})$ where $\sgn(\cdot)$ is as in \cref{body-section:introduction}.
In particular, $\sgn(\m{A}\T\m{B}) \equiv \m{O}_{1}\m{O}_{2}\T$ where $\m{O}_{1}\m{C}\m{O}_{2}\T$ denotes the singular value decomposition of $\m{A}\T\m{B}$.
In order to characterize the power of our test procedure in the setting of \cref{body-eq:T1}, the (sequences of) matrices $\m{U}_{0}$ and $\m{U}_{1}$ are without loss of generality assumed to be aligned in the following manner.

\begin{assumption}[Aligned hypotheses]
    \label{body-assumption:aligned-alternative}
    The hypothesized matrices $\m{U}_{0}, \m{U}_{1} \in \Theta$ satisfy $\m{R}({\m{U}_{0},\m{U}_{1}}) = \m{I}_{r}$.
\end{assumption}

\cref{body-assumption:aligned-alternative} can always be satisfied by first starting with generic orthonormal matrices $\m{U}_{0}, \m{U}^{\prime}_{1} \in \Theta$ and then specifying $\m{U}_{1} = \m{U}^{\prime}_{1}\m{R}(\m{U}^{\prime}_{1},\m{U}_{0})$, hence $\m{R}({\m{U}_{0},\m{U}_{1}}) = \m{I}_{r}$.
Importantly, the columns of $\m{U}_{1}$ span the same subspace as the columns of $\m{U}^{\prime}_{1}$, and the pairwise Euclidean distances among the rows of $\m{U}_{1}$ are identical to those of $\m{U}^{\prime}_{1}$.

\cref{body-theorem:power-analysis} quantifies the statistical power of our test based on the magnitude of the discrepancy $d_{n} \coloneqq \tinorm{\m{U}_{0} - \m{U}_{1}}$, where for reference $d_{n} \leq 2\sqrt{\mu r/n}$ always holds under \cref{body-assumption:delocalization}.
The proof of \cref{body-theorem:power-analysis} is provided in \suppref{supp-pf:power-analysis}.

\begin{theorem}[Asymptotic power of the testing procedure]
    \label{body-theorem:power-analysis}
    Let $\alpha \in (0,1)$.
    Under \cref{body-assumption:delocalization,body-assumption:snr,body-assumption:gap-lam1-lam2,body-assumption:plug-in-asms,body-assumption:aligned-alternative}, consider the test (sequence) $\widetilde{\phi}_{\alpha,n}(\mh{U}) = \Is{\tplug \geq F_{\gumbelrv}\inv(1-\alpha)}$ for the hypotheses in \cref{body-eq:T1}.
    The following properties hold.
    \begin{enumerate}[label=(\roman*)]
        \item (Consistent regime)
        If ${d}_{n} \gg \widetilde{b}_{n}(\mu r)^{1/2}$, then $\E_{\hypa}[\widetilde{\phi}_{\alpha,n}(\mh{U})] \to 1$ as $n \to \infty$.
        \item (Inconsistent regime)
        If $d_{n} \ll \widetilde{a}_{n}(\mu r)^{-1/2}$, then $\E_{\hypa}[\widetilde{\phi}_{\alpha,n}(\mh{U})] \to \alpha$ as $n \to \infty$.
    \end{enumerate}
\end{theorem}

\subsection{Relaxing the Gaussian noise assumption}
\label{body-section:non-Gaussian-noise}

The focus of this paper is on the Gaussian matrix denoising model.
Nevertheless, here we briefly illustrate that our approach and extreme value results also hold in certain settings without requiring Gaussian noise.
The proof of \cref{body-proposition:non-gaussian-noise-rank-1} is provided in \suppref{supp-section:non-gaussian-noise-rank-1}.

\begin{proposition}[Gumbel convergence under general noise, rank-one signal matrix with equiangular singular vectors]
    \label{body-proposition:non-gaussian-noise-rank-1}
    Consider the matrix denoising model with $n = m$, $r = 1$ and $\m{u} = \m{v} = \1_{n} /\sqrt{n}$.
    Further, suppose that the entries of the noise matrix $\m{E}$ are i.i.d. with mean zero, unit variance, and finite moment generating function on an interval containing zero.
    If $s_{r} \gg \sqrt{n \log^{3} n}$, then the conclusion in \cref{body-theorem:Gumbel-convergence} holds.
\end{proposition}

Generalizations of \cref{body-proposition:non-gaussian-noise-rank-1} are beyond the scope of this paper and thus left for future work.

\begin{remark}[Proof overview of \cref{body-proposition:non-gaussian-noise-rank-1}]
    Proving \cref{body-proposition:non-gaussian-noise-rank-1} involves two key steps, resembling the proof of \cref{body-theorem:Gumbel-convergence} but requiring numerous modifications throughout.
    
    The first key step is to establish a high-probability uniform row-wise approximation of the form
    $
    \mh{u}\sgn(\mh{u}\T\m{u}) - \m{u}
    \approx
    \m{E}\m{v}s_{r}^{-1}
    $.
    This is possible by replacing the Gaussian matrix concentration inequalities in \citet[Lemma 19]{yan_entrywise_2024} that were central in proving \cref{body-lemma:first-order-approximation} with concentration results for more general noise distributions.
    In the rank-one setting, the $\ell_{2,\infty}$ norm reduces to the $\ell_{\infty}$ norm.

    The second key step is to establish that $\tinorm{\m{E}\m{v}s_{r}^{-1}}$ converges in distribution to a standard Gumbel distribution after proper centering and scaling.
    Letting $X_{i,n}$ denote the $i$-th row of $\m{E}\m{v}s_{r}^{-1}$ which satisfies $X_{i,n} = \frac{1}{s_{r}\sqrt{n}}\sum_{j=1}^{n}E_{i,j}$, we consider the reference quantity $Y_{i,n} \coloneqq \frac{1}{s_{r}\sqrt{n}}\sum_{j=1}^{n}Z_{i,j}$, where $\m{Z}\coloneqq (Z_{i,j})_{i,j\in\dpar{n}}$ has i.i.d. $\NN(0,1)$ entries. 
    Notably, by the proof of \cref{body-theorem:Gumbel-convergence} in \suppref{supp-section:proof-of-Gumbel-convergence}, the maximum of $|Y_{i,n}|$ over $i\in\dpar{n}$ is in the Gumbel domain of attraction, more specifically, $[F_{|Y_{i,n}|}(a_{n}x + b_{n})]^{n} \to \exp(-\exp(-x))$ for all $x\in\R$, with $(a_{n})_{n \geq 1}$ and $(b_{n})_{n \geq 1}$ as defined in \cref{body-theorem:Gumbel-convergence}.
    From here, we first show that $\lim_{n\to\infty} [1-F_{|X_{i,n}|}(a_{n}x + b_{n})]/[1-F_{|Y_{i,n}|}(a_{n}x + b_{n})] \to 1$ for all $x\in\R$ using the Bahadur--Rao theorem \citep{bahadur_deviations_1960} in large deviation theory.
    The remainder of the proof uses this tail ratio convergence along with the benchmark Gumbel convergence in the Gaussian noise setting to show that $[F_{|X_{i,n}|}(a_{n}x + b_{n})]^{n} \to \exp(-\exp(-x))$ for all $x\in\R$.
\end{remark}

\section{Numerical examples}
\label{body-section:Numerical-simulations}

This section provides a suite of numerical examples and illustrations using simulations.
For ease of presentation, we sometimes refer to parametrized sequences of matrices (corresponding to sequences of null and alternative hypotheses) simply by their representatives.
For example, we may identify $\{\1_{n}\}_{n \ge 1}$ simply by $\1_{n}$, where large values of $n$ are of primary interest.

\subsection{Test statistic distribution under the null hypothesis}
\label{body-section:simulation-setup-convergence}

This section empirically investigates the finite-sample properties of \cref{body-theorem:Gumbel-convergence,body-theorem:Gumbel-convergence-plugin}. 
Here, we consider $\m{M}$ having $n = 3000$ rows and $m = 1.2 n = 3600$ columns with rank $r = 5$. 
The left and right matrices of signal singular vectors, $\m{U}$ and $\m{V}$, are given by orthonormalizing the columns of two independent random matrices $\m{N}_{1} \in \R^{n \times r}$ and $\m{N}_{2} \in \R^{m \times r}$ each having i.i.d. standard normal entries, namely $\m{U} = \operatorname{QR}(\m{N}_{1})$ and $\m{V} = \operatorname{QR}(\m{N}_{2})$, where $\operatorname{QR}(\cdot)$ denotes the orthonormal factor obtained from an economy-size QR decomposition.
This configuration results in $\m{U}$ and $\m{V}$ being sampled from a unitarily invariant distribution on the the set of $r$-frames in $\R^{n}$ and $\R^{m}$, respectively, so that $\mu$ is treated as a constant.
The noise matrix is taken to have i.i.d. entries of the form $E_{i,j} \sim \NN(0,\sigma^{2})$ with $\sigma = 1$. 

\cref{body-fig:convergence-in-distribution} displays three pairs of histograms and corresponding quantile-quantile plots for three statistics involving the two-to-infinity norm, each compared to the theoretical density function and quantiles of the standard Gumbel distribution.
The first (oracle) statistic corresponds to \cref{body-theorem:Gumbel-convergence} and uses knowledge of the signal singular values.
The second (de-biased) statistic corresponds to \cref{body-theorem:Gumbel-convergence-plugin} and de-biases the sample singular values without knowledge of the signal singular values.
For comparison, the third (uncorrected) statistic corresponds to na\"{i}vely estimating the population singular values with their sample counterparts.
Here, the signal singular values are taken to be equally spaced between $s_{r} = 0.25\sigma \sqrt{rn}\log^{1.01} n$ and $s_{1} = \sigma \sqrt{rn}\log^{1.01} n$, and the noise matrix $\m{E}$ is independently simulated $\nmc = 1800$ times to create a random sample of matrices $\mh{M}$, keeping $\m{M}$ fixed.
In this example, the signal singular values only slightly exceed the minimum signal strength imposed by the condition $s_{r} \gg \sqrt{\mu \rlog}(\sigma\sqrt{n})$ in \cref{body-assumption:snr}.
\cref{body-fig:convergence-in-distribution} illustrates the fidelity of asymptotic approximations in finite samples in this near-boundary setting and demonstrates the importance of correctly de-biasing the sample singular values.

\cref{body-tab:size-simulations} reports tail probability and interval probability estimates under $\hypn$ in \cref{body-eq:T1}.
Here, the signal singular values are equally spaced between $s_{r}$ and $s_{1} = 3 s_{r}$ for different values of $s_{r}$.
The noise matrix $\m{E}$ is independently simulated $\nmc = 5000$ times to compute the Monte Carlo probability estimates, keeping $\m{M}$ fixed.
This procedure is then repeated for various choices of $n$ and $s_{r}$ to observe their effect on the convergence of $\tplug$ to $G$.
Notably, if $s_{r}$ is sufficiently large, then the probability estimates no longer appreciably improve when further increasing $s_{r}$, all else equal, corroborating \cref{body-theorem:CDF-bounds}.

\begin{figure}[t]
    \centering
    \begin{minipage}{.33\textwidth}
        \includegraphics[width=\linewidth]{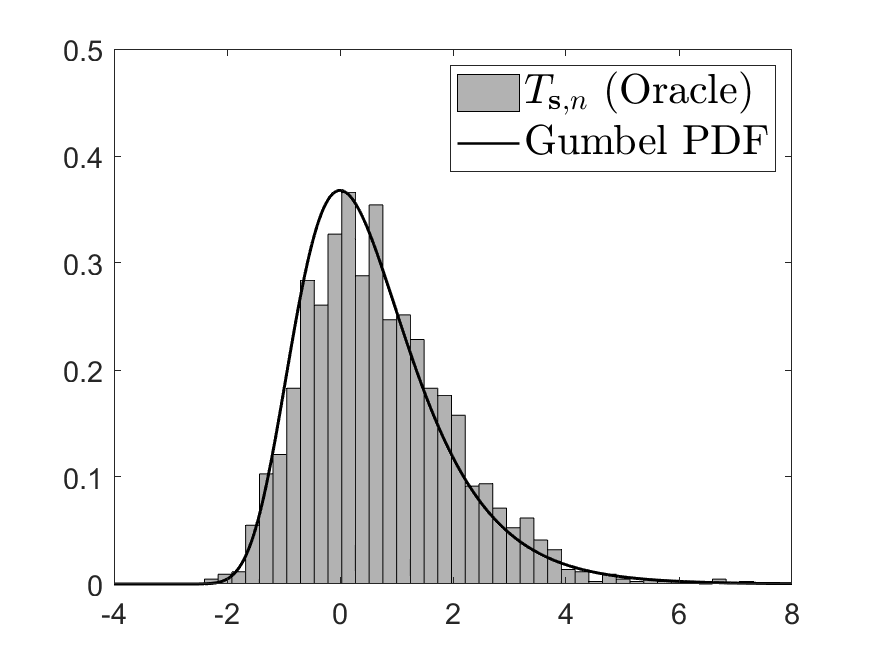}
    \end{minipage}%
    \begin{minipage}{.33\textwidth}
        \includegraphics[width=\linewidth]{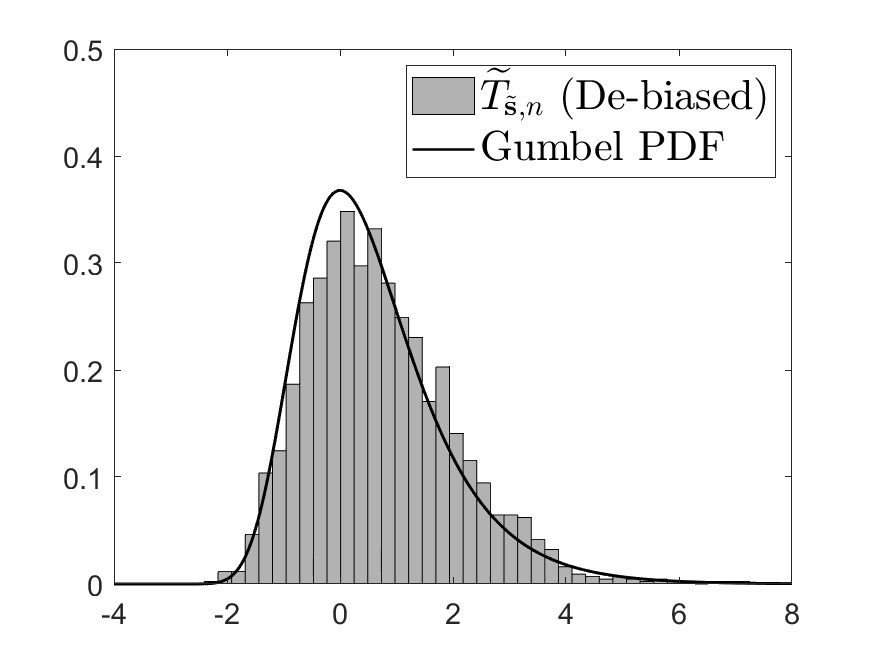}
    \end{minipage}%
    \begin{minipage}{.33\textwidth}
        \includegraphics[width=\linewidth]{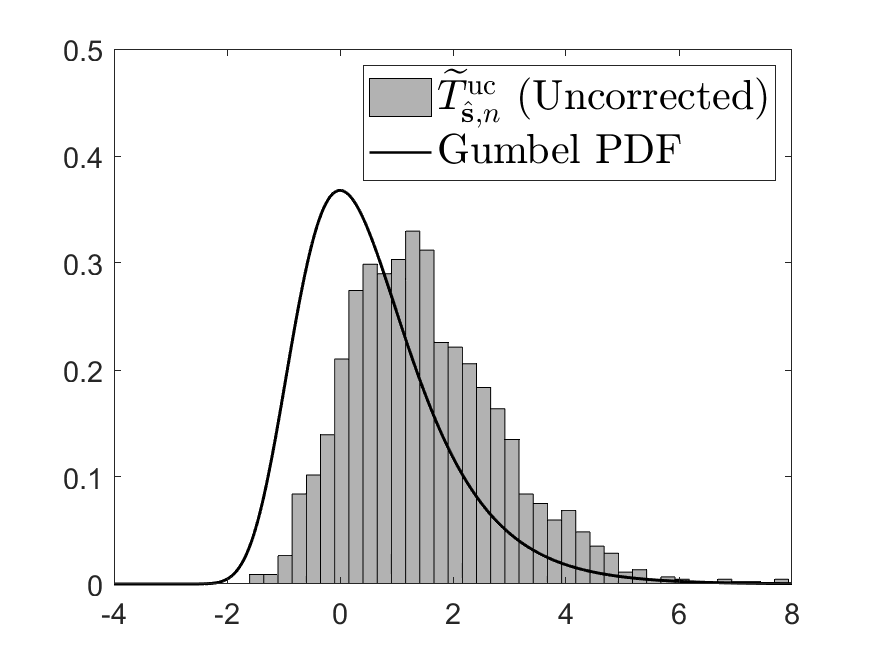}
    \end{minipage}
    \begin{minipage}{.33\textwidth}
        \includegraphics[width=\linewidth]{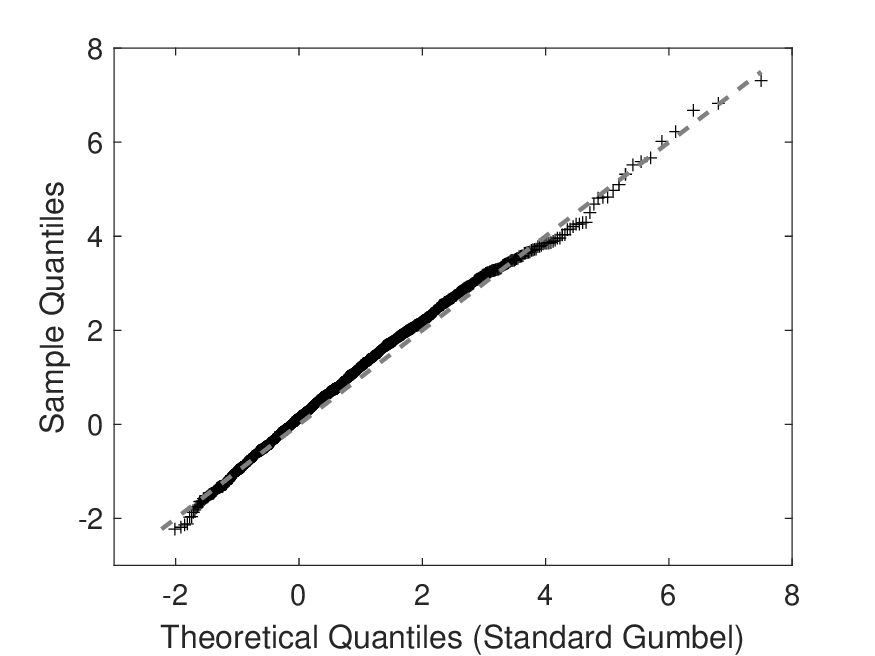}
    \end{minipage}%
    \begin{minipage}{.33\textwidth}
        \includegraphics[width=\linewidth]{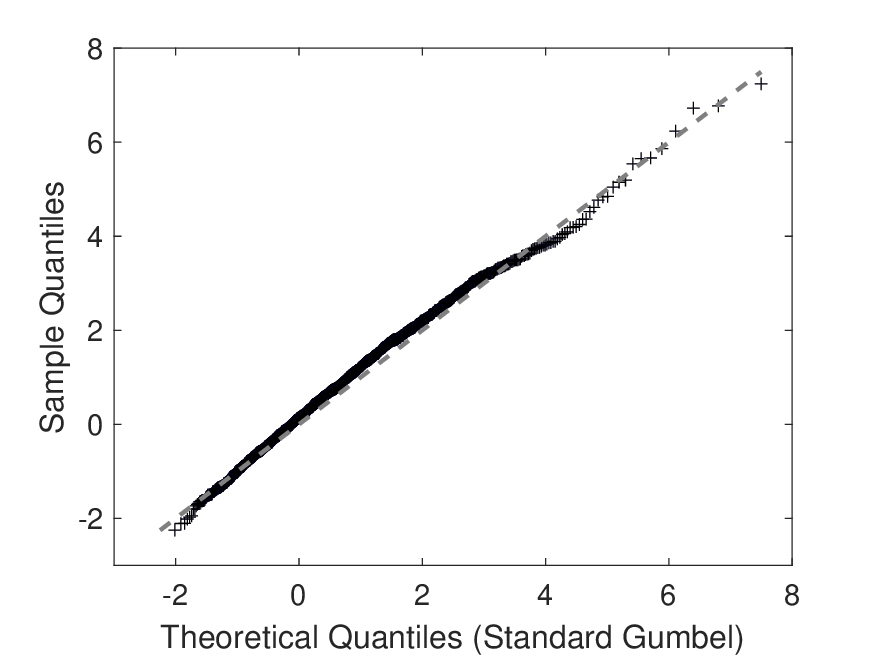}
    \end{minipage}%
    \begin{minipage}{.33\textwidth}
        \includegraphics[width=\linewidth]{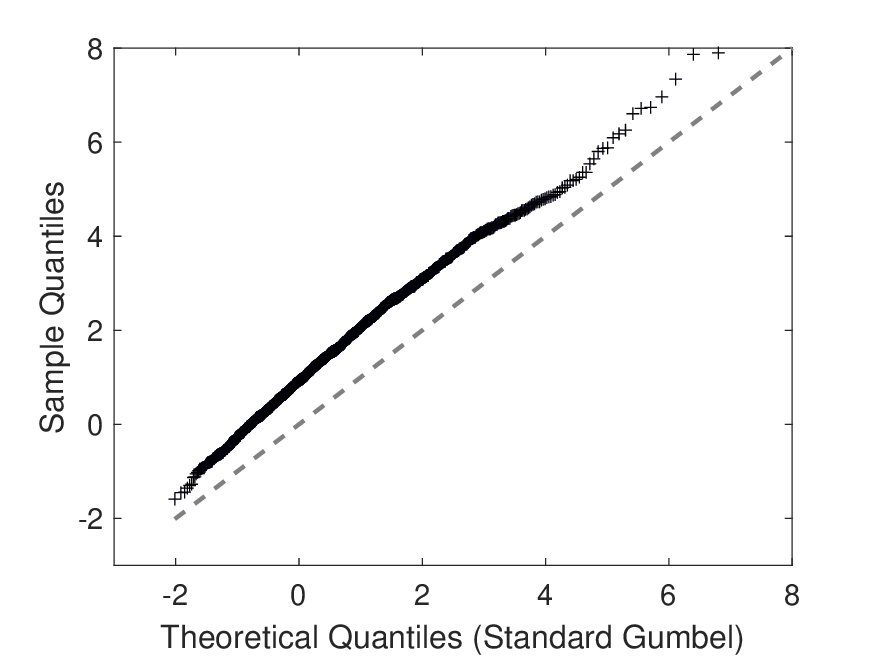}
    \end{minipage}
    \caption{Empirical test statistic distributions for simulated data under $\hypn$ in \cref{body-eq:T1}.
    Left column: test statistic $\tstat$ using the signal singular values.
    Middle column: test statistic $\tplug$ using the de-biased sample singular values.
    Right column: test statistic $\tpluguc$ using the uncorrected sample singular values.
    See \cref{body-section:simulation-setup-convergence}.} 
    \label{body-fig:convergence-in-distribution}
\end{figure}

\begin{table}[t]
    \centering
    \begin{tabular}{p{1cm}|p{2cm}p{2cm}p{2cm}p{2cm}p{2cm}}
    \hline
        \multicolumn{6}{c}{Estimate of upper tail probability $\P\{\tplug\geq F_G\inv(0.9)\}$}\\
        \hline
        \diagbox[width=1.4cm, height=0.7cm]{$n$}{$s_r$}&\centering $\sigma n^{0.5}$&\centering $\sigma n^{0.625}$&\centering $\sigma n^{0.750}$&\centering $\sigma n^{0.875}$ & \centering \arraybackslash $\sigma n$\\
       \hline
       \centering 100 &\centering 0.332 &\centering 0.084 &\centering 0.069 &\centering 0.066 & \centering \arraybackslash  0.066 \\
       \centering 800 &\centering 0.712 &\centering 0.131 &\centering 0.100 &\centering 0.081 & \centering \arraybackslash  0.089\\
       \centering 1600 &\centering  0.863 &\centering 0.128&\centering 0.091&\centering  0.100& \centering \arraybackslash 0.093\\
       \centering $\infty$ &\centering  NA &\centering 0.1 &\centering 0.1&\centering  0.1& \centering \arraybackslash 0.1\\
       \hline
       \multicolumn{6}{c}{Estimate of lower tail probability $\P\{\tplug\leq F_G\inv(0.1)\}$}\\
       \hline
       \centering 100 &\centering 0.056 &\centering 0.144 &\centering 0.147 &\centering 0.139 & \centering \arraybackslash  0.147 \\
       \centering 800 &\centering 0.001 &\centering 0.068 &\centering 0.110 &\centering 0.117 & \centering \arraybackslash  0.121\\
       \centering 1600&\centering 0.000 &\centering 0.070&\centering  0.108 &\centering  0.104& \centering \arraybackslash 0.111\\
       \centering $\infty$ &\centering  NA &\centering 0.1 &\centering 0.1&\centering  0.1& \centering \arraybackslash 0.1\\
       \hline
       \multicolumn{6}{c}{Estimate of interval probability $\P\{\tplug\in [F_G\inv(0.2),F_G\inv(0.8)]\}$}\\
       \hline
       \centering 100 &\centering 0.414 &\centering 0.579 &\centering 0.581 &\centering 0.592 & \centering \arraybackslash  0.580 \\
       \centering 800 &\centering 0.144 &\centering 0.591 &\centering 0.593 &\centering 0.602 & \centering \arraybackslash  0.588\\
       \centering 1600&\centering 0.051 &\centering 0.602&\centering  0.605 &\centering  0.597& \centering \arraybackslash 0.602\\
       \centering $\infty$ &\centering  NA &\centering 0.6 &\centering 0.6&\centering  0.6& \centering \arraybackslash 0.6\\
       \hline
    \end{tabular}
    \caption{Estimates for upper tail, lower tail, and interval probabilities under $\hypn$ in \cref{body-eq:T1} when $\sigma = 1$. 
    See \cref{body-section:simulation-setup-convergence}.}
    \label{body-tab:size-simulations}
\end{table}

\subsection{Statistical power}

This section empirically investigates the finite-sample power of our proposed testing methodology, thereby complementing the theoretical findings in \cref{body-theorem:power-analysis}.

\subsubsection{Row discrepant alternatives and power comparison}
\label{body-section:power-against-local-alternatives}

Here, we investigate the performance of our testing methodology as a function of the closeness between $\m{U}_{0}$ and $\m{U}_{1}$ in \cref{body-eq:T1}.
Towards this end, for aligned hypotheses as in \cref{body-assumption:aligned-alternative}, we introduce and consider the concept of \emph{row discrepancy} which simply counts the number of rows of $\m{U}_{1}$ that differ from $\m{U}_{0}$. 

\begin{definition}[Row discrepancy of an alternative matrix]
\label{body-defn:locality}
    Under \cref{body-eq:T1} and \cref{body-assumption:aligned-alternative}, the row discrepancy of the alternative matrix $\m{U}_{1}$ with respect to the null matrix $\m{U}_{0}$ is defined as
    \begin{align*}
        \loc{\m{U}_{1};\m{U}_{0}}
        \coloneqq
        \sum_{i=1}^{n}
        \I{(\m{U}_{0})_{i,\all}
        \neq
        (\m{U}_{1})_{i,\all}}
        .
    \end{align*}
    Further, the alternative matrix $\m{U}_{1}$ is said to be $\xi$-row discrepant from $\m{U}_{0}$ if $\loc{\m{U}_{1};\m{U}_{0}} = \xi$. 
    Necessarily, $\xi \in \{0, 1,\dots, n\}$.
\end{definition}

Concretely, consider the rank-one setting $\m{M} = s_{r}\m{uv}\T$ with $n = m$, where for clarity $r=1$.
We consider testing the hypotheses
\begin{align*}
    \hypn:
    \m{u}=\m{u}_{0}
    \quad
    \vs
    \quad
    \hypa:
    \m{u}=\m{u}_{1}
    ,
\end{align*}
where 
\begin{align}
\label{body-eq:T2}
    \m{u}_{0}
    \coloneqq
    \frac{1}{\sqrt{n}}\1_{n}
    \quad
    \text{ and }
    \quad
    \m{u}_{1}
    \coloneqq
    \frac{1}{\sqrt{n}}
    \begin{bmatrix}
        \1_{(n-\xi)} \\
        -\1_{\xi}
    \end{bmatrix}
    .
\end{align}
Here, $\m{u}_{1}$ is implicitly parametrized by $\xi$ and is $\xi$-row discrepant from $\m{u}_{0}$ for $\xi \in \{0, 1, \dots, \lfloor (n-1)/2 \rfloor\}$.
Further, here
\begin{equation*}
    d_{n}
    \coloneqq
    \tinorm{\m{u}_{0} - \m{u}_{1}}
    =
    2/\sqrt{n}
\end{equation*}
irrespective of $\xi$.
By comparing the value of $d_{n}$ to the threshold for consistent testing stated in \cref{body-theorem:power-analysis}, i.e., $\widetilde{b}_{n}(\mu r)^{1/2} \simeq {\sigma (\mu r \log n)^{1/2}}/s_{r}$, we have $d_{n} \gg \widetilde{b}_{n}(\mu r)^{1/2}$ provided $s_{r} \gg \sqrt{\mu r \log n}(\sigma \sqrt{n})$, which is satisfied by \cref{body-assumption:snr}.
Hence, by \cref{body-theorem:power-analysis}, the vector $\m{u}_{1}$ is always in the consistent testing regime for our proposed procedure under the present model assumptions.

A major advantage of using a test statistic based on the two-to-infinity norm lies in its sensitivity to fine-grained structure and correspondingly its ability to distinguish between null and alternative matrices that have small row discrepancy.
For comparison, we also consider existing distributional theory for the Frobenius norm projection distance developed in \citet{xia_normal_2021} which yields an alternative approach to testing the hypotheses.
In particular, under i.i.d. standard Gaussian noise, \citet[Theorem 4]{xia_normal_2021} applies and establishes that
\begin{align*}
     \tfrob
     \coloneqq
     \frac{\fnorm{\mh{U}\mh{U}\T-\m{U}\m{U}\T}^{2} + \fnorm{\mh{V}\mh{V}\T - \m{V}\m{V}\T}^{2} - 2(n+m-2r)\fnorm{\m{S}\inv}^{2}}
     {\sqrt{8(n+m-2r)}\fnorm{\m{S}^{-2}}}
     \rightsquigarrow
     Z
     \sim
     \NN(0,1)
\end{align*}
as $n \to \infty$ provided $s_{r} \gg \sigma\sqrt{rn}$ and $r \ll n$, both of which hold in our setting.

\begin{figure}[t]
    \centering
    \begin{minipage}{.5\textwidth}
        \includegraphics[width=\linewidth]{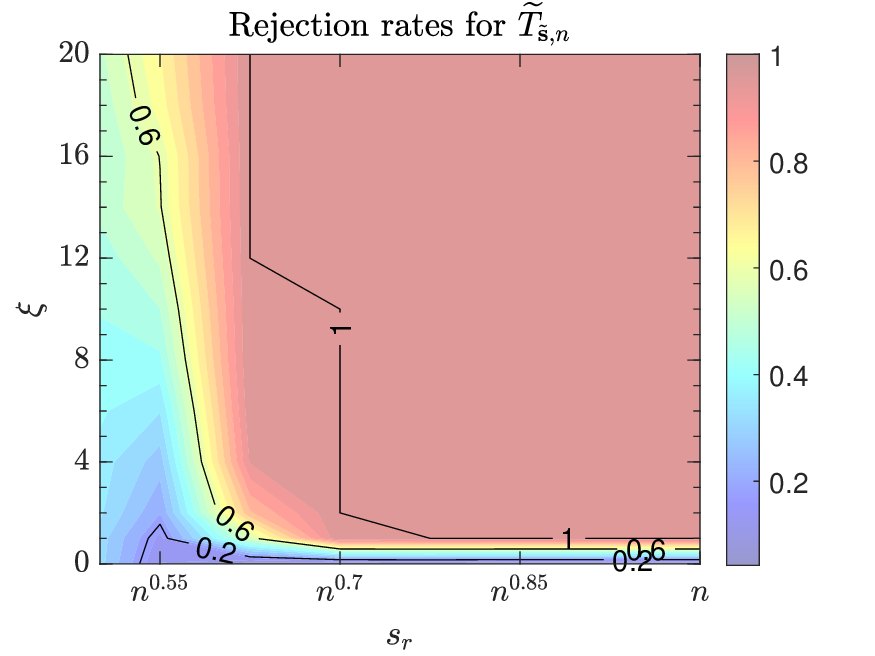}
    \end{minipage}%
    \begin{minipage}{.5\textwidth}
        \includegraphics[width=\linewidth]{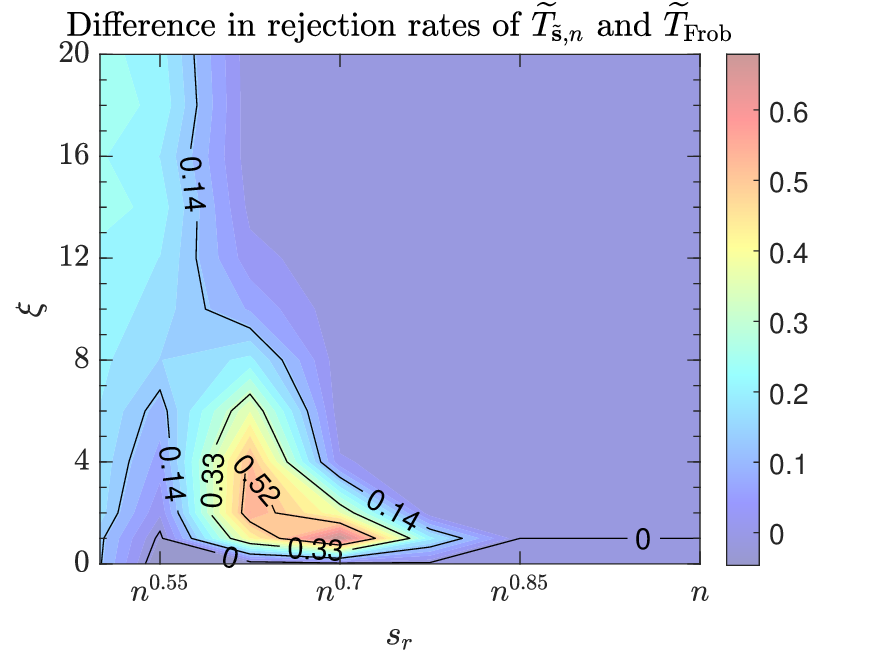}
    \end{minipage}
    \caption{Contour plots depicting empirical rejection rates for the tests with rejection rules $\tplug \geq F\inv_{G}(0.95)$ and $\tfplug \geq F_{Z}\inv(0.95)$, respectively, for various combinations of the signal strength $s_{r}$ and the row discrepancy $\xi$.
    See \cref{body-section:power-against-local-alternatives}.}
    \label{body-fig:power-local}
\end{figure}

\begin{table}[t]
    \centering
    \begin{tabular}{p{1cm}|p{2cm}p{2cm}|p{2cm}p{2cm}}
    \hline
        \multicolumn{5}{c}{Rejection rate under $\hypa$ in \cref{body-eq:T2}; $\loc{\m{u}_{1};\m{u}_{0}}=1$}\\
        \hline
        &\multicolumn{2}{c|}{$\tplug$}&\multicolumn{2}{c}{$\tfplug$}\\
        \hline
        \diagbox[width=1.4cm, height=0.6cm]{$n$}{$s_{r}$}& \centering $\sigma {n}^{1/2}\log^{2/3}n$& \centering $\sigma n^{4/5}$& \centering $\sigma {n}^{1/2}\log^{2/3}n$& \centering \arraybackslash $\sigma n^{4/5}$\\
       \hline
        \centering 20& \centering 0.937& \centering 0.923& \centering 0.596& \centering \arraybackslash  0.576\\
        \centering 100& \centering 0.999& \centering 1.000& \centering 0.592& \centering \arraybackslash  0.832\\
        \centering 400& \centering 1.000& \centering 1.000& \centering 0.433& \centering \arraybackslash  0.948\\
        \centering 700& \centering 1.000& \centering 1.000& \centering 0.349& \centering \arraybackslash  0.970\\
        \centering 1000& \centering 1.000& \centering 1.000& \centering 0.308& \centering \arraybackslash  0.982\\
       \hline
    \end{tabular}
    \caption{Empirical rejection rates under $\hypa$ in \cref{body-eq:T2} for the two tests with rejection rules $\tplug \geq F\inv_{G}(0.95)$ and $\tfplug \geq F_{Z}\inv(0.95)$, respectively.
    See \cref{body-section:power-against-local-alternatives}.}
    \label{body-tab:power-simulations}
\end{table}

\cref{body-fig:power-local,body-tab:power-simulations} compare the testing approaches $\widetilde{\phi}_{\alpha,n}(\mh{u}) = \Is{\tplug \geq F_{\gumbelrv}\inv(1-\alpha)}$ and $\widetilde{\phi}_{\alpha,n,\operatorname{F}}(\mh{u}) = \Is{\tfplug \geq F_{Z}\inv(1-\alpha)}$, where
\begin{align*}
    \tplug
    =
    \widetilde{a}_{n}\inv (\tinorm{\mh{u}\sgn(\mh{u}\T\m{u}_{0}) - \m{u}_{0}}
    -
    \widetilde{b}_{n})
    +
    \log(\abplug)
    ,
\end{align*}
and
\begin{align*}
    \tfplug
    =
    \frac{\fnorm{\mh{u}\mh{u}\T -\m{u}_{0}\m{u}_{0}\T}^{2}
    +
    \fnorm{\mh{v}\mh{v}\T - \m{v}\m{v}\T}^{2}
    -
    4(n-r){\widetilde{s}_{r}^{-2}}}{\sqrt{16(n-r)}{\widetilde{s}_{r}^{-2}}}
    .
\end{align*}
In more detail, under $\hypa$ in \cref{body-eq:T2}, the signal matrix is $\m{M} = s_{r}\m{u}_{1}\m{v}\T \in \R^{n \times m}$, where $\m{u}_{1}$ is specified in \cref{body-eq:T2}, $\m{v} \coloneqq \1_{m}/\sqrt{m}$, $n = m$, and $\m{E}$ is generated with i.i.d. $\NN(0,\sigma^{2})$ entries with $\sigma=1$.
Here, by assumption, $\m{u}_{0}$, $\m{u}_{1}$, and $\m{v}$ are all delocalized, hence $\mu=1$.

In \cref{body-fig:power-local}, we compare the empirical rejection rates of the two testing approaches when $n = 400$ is fixed while varying $s_{r}$ and $\xi$.
For each $(s_{r}, \xi)$ pair, we compute the plug-in test statistics $\tplug$ and $\tfplug$ for $\nmc = 3000$ simulation replicates each and report the empirical rejection rates under nominal level $\alpha = 0.05$.
The bottom-center regions of the contour plots demonstrate that our test is better able to detect the alternative than using $\tfplug$, especially for small values of row discrepancy, $1 \leq \xi \leq 8$.

\cref{body-tab:power-simulations} further compares the empirical rejection rates between the two testing approaches when $\xi = 1$ is fixed, while varying $n$ under two signal regimes. 
For each $(s_{r}, n)$ pair, the plug-in test statistics are computed $\nmc=8000$ times.
\cref{body-tab:power-simulations} confirms that the test based on $\tplug$ consistently rejects the null for both weak and strong signals as $n$ grows.
In contrast, the performance of the test based on $\tfplug$ deteriorates under the weaker signal setting as $n$ grows and improves only gradually under the stronger signal setting.

Above, we have shown the advantage of a two-to-infinity norm test statistic compared to a Frobenius norm test statistic for parameter space $\Theta$ per \cref{body-eq:hypothesis-parameter-space}.
That being said, we emphasize that \citet[Theorem 4]{xia_normal_2021} holds under a weaker assumption on $s_{r}$ than required in \cref{body-assumption:snr} and does not require restricting $\tinorm{\m{U}}$ per \cref{body-eq:hypothesis-parameter-space}.
The slightly stronger assumptions employed in this paper enable more fine-grained inferential theory compared to adopting weaker structural assumptions on the matrix denoising model.

\subsubsection{Power phase transition}
\label{body-section:power-phase-transition}

In this section, we consider testing the hypotheses in \cref{body-eq:T1}.
We compute the empirical rejection rate based on various combinations of $n$ and $d_{n}$ to illustrate \cref{body-theorem:power-analysis}.
Matrices $\m{U}_{1}$ are constructed by interpolating between $\m{U}_{0}$ and a delocalized $r$-frame $\m{W} \in \R^{n \times r}$ whose columns span a random $r$-dimensional subspace in the orthogonal complement of the column space of $\m{U}_{0}$.
To ensure that each $\m{U}_{1}$ has orthonormal columns, we generate $\m{U}_{1} =\operatorname{QR}( (1-t)\m{U}_{0} + t\m{W} )$ for $t \in [0,1]$.

\begin{figure}[t]
    \centering
    \begin{minipage}{.5\textwidth}
        \includegraphics[width=\linewidth]{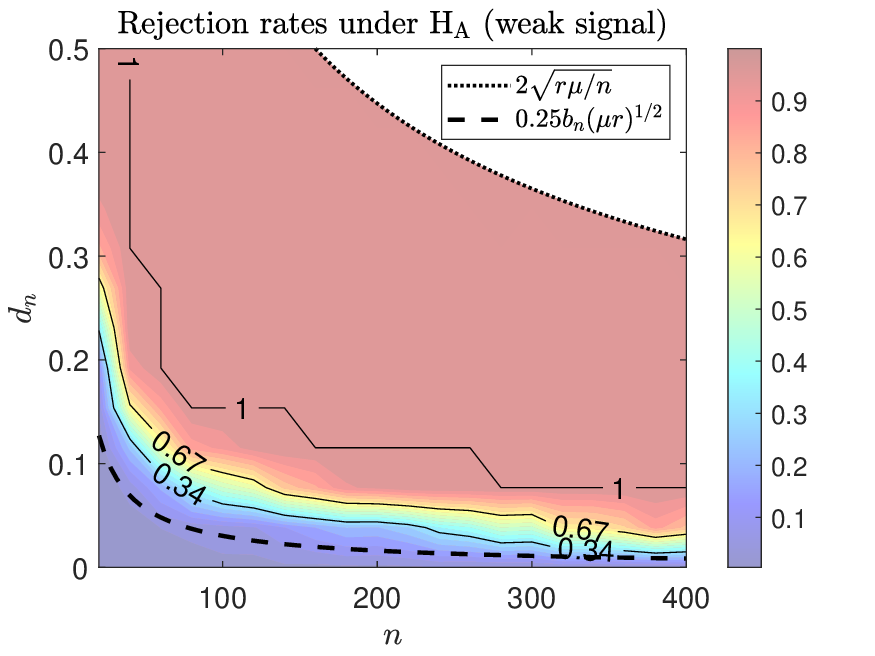}
    \end{minipage}%
    \begin{minipage}{.5\textwidth}
        \includegraphics[width=\linewidth]{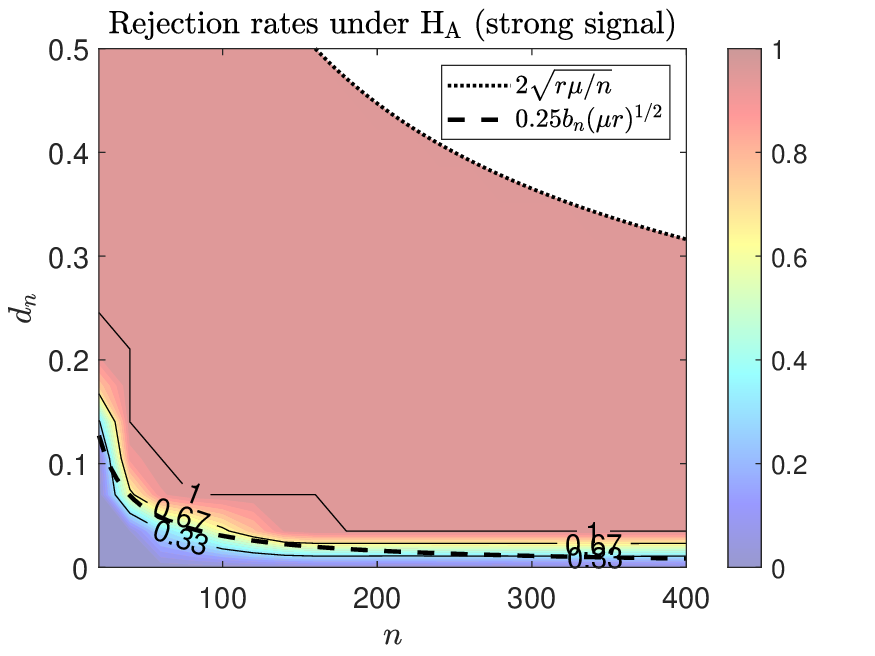}
    \end{minipage}
    \caption{Empirical rejection rates using $\tplug$ for various pairs of $(n,d_{n})$.
    Above the dotted line is the exterior of the parameter space $\Theta$, where the assumptions for \cref{body-theorem:power-analysis} do not hold.
    The dashed line represents the threshold for consistent testing.
    See \cref{body-section:power-phase-transition}.}
    \label{body-fig:power-phase}
\end{figure}        

In both panels of \cref{body-fig:power-phase}, the underlying signal matrix is specified as $\m{M} = \m{U}_{1}\m{SV}\T$ with fixed rank $r = 10$ for each pair $(n,d_{n})$. 
The left panel corresponds to a weak signal regime with the non-zero signal singular values being equally spaced between $s_{r} = 1.2\sigma {n}^{1/2}\log^{2/3}n$, and $s_{1} = 3\sigma {n}^{1/2}\log^{2/3}n$, while the right panel corresponds to a strong signal regime with the non-zero signal singular values being equally spaced between $s_{r} = n$, and $s_{1} = 3n$.
For each pair $(n,d_{n})$, the noise matrix $\m{E}$ is generated independently $\nmc = 2000$ times, with entries drawn i.i.d. from $\NN(0, \sigma^{2})$ with $\sigma = 1$.
The empirical rejection rate is computed as the proportion of samples for which $\tplug \geq F^{-1}_{G}(0.95)$.
\cref{body-fig:power-phase} corroborates the phase transition behavior predicted by \cref{body-theorem:power-analysis}.

\subsection{Robustness to misspecification of noise distribution}
\label{body-section:departures-from-Gaussianity}

\begin{figure}[t]
    \centering
    \begin{minipage}{.33\textwidth}
        \includegraphics[width=\linewidth]{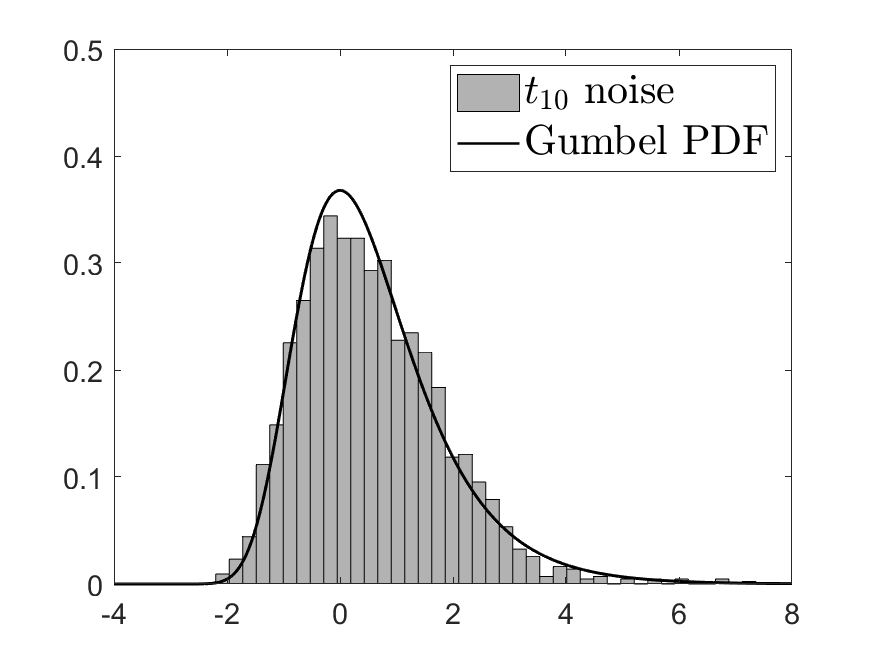}
    \end{minipage}%
    \begin{minipage}{.33\textwidth}
        \includegraphics[width=\linewidth]{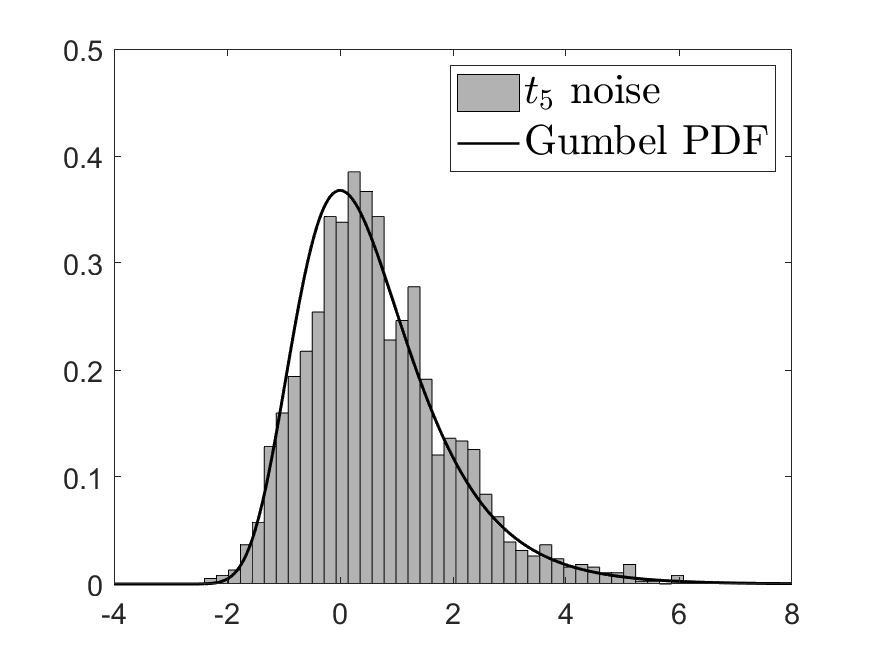}
    \end{minipage}%
    \begin{minipage}{.33\textwidth}
        \includegraphics[width=\linewidth]{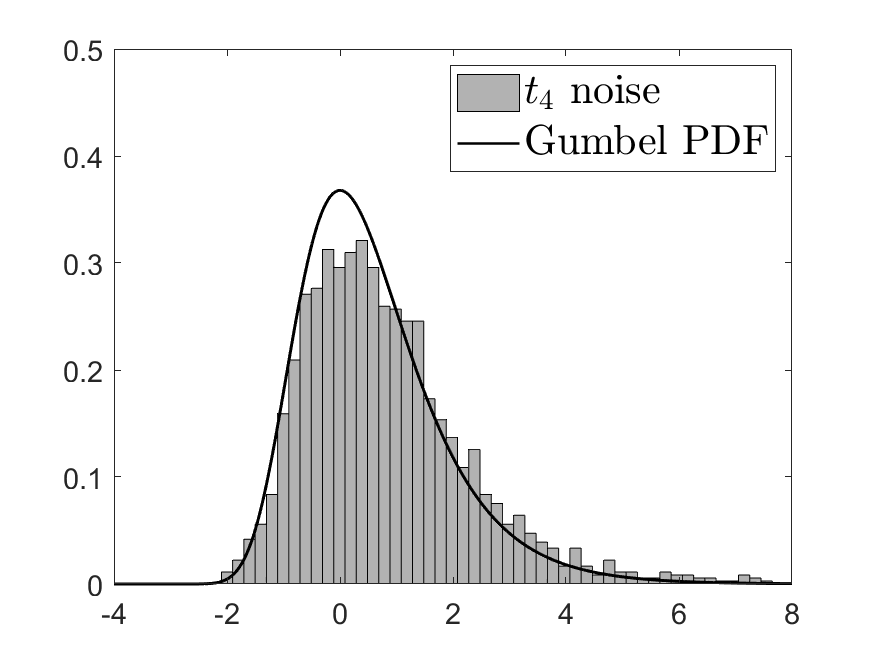}
    \end{minipage}
    \begin{minipage}{.33\textwidth}
        \includegraphics[width=\linewidth]{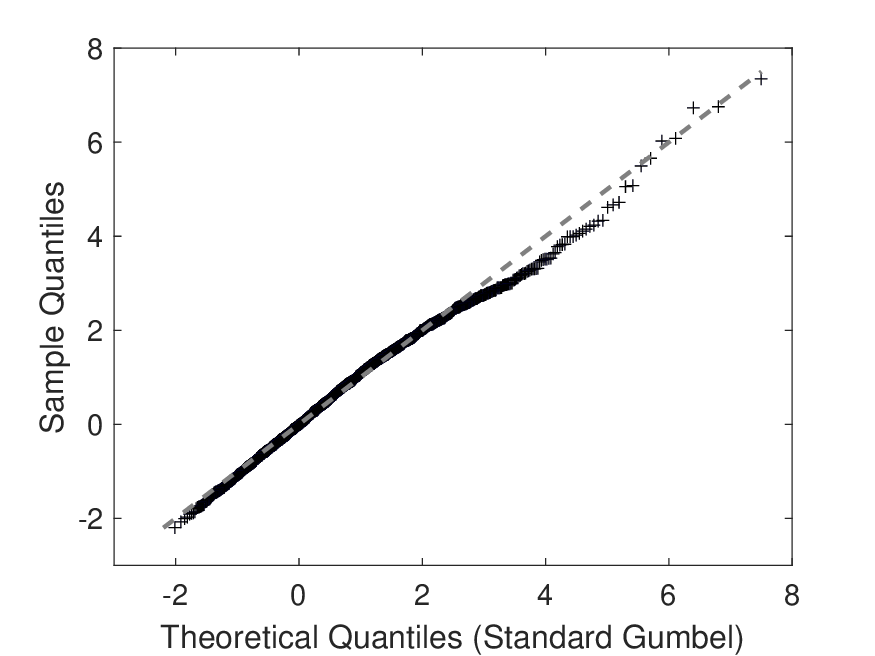}
    \end{minipage}%
    \begin{minipage}{.33\textwidth}
        \includegraphics[width=\linewidth]{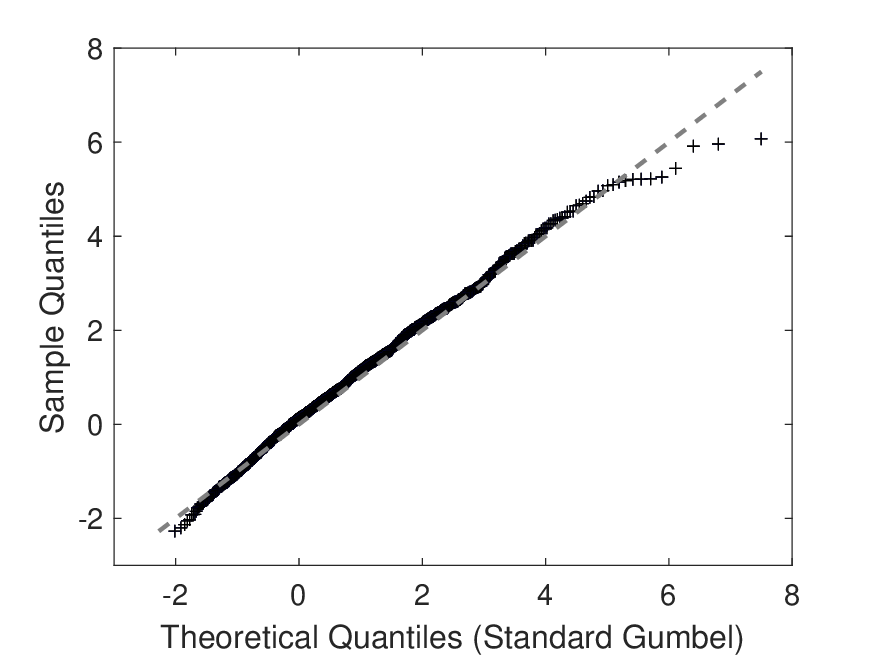}
    \end{minipage}%
    \begin{minipage}{.33\textwidth}
        \includegraphics[width=\linewidth]{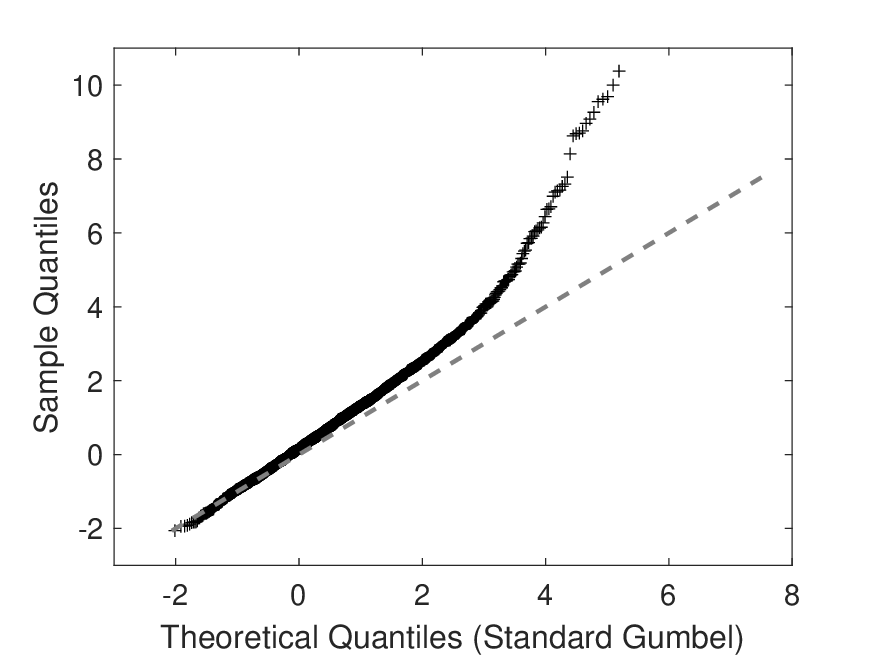}
    \end{minipage}
    \caption{Values of $\tplug$ under noise from different standardized $t$ distributions.
    See \cref{body-section:departures-from-Gaussianity}.} 
    \label{body-fig:robustness-nongaussian}
\end{figure}

This section empirically investigates the extent to which the convergence in distribution in \cref{body-theorem:Gumbel-convergence-plugin} approximately holds for non-Gaussian noise distributions.
To this end, the entries of $\m{E}$ are taken to be i.i.d. from the standardized Student $t_{\nu}$ distribution, $E_{i,j} \overset{\dd}{=} \sigma \tau_{\nu} \sqrt{(\nu-2)/\nu}$, where $\tau_{\nu} \sim t_{\nu}$.
We consider degrees of freedom $\nu = 4, 5, 10$, noting that the $t_{\nu}$ distribution exhibits heavier tails as $\nu$ decreases.
In particular, when $\nu = 10$ the distribution has mean zero and variance one but slightly heavier tails than $\NN(0,1)$, whereas when $\nu = 4$ the distribution has mean zero and variance one but no finite fourth moment.
Here, $\nu = 5$ is the smallest choice of $\nu$ that results in the distribution having finite excess kurtosis.

We set $\sigma=1$, $r=5$, $n=2500$, and $m=3000$.
The matrices $\m{U}$ and $\m{V}$ are generated as in \cref{body-section:simulation-setup-convergence} by orthonormalizing the columns of Gaussian random matrices.
The matrix $\m{S}$ has equally spaced diagonal entries between $s_{r}=n$ and $s_{1}=3n$.
For each noise distribution, $\nmc = 1800$ i.i.d. copies of $\tplug$ are obtained.
\cref{body-fig:robustness-nongaussian} shows that the empirical distribution of $\tplug$ aligns reasonably well with the Gumbel distribution under $t_{10}$ distributed noise, and to a slightly lesser extent under $t_{5}$ distributed noise.
In contrast, the distribution of $\tplug$ is no longer well-approximated by the Gumbel distribution under $t_{4}$ noise.
These examples suggest that the conclusions of \cref{body-theorem:Gumbel-convergence,body-theorem:Gumbel-convergence-plugin} may be moderately robust to certain moment-matched non-Gaussian heavy-tailed departures from \cref{body-assumption:noise}.

\supprefcap{supp-section:additional-experiments-non-Gaussian} provides further numerical experiments regarding the relaxation of the noise distribution assumption, highlighting in particular the large-sample behavior of the test statistic under non-Gaussian block-structured models, where the noise matrix has centered Bernoulli or Poisson entries.

\section{Discussion}
\label{body-section:discussion}

This paper establishes novel extreme value distributional theory for $\ell_{2,\infty}$ singular subspace estimation and inference in the matrix denoising model, with a focus on Gaussian noise.
In particular, we characterize properties of the two-to-infinity norm for estimation and testing when considering a low-rank truncation of an observable data matrix.
Our results develop upon and bring together previously separate tools from the study of spectral methods, saddle point approximation methods, and extreme value theory.

The proposed test statistic, involving the maximum row-wise Euclidean norm, is well-suited for detecting alternative hypotheses that differ from the null in only a few entries or rows.
In contrast, existing works described throughout \cref{body-section:related_work} either
(i) test for aggregate differences in subspace structure but are unable to detect small row-wise differences between hypotheses or 
(ii) test for differences between one or several specified rows in singular subspace matrices but ignore aggregate subspace structure.

In this paper, \cref{body-assumption:noise,body-assumption:matrix-size,body-assumption:delocalization,body-assumption:snr,body-assumption:plug-in-asms} stem from general considerations and developments in the use of subspace perturbation analysis and random matrix theory techniques to studying the matrix denoising model.
On the other hand, \cref{body-assumption:gap-lam1-lam2,body-assumption:aligned-alternative} are specific to this paper, preventing potential ambiguities and boundary-type situations.
Taken together, these assumptions, which are well motivated and mild given the state of the existing literature, enable our development of extreme value theory for singular subspace estimation.

We conclude by mentioning several possible extensions and opportunities for future work.

\begin{itemize}
    \item \emph{Plug-in inference under heteroskedastic noise.}
    \cref{body-theorem:Gumbel-convergence} holds under the assumption that the rows of the noise matrix are i.i.d. Gaussian vectors with a diagonal covariance matrix, while the plug-in version, \cref{body-theorem:Gumbel-convergence-plugin}, holds under the assumption that the covariance matrix is isotropic, to facilitate the singular value de-biasing procedure in \cref{body-proposition:debiased-singular-values}.
    In principle, to handle a general diagonal covariance matrix satisfying \cref{body-assumption:noise}, one could instead consider applying a whitening transformation to the data, though at the expense of more involved analysis and presentation.
    Approaches for whitening rows or columns in the presence of heteroskedastic noise have been recently studied in the context of matrix denoising \citep{gavish_matrix_2023} and principal component analysis \citep{hong_asymptotic_2018,hong_heppcat_2021,hong_optimally_2023}.

    \item \emph{Imbalanced matrix dimensions.}
    The first-order approximation in \cref{body-lemma:first-order-approximation} no longer holds when the matrix dimensions $n$ and $m$ are highly imbalanced, say, when $n$ is much smaller than $m$.
    In such settings, it is preferable to instead consider the eigendecomposition of a diagonally modified Gram-type matrix  \citep{cai_subspace_2021,yan_inference_2024}.
    There, a quadratic term with respect to the noise matrix appears in the principal subspace perturbation analysis, and a more delicate treatment of bias is required.

    \item \emph{Non-Gaussian noise.}
    The main objective of this paper is to provide comprehensive extreme value theory for singular subspace estimation in the Gaussian matrix denoising model.
    In addition to doing so, this paper also provides a partial theoretical and empirical investigation of non-Gaussian noise settings.
    Characterizing non-Gaussian noise settings, to the extent that doing so is possible, remains an interesting direction for future work.
    
    \item \emph{Two-sample testing.}
    In settings where more than one data matrix is available, it may be of interest to extend the ideas in this paper to accommodate, for example, two-sample spectral-based testing problems, as have been previously considered elsewhere but for network data \citep{tang_semiparametric_2017,tang_nonparametric_2017,li_two-sample_2018}.
\end{itemize}

\section*{Acknowledgment}
\label{section:acknowledgment}
J. Chang was supported in part by a Summer Research Fellowship from the Department of Statistics at the University of Wisconsin--Madison.
J. Cape was supported in part by the National Science Foundation under grant DMS 2413552 and by the University of Wisconsin--Madison, Office of the Vice Chancellor for Research and Graduate Education, with funding from the Wisconsin Alumni Research Foundation.
The authors thank Yuling Yan for helpful comments and technical discussions.

\bibliographystyle{apalike}
\bibliography{references}   

\appendix

\section{Proofs of results in \cref{\mainlabel{body-section:preliminaries}}}

\subsection{Proof of \cref{\mainlabel{body-lemma:first-order-approximation}}}
\label{supp-pf:first-order-approximation}
\begin{proof}
    By \cref{\mainlabel{body-assumption:noise},\mainlabel{body-assumption:snr}}, the noise matrix has zero-mean Gaussian entries with variances bounded above by $\sigma^{2}$, and $s_{r}\gg \sqrt{\mu\rlog}(\sigma\sqrt{n})\gg \sigma \sqrt{n}$, satisfying the requirements for \citet[Proposition 2]{yan_entrywise_2024} to hold. 
    Thus, alongside \cref{\mainlabel{body-assumption:matrix-size}}, it holds that for any $\mathcal{I}\subseteq \dpar{n}$ such that $|\mathcal{I}| = I$,
    \begin{align*}
        \tnorm{
            (\mh{U}\RU - \m{U})_{\mathcal{I},\all}
        }
        \lesssim
        \frac{\sigma \sqrt{I + r + \log n}}{s_{r}}
        +
        \frac{\sigma^{2}n}{s_{r}^{2}}
        \sqrt{\frac{\mu r}{n}}
    \end{align*}
    and
    \begin{align*}
        \tnorm{
            (\mh{U}\RU - \m{U} - \m{EVS}\inv)_{\mathcal{I},\all}
        }
        &\lesssim
        \frac{\sigma\sqrt{I + r + \log n}}{s_{r}}
        \blr{
            \frac{\sigma\sqrt{I + n + \log n}}{s_{r}} 
            + 
            \frac{\sigma^{2}n}{s_{r}^{2}}
        }
        \nonumber
        \\
        &\qquad +
        \plr{
            \frac{\sigma^{2}n}{s_{r}^{2}} 
            + 
            \frac{\sigma\sqrt{r + \log n}}{s_{r}}
        }
        \sqrt{
            \frac{\mu r}{n}
        },
    \end{align*}
    each with probability exceeding $1-O(n^{-10})$.
    Taking $\mathcal{I}_{i}\coloneqq \{i\}$, and applying the Bonferroni inequality, then the union bound to obtain an upper bound on the maximum over all $i\in \dpar{n}$ yields that
    \begin{align*}
        \max_{i\in\dpar{n}} 
            \tnorm{
                (\mh{U}\RU - \m{U})_{i,\all}
            }
            \lesssim
            \frac{\sigma\sqrt{r + \log n}}{s_{r}}
            + \frac{\sigma^{2}\sqrt{\mu r n}}{s_{r}^{2}}
    \end{align*}
    and
    \begin{align*}
            \max_{i\in\dpar{n}} 
            \tnorm{
                (\mh{U}\RU - \m{U} - \m{EVS}\inv)_{i,\all}
            }
            \lesssim
            \frac{\sigma\sqrt{r + \log n}}{s_{r}}
            \plr{
                \frac{\sigma\sqrt{n}}{s_{r}} 
                +
                \sqrt{
                    \frac{\mu r}{n}
                }
            }
            + 
            \frac{\sigma^{2}\sqrt{\mu r n}}{s_{r}^{2}}
    \end{align*}
    each hold with probability exceeding $1-O(n^{-9})$.
\end{proof}

\subsection{Proof of \cref{\mainlabel{body-lemma:MGF}}}
\label{supp-pf:MGF}
\begin{proof}
By hypothesis, the diagonal matrices $\m{D}$ and $\m{S}$ have strictly positive main diagonal elements, and $\m{V}$ has full column rank, hence the real symmetric matrix $\m{P} \coloneqq 2 \m{S}\inv \m{V}\T \m{D} \m{V} \m{S}\inv \in \R^{r \times r}$ is positive definite with eigenvalues denoted by $\lambda_{1} \geq \dots \geq \lambda_{r} > 0$.
The special case $\m{D} = \sigma^{2}\m{I}_{m}$ yields $\m{P} = 2\sigma^{2}\m{S}^{-2}$, hence $\lambda_{j} = 2\sigma^{2}/s_{r-j+1}^{2}$ for $1 \leq j \leq r$.

Applying \citet[Lemma 2.1]{chen_spectral_2021} gives the stated, general upper and lower bounds for $\lambda_{1}$, namely
\begin{equation*}
    \lambda_{1}
    =
    \tnorm{
    2 \m{S}\inv \m{V}\T \m{D} \m{V} \m{S}\inv
    }
    \leq
    2
    \tnorm{ \m{S}\inv }
    \tnorm{ \m{V}\T \m{D} \m{V} } 
    \tnorm{ \m{S}\inv }
    \leq
    2 \sigma_{1}^{2} / s_{r}^{2}
\end{equation*}
and
\begin{equation*}
    \lambda_{1}
    =
    \tnorm{
    2 \m{S}^{-2} \m{V}\T \m{D} \m{V}
    }
    \geq
    2
    \tnorm{ \m{S}^{-2} }
    \lambda_{\min}(\m{V}\T \m{D} \m{V}) 
    \geq
    2 \sigma_{m}^{2} / s_{r}^{2}
    ,
\end{equation*}
where we have used the fact that $\m{S}\inv \m{V}\T \m{D} \m{V} \m{S}\inv$ and $\m{S}^{-2} \m{V}\T \m{D} \m{V}$ have the same eigenvalues.

For each $1 \leq i \leq n$, recall that
$
X_{i}
=
\tnorm{\m{E}_{i, \all}\T\m{VS}\inv}
$
where
$
\m{E}_{i, \all}\T\m{VS}\inv
\sim
\NN_{r}(\m{0}, \m{S}\inv\m{V}\T\m{D}\m{V}\m{S}\inv)
$.
A direct application of \citet[Theorem 3.2a.1]{mathai_quadratic_1992} yields that the moment generating function of the quadratic form $X_{i}^{2}$ is given by
\begin{align*}
        M_{X_{i}^{2}}(t)
        =
        \det(
        \m{I}_{r} - t\m{P}
        )^{-1/2}
        =
        \prod_{j=1}^{r}
        (1-\lambda_{j} t)^{-1/2}
        ,
        \quad
        \text{for }
        t < 1/\lambda_{1}
        ,
\end{align*}
where the final equality uses the fact that the eigenvalues of $\m{I}_{r} - t\m{P}$ are of the form $1 - \lambda_{j} t$ and that the matrix determinant equals the product of the matrix eigenvalues.
This completes the proof of \cref{\mainlabel{body-lemma:MGF}}.
\end{proof}

\section{Proofs of main results}
\label{supp-section:proof-of-main-results}

\subsection{Towards proving \cref{\mainlabel{body-proposition:saddle-point-tail-equivalence}}}
\label{supp-section:proof-of-saddle-point-approximation}

\cref{\mainlabel{body-proposition:saddle-point-tail-equivalence}} asserts that the distribution of $X$ is tail-equivalent to that of the generalized gamma random variable $H$ defined in \cref{\mainlabel{body-section:Main-results}}. The proof relies on a saddle point density approximation \citep{daniels_saddlepoint_1954}, which we derive in detail. \cref{supp-section:saddle-point-approx-overview} briefly discusses saddle point approximations in more generality.

\subsubsection{Deriving the saddle point density approximation}
\label{supp-section:deriving-saddle-point}

The density in question is that of the random variable $X$, or equivalently, the density of the Euclidean norm of each length-$r$ row of $\m{EVS}\inv$. To avoid complications involving the $\sqrt{\lambda_{1}}$ scaling that changes the distribution of $H$ as $n$ grows, we work with the re-scaled random variables $\xp\coloneqq \lambda_{1}^{-1/2}X$ and $H_{1} = \lambda_{1}^{-1/2} H \sim \operatorname{GG}(1,\mul,2)$, which are stable in $n$. Also, to avoid complications involving square roots, we first study the density of $\xp^2$, and then transform it back to that of $\xp$, using the fact that $\xp\geq 0$ with probability one, hence $\P(\xp\leq x) = \P(\xp^2\leq x^2)$ for all $x\geq 0$. In other words, 
\begin{equation*}
    f_{\xp}(x) 
    = 
    2xf_{\xp^{2}}(x^2),
    \qquad 
    x\geq 0.
\end{equation*}
By \cref{\mainlabel{body-lemma:MGF}}, and the fact that $M_{\xp^{2}}(t) = M_{X^{2}}(t/\lambda_{1})$, we have $M_{\xp^2}(t)=\prod_{j=1}^r (1- t\lambda_j / \lambda_{1})^{-{1}/{2}}$ and $K_{X^2}(t) = -\frac{1}{2}\sum_{j=1}^r \log(1-t \lambda_{j}/\lambda_{1})$, provided $t<1$. The following lemma asserts that the function $\ell(t)\coloneqq tx-K_{\xp^2}(t)$ has a unique saddle point on $t\in(-\infty,1)$, which serves as the basis for the ensuing saddle point density approximation. The proof of \cref{supp-lemma:saddle-point} is provided in \cref{supp-pf:saddle-point}.

\begin{lemma}\label{supp-lemma:saddle-point}
    For $x>0$, the function $\ell(t)\coloneqq tx-K_{\xp^2}(t)$ has a unique real-valued saddle point ${\spt_{x}}<1$, which satisfies $K_{\xp^2}'({\spt_{x}})=x$ and $K_{\xp^2}''({\spt_{x}}) > 0$. In addition, if $x>\E[\xp^{2}]$, then 
    \begin{align}
        {\spt_{x}}
        = 
        1
        -
        \frac{\mul}{2x} 
        +
        g(x)
        > 0,
        \label{supp-eq:saddle-point}
    \end{align}
    where $g(x) = O(x^{-2})$ as $x\to\infty$. Moreover, if $r=1$, then ${\spt_{x}} = 1 - \frac{1}{2x}$.
\end{lemma}

The condition $x>\E[\xp^2]$ eventually holds since $x\to\infty$, thus \cref{supp-eq:saddle-point} is used throughout the subsequent derivations. 

The \emph{saddle point approximation} \citep{daniels_saddlepoint_1954} of the density $f_{\xp^2}$ evaluated at $x$ is defined as
\begin{align}
    \label{supp-eq:saddle-point-approximation}
    \widetilde{f}_{\xp^2}(x) 
    \coloneqq
    \frac{\exp\plr{K_{\xp^2}({\spt_{x}})-{\spt_{x}}}}
    {\sqrt{K_{\xp^2}''({\spt_{x}})}},
\end{align}
where $\spt_{x}$ is the unique saddle point of the function $\ell(t)\coloneqq tx - K_{\xp^2}(t)$ in \cref{supp-lemma:saddle-point}. The relative error of the saddle point density approximation at the value $x$ is defined by
\begin{equation*}
    \epsilon(x) 
    \coloneqq
    \frac{f_{\xp^2}(x) - \widetilde{f}_{\xp^2}(x)}
    {\widetilde{f}_{\xp^2}(x)}.
\end{equation*}
In order to facilitate the analysis of the relative error at $x$, where $x>\E[\xp^2]$ is considered to be a fixed value for now, define a random variable $\xp^{2}[\spt_{x}]$ whose probability density function is the exponentially tilted density of ${\xp^2}$ by $\spt_{x}$, namely
\begin{align}
    \label{supp-eq:exponential-tilt}
    f_{\xp^{2}[\spt_{x}]}(y)  
    \coloneqq
    \exp\plr{\spt_{x} y-K_{\xp^2}(\spt_x)}
    f_{\xp^2}(y),
\end{align}
where the support of this density is $y\geq 0$.
Notice that the tilted density $f_{\xp^{2}[\spt_{x}]}$ is a proper density function since it is non-negative and satisfies
\begin{align*}
    \int_{-\infty}^\infty f_{\xp^{2}[\spt_{x}]}(y)\dd y 
    = 
    \frac{1}{{M_{\xp^2}(\spt_{x})}}\int_{-\infty}^\infty {\exp({\spt_{x} y}) f_{\xp^2}(y)}\dd y 
    = 1.
\end{align*}
Furthermore, the moment generating function of $\xp^{2}[\spt_{x}]$ is related to that of $\xp^2$ (given in \cref{\mainlabel{body-lemma:MGF}}) by
\begin{align*}
    M_{\xp^{2}[\spt_{x}]}(s) 
    =
    \frac{1}{M_{\xp^2}(\spt_{x})}
    \int_{-\infty}^\infty 
    \exp(\spt_{x} y +sy)f_{\xp^2}(y)\dd y
    =
    \frac{M_{\xp^2}(\spt_x+s)}{M_{\xp^2}(\spt_x)},
    \quad
    \text{ for } s < 1 - \spt_{x},
\end{align*}
while the characteristic function of $\xp^{2}[\spt_{x}]$ is
\begin{align}
    \label{supp-eq:characteristic-function-s-tilted-density}
    \varphi_{\xp^{2}[\spt_{x}]}(s)
    =
    \frac{M_{\xp^2}(\spt+\ir s)}
    {M_{\xp^2}(\spt)},
\end{align}
for $s\in\R$.
In particular, rearranging the terms in \cref{supp-eq:exponential-tilt} and plugging in $y=x$ more clearly shows the relationship
\begin{align}
    \label{supp-eq:saddle-point-standard-form}
    f_{\xp^2}(x) 
    =
    \plr{\frac{\exp\plr{K_{\xp^2}({\spt_{x}})-{\spt_{x}}x}}{\sqrt{K_{\xp^2}''({\spt_{x}})}}}
    \plr{\sqrt{K_{\xp^2}''({\spt_{x}})}  f_{\xp^{2}[\spt_{x}]}(x)}
    \eqqcolon
    {\widetilde{f}_{\xp^2}(x)}
    \plr{1+\epsilon(x)}.
\end{align}

To further study the term $1+\epsilon(x) = \sqrt{K_{\xp^2}''({\spt_{x}})}  f_{\xp^{2}[\spt_{x}]}(x)$ in \cref{supp-eq:saddle-point-standard-form}, we construct a random variable whose density evaluated at a specific point is equal to $1+\epsilon(x)$. Consider the random variable $W_x \coloneqq {\xp^{2}[\spt_{x}]}/{\sqrt{K_{\xp^2}''(\spt_{x})}}$, where $x>\E[\xp^{2}]$ is still held fixed. Since $W_{x}$ is a scalar multiple of the random variable $\xp^{2}[{\spt_{x}}]$, it has the density
\begin{align}
\label{supp-eq:density-change-of-variables}
    f_{W_x}(w_x(z))
    =
    f_{\xp^{2}[\spt_{x}]}(z)
    \sqrt{K_{\xp^2}''({\spt_{x}})}  ,
\end{align}
where $w_x(z)\coloneqq {z}/{\sqrt{K_{\xp^2}''({\spt_{x}})}}$, and the support of this density is $z \geq 0$. Thus, plugging in $z=x$ to \cref{supp-eq:density-change-of-variables} yields $1+\epsilon(x) = f_{W_x}(w_x(x))$, that is, \cref{supp-eq:saddle-point-standard-form} can be written as
\begin{equation}
    f_{\xp^2}(x) 
    =
    {\widetilde{f}_{\xp^2}(x)}
    f_{W_x}(w_x(x)).
    \label{supp-eq:saddle-point-simplified}
\end{equation}
In other words, the analysis of the relative error of the saddle point density approximation reduces to analyzing the density function of the random variable $W_x$ at the point $w_x(x)={x}/{\sqrt{K_{\xp^2}''({\spt_{x}})}}$. The motivation for constructing the random variable $W_{x}$ is that its characteristic function satisfies
\begin{align}
    \varphi_{W_x}(s) 
    = \int_{-\infty}^{\infty} \exp(\ir s w_x) f_{W_x}(w_x)\dd w_x 
    &=
    \int_{-\infty}^{\infty} 
    \exp\left(\frac{\ir sz}{\sqrt{K_{\xp^2}''(\spt_{x})}}\right)
    f_{\xp^{2}[\spt_{x}]}\left(z\right)
    \dd z 
    \nonumber\\
    &= 
    \varphi_{\xp^{2}[\spt_{x}]}
    \left(\frac{s}{\sqrt{K_{\xp^2}''(\spt_{x})}}\right)
    \nonumber \\
    &\overset{\eqref{supp-eq:characteristic-function-s-tilted-density}}{=}
    \frac{M_{\xp^2}\left({\spt_{x}}+\frac{\ir s}{\sqrt{K_{\xp^2}''({\spt_{x}})}}\right)}
    {M_{\xp^2}({{\spt_{x}}})},
    \label{supp-eq:transformed-cf}
\end{align}
which is a key result used in \cref{supp-section:lemmas-for-saddlepoint-approx} to show that $\lim_{x\to\infty} (1+\epsilon(x)) = \lim_{x\to\infty}f_{W_x}(w_x(x))
=
\frac{
    (\mul/2)^{\frac{\mul-1}{2}}
    \exp(-\mul/2)
}
{
\Gamma(\mul/2)
}$,
implying that the saddle point density approximation of the density $f_{\xp^{2}}$ becomes exact up to a multiplicative constant as $x\to\infty$.

\subsubsection{Tail behavior of the saddle point approximation error}
\label{supp-section:lemmas-for-saddlepoint-approx}

\cref{supp-lemma:saddle-point-approx-explicit} provides the form of the saddle point density approximation derived in \cref{supp-section:deriving-saddle-point}. The proof of \cref{supp-lemma:saddle-point-approx-explicit} is provided in \cref{supp-pf:saddle-point-approx-explicit}.
\begin{lemma}
    \label{supp-lemma:saddle-point-approx-explicit}
    For $r\geq 2$, $\mul < r$, and $x > \E[\xp^{2}]$, the saddle point density approximation satisfies
    \begin{align*}
        \widetilde{f}_{\xp^2}(x)
        =
        \frac{
            2^{\frac{\mul-1}{2}}
            x^{\frac{\mul - 2}{2}}
            \clr{
                1
                +
                \frac{2x^2 g(x)}{\mul x- 2 x^2 g(x)}
            }^{\mul/2}
            \clr{
                \prod_{j=\mul + 1}^r (1- \spt_{x}\lambda_j /\lambda_{1})^{-1/2}
            }
            \exp
            \left(-x+\frac{\mul}{2}-xg(x)\right)
        }
        {
            \mul^{\frac{\mul-1}{2}}
            \sqrt{
                \frac{1}{\left\{1- 2xg(x)/\mul\right\}^2} 
                +
                \frac{\mul}{4x^{2}}
                \sum_{j=\mul + 1}^r 
                \frac{\lambda_j^2}{{(1-\lambda_j \spt_{x})^2}}
            }
        },
    \end{align*}
    where $g(x)= O(x^{-2})$ is as in \cref{supp-lemma:saddle-point}. 
\end{lemma}
Next, \cref{supp-lemma:fW-convergence} shows that $f_{W_x}(w_x(x))$, which equals $1+\epsilon(x)$, converges to a constant as $x\to\infty$. The proof of \cref{supp-lemma:fW-convergence} is provided in \cref{supp-pf:fW-convergence}.
\begin{lemma}
    \label{supp-lemma:fW-convergence}
    For $r\geq 2$ and $\mul < r$, it holds that
    \begin{align*}
        f_{W_{x}}(w_x(x)) 
        =
        \clr{
            \frac{
                (\mul/2)^{\frac{\mul-1}{2}}
                \exp(-\mul/2)
            }
            {
            \Gamma(\mul/2)
            }
        }
        \clr{
            1+O(x^{-1})
        }
    \end{align*}
    as $x\to\infty$, where $\Gamma(\cdot)$ denotes the gamma function.
\end{lemma} 

\subsubsection{Proof of \cref{\mainlabel{body-proposition:saddle-point-tail-equivalence}}}

The lemmas in \cref{supp-section:lemmas-for-saddlepoint-approx} enable the proof of \cref{\mainlabel{body-proposition:saddle-point-tail-equivalence}}.

\begin{proof}
\noindent
\textbf{(i)} If $r=1$, then $\xp^{2}$ is gamma distributed with shape parameter $1/2$ and scale parameter $1$, hence $f_{\xp^2}(x) = \exp(-x)/\sqrt{\pi x}$.  Since $X\geq 0$ with probability one, it holds that $\P(\xp\leq x) = \P(\xp^2\leq x^2)$ for all $x\geq 0$,  hence 
\begin{equation*}
    f_{\xp}(x) 
    =
    2xf_{\xp^2}(x^2)\I{x\geq 0} 
    = 
    \frac{2\exp(-{x^2})}
    {\sqrt{\pi}}
    \I{x\geq 0},
\end{equation*}
which is the half-normal density function $f_{H_{1}}(x)$ with parameter $\rho=\sqrt{1/2}$. It follows that
\begin{align}
    \label{supp-eq:tail-exact-r-1}
    \frac{1-F_{H_{1}}(x)}{1-F_{\xp}(x)}
    = 1
    \quad \text{ for all } x\geq 0.
\end{align}

\noindent
\textbf{(ii)} If $r \geq 2$ and $\mul = r$, then by \cref{\mainlabel{body-assumption:noise}}, $\m{D}= \sigma^{2}\m{I}_{m}$, and $\m{S} = s_{r}\m{I}_{r}$. In particular, $\xp$ is a squared sum of $r$ i.i.d.~$\NN(0,1/2)$ random variables, which by definition is equal in distribution to $H_{1}$. This implies that
\begin{align}
    \label{supp-eq:tail-exact-r-eq-mul}
    \frac{1-F_{H_{1}}(x)}{1-F_{\xp}(x)}
    = 1
    \quad \text{ for all } x\geq 0.
\end{align}

\noindent
\textbf{(iii)} For $r\geq 2$ and $\mul < r$, by \cref{supp-eq:saddle-point-simplified,supp-lemma:saddle-point-approx-explicit}, we have
\begin{align*}
    f_{\xp}(x)
    &=
    2xf_{\xp^2}(x^2)\I{x\geq 0}
    \\
    &= 
    2x \widetilde{f}_{\xp^2}(x^2)f_{W_{x^2}}(w_{x^2}(x^2))\I{x\geq 0}
    \\
    &=
    \blr{
        \frac{
            2^{\frac{\mul + 1}{2}}
            x^{\mul - 1}
            \clr{
                1
                +
                \frac{2x^{4} g(x^{2})}{\mul x^{2}- 2 x^{4} g(x^{2})}
            }^{\mul/2}
            \clr{
                \prod_{j=\mul + 1}^r (1- \spt_{x}\lambda_j /\lambda_{1})^{-1/2}
            }
            \exp
            \plr{
                -x^{2}+\frac{\mul}{2}-x^{2}g(x^{2})
            }
        }
        {
            \mul^{\frac{\mul-1}{2}}
            \sqrt{
                \frac{1}{\left\{1- 2x^{2}g(x^{2})/\mul\right\}^2} 
                +
                \frac{\mul}{4x^{4}}
                \sum_{j=\mul + 1}^r 
                \frac{\lambda_j^2}{{(1-\lambda_j \spt_{x})^2}}
            }
        }
    }
    \\
    &\qquad
    \times 
    f_{W_{x^2}}(w_{x^2}(x^2))
    \I{x\geq 0}
    \\
    &=
    \blr{
        \frac{
            2^{\frac{\mul+1}{2}}
            x^{\mul - 1}
            \clr{
                \prod_{j=\mul + 1}^r (1-  \lambda_j /\lambda_{1})^{-1/2}
            }
            \exp
            \plr{
                -x^{2}
                +
                \frac{\mul}{2}
            }
        }
        {
            \mul^{\frac{\mul-1}{2}}
        }
    }
    \clr{
        1+O(x^{-1})
    }
    \\
    &\qquad
    \times 
    \clr{
            \frac{
                (\mul/2)^{\frac{\mul-1}{2}}
                \exp(-\mul/2)
            }
            {
            \Gamma(\mul/2)
            }
        }
        \clr{
            1+O(x^{-2})
        }
    \I{x\geq 0}
    \\
    &=
    \frac{2}{\Gamma(\mul/2)}x^{\mul-1} \exp(-x^{2})
    \clr{
        \prod_{j=\mul + 1}^r 
        \plr{1-  \frac{\lambda_j}{\lambda_{1}}
        }
    }^{-1/2}
    \clr{
        1+ O(x^{-1})
    }
    \I{x\geq 0}
    \\
    &= f_{H_{1}}(x)
    \clr{
        \prod_{j=\mul + 1}^r 
        \plr{1-  \frac{\lambda_j}{\lambda_{1}}
        }
    }^{-1/2}
    \clr{
        1+ O(x^{-1})
    },
\end{align*}
for $x$ large enough, where we have used the fact that $g(x)= O(x^{-2})$ and ${\spt_{x}} =  1 + O(x^{-1})$ by \cref{supp-lemma:saddle-point}, alongside \cref{supp-lemma:fW-convergence}.
It follows that
\begin{align}
    \label{supp-eq:tail-convergence-2xfx}
    f_{H_{1}}(x)
    = 
    f_{\xp}(x)
    \clr{
        \prod_{j=\mul + 1}^r 
        \plr{1-  \frac{\lambda_j}{\lambda_{1}}
        }
    }^{-1/2}
    \clr{
        1 + O(x^{-1}) 
    }
\end{align}
for $x$ large enough, hence, for any $n,m\geq r$,
\begin{align}
    \label{supp-eq:tail-exact-r-gtr-1}
    \frac{1-F_{H_{1}}(x)}{1-F_{\xp}(x)} 
    =  
    \frac{\int_{x}^{\infty}f_{H_{1}}(x)\dd x}
    {\int_{x}^{\infty}f_{\xp}(x)\dd x}
    =
    \clr{
        \prod_{j=\mul + 1}^r 
        \plr{1-  \frac{\lambda_j}{\lambda_{1}}
        }
    }^{-1/2}
    \clr{
        1+ O(x^{-1})
    }
\end{align}
for $x$ large enough.
By \cref{supp-eq:tail-exact-r-1,supp-eq:tail-exact-r-eq-mul,supp-eq:tail-exact-r-gtr-1}, for any $n,m\geq r$,
\begin{align*}
    \frac{1-F_{H}(x)}{1-F_{X}(x)} 
    =
    \frac{1-F_{H_{1}}(\lambda_{1}^{-1/2}x)}{1-F_{\xp}(\lambda_{1}^{-1/2}x)} 
    =
    \begin{cases}
            1
            &
            \text{if }\;
            r = 1 
            \text{ or }\;
            \mul = r
            , \\
            \clr{
                \prod_{j=\mul+1}^{r}\left(1-\frac{\lambda_{j}}{\lambda_{1}}\right)
            }^{1/2}
            \clr{1 + O(\lambda_{1}^{1/2}x^{-1})}
            &
            \text{if }\;
            r > 1
            \text{ and }\;
            \mul < r
        \end{cases}
\end{align*}
for $x$ large enough, which completes the proof of \cref{\mainlabel{body-proposition:saddle-point-tail-equivalence}}.
\end{proof}

\subsubsection{Saddle point density approximations: a brief overview}
\label{supp-section:saddle-point-approx-overview}

Saddle point approximations are often employed to approximate densities and interval probabilities involving a sum of random variables \citep{kolassa_series_2006,butler_saddlepoint_2007}. \citet{daniels_saddlepoint_1954} showed that for a sample of $N$ independent and identically distributed random variables $Z_{1},\dots,Z_{N}$ with a non-degenerate MGF and a continuous and integrable characteristic function, the relative error (see \cref{supp-eq:density-change-of-variables}) of the saddle point density approximation of $f_{\overline{Z}_{N}}(x)$, the density of the sample mean $\overline{Z}_{N}$, satisfies
$
    1+\epsilon(x) 
    =
    c_0 + O(N^{-1})
$
for all $x$ and some constant $c_0$. That is, the {relative} error is uniformly controlled for all $x$ as $N$ grows, meaning that with a large enough sample size, the re-scaled saddle point density approximation $c_0 \widetilde{f}_{\overline{Z}_{N}}(x)$ is accurate even for large values of $x$ where ${f}_{\overline{Z}_{N}}(x)$ is very small. This control over relative error is a distinguishing feature of saddle point density approximations compared to Edgeworth expansions, where only the absolute error is controlled over the domain of the density function. The approach of \citet{daniels_saddlepoint_1954} relies on approximating the inverse Fourier transform of the characteristic function to recover the corresponding density function, utilizing the method of steepest descent for approximating integrals of complex-valued functions. Although the derivation of the saddle point approximation in \cref{supp-eq:saddle-point-standard-form} avoids the inverse Fourier transform entirely, the presence of the saddle point $\spt$ in the saddle point approximation is a result of choosing the steepest descent contour that passes through the saddle point $\spt$.

The random variable of interest, $X^{2}$, is typically not a sum of i.i.d.~random variables, yet \cref{supp-lemma:fW-convergence} shows that $1+\epsilon(x)$, equivalently $f_{W_x}(w_x(x))$, converges to a constant $c_0$ as $x\to \infty$, which, by \cref{supp-eq:saddle-point-simplified}, results in ${f}_{X^{2}}(x)/\widetilde{f}_{X^{2}}(x)$ converging to a constant as $x\to\infty$. In this case, the saddle point density approximation is called \emph{tail exact} \citep{barndorff-nielsen_tail_1999}. In general, tail exactness does not guarantee convergence of the approximation to the true density for small, fixed values of $x$ near the mean of $X^{2}$. Nevertheless, a tail exact approximation of $X^{2}$ is sufficient to show the tail equivalence (see \cref{\mainlabel{body-defn:tail-equivalence}}) of $X$ and the generalized gamma random variable $H$.

\subsection{Towards proving \cref{\mainlabel{body-theorem:Gumbel-convergence}}}
\label{supp-section:proof-of-Gumbel-convergence}
              
\subsubsection{Technical lemmas}

The following lemma from extreme value theory enables deriving the normalizing sequences of a distribution via those of a tail equivalent distribution. This tool is useful when it is more straightforward to compute the normalizing sequences of the tail equivalent distribution than to compute those of the original distribution.
\begin{lemma}[\citealp{resnick_extreme_2007}, Proposition 1.19]
    \label{supp-lemma:tail-equivalence}
    Suppose $S$ is a distribution function in the domain of attraction of the standard Gumbel distribution with distribution function $F_{\gumbelrv}$, that is, there exist sequences $(\alpha_{n})_{n\geq 1}$ and $(\beta_{n})_{n\geq 1}$ such that as $n\to\infty$,
    \begin{align*}
        S^n(\alpha_{n} x + \beta_{n}) 
        \to 
        F_{\gumbelrv}(x),
        \quad
        \forall x\in \R.
    \end{align*}
    Moreover, let $T$ be a distribution function that is tail equivalent to $S$ per \cref{\mainlabel{body-defn:tail-equivalence}}, namely
    \begin{align*}
        \lim_{x \to x_{0}}
        \frac{1-S(x)}{1-T(x)}
        =
        A \in (0, \infty),
    \end{align*}
    where $x_{0}$ is the common right endpoint of $S$ and $T$.
    Then, as $n\to\infty$,
    \begin{align*}
        T^n(\alpha_{n}(x-\log(A)) + \beta_{n})
        \to
        F_{\gumbelrv}(x), 
        \quad
        \forall x\in \R.
    \end{align*}
\end{lemma}

In other words, \cref{supp-lemma:tail-equivalence} implies that $\alpha_{n}\inv(\max_{i\in\dpar{n}} T_i - \beta_{n}) +\log(A)\rightsquigarrow \gumbelrv$, where $T_{1},\dots,T_{n}$ are i.i.d. samples from the distribution $T$.

\cref{supp-lemma:survival-function} characterizes the asymptotic behavior of the survival function of the reference distribution, $\overline{F}_{H}(x) \coloneqq 1-F_{H}(x)$. The proof of \cref{supp-lemma:survival-function} is provided in \cref{supp-pf:survival-function}.

\begin{lemma}[Survival function for generalized gamma distribution]
    \label{supp-lemma:survival-function}
    Let $H$ denote a random variable following the generalized gamma distribution with scale parameter $\sqrt{\lambda_{1}}$ and shape parameters $\mul$ and $2$, namely $H \sim \operatorname{GG}(\sqrt{\lambda_{1}},\mul,2)$, with density $f_{H}(x) =\frac{2}{\lambda_{1}^{\mul/2}\Gamma(\mul/2)}x^{\mul-1}\exp\left(-\frac{x^2}{\lambda_{1}}\right)\I{x\geq 0}$. The survival function $\overline{F}_{H}(x) \coloneqq 1- F_{H}(x)$ satisfies
    \begin{align*}
        \overline{F}_{H}(x) 
        =
        \begin{cases}
            \frac{\lambda_{1}^{1-\mul/2}}{\Gamma(\mul/2)}x^{\mul-2}\exp(-x^{2}/{\lambda_{1}})(1+ r(x)) 
            &
            \text{if } x > 0,
            \\
            1 
            &
            \text{if } x \leq 0,
        \end{cases}
    \end{align*}
    where $0 < r(x) \leq \frac{\lambda_{1}(\mul -2)}{2x^2 - \lambda_{1}(\mul -2)} = O(x^{-2})$ as $x\to\infty$.
\end{lemma}

\cref{supp-lemma:generalized-gamma-normalizing-constants} computes the normalizing sequences of the generalized gamma random variable $H$ which is tail equivalent to the target random variable $X$ by \cref{\mainlabel{body-proposition:saddle-point-tail-equivalence}}. The proof of \cref{supp-lemma:generalized-gamma-normalizing-constants} is provided in  \cref{supp-pf:generalized-gamma-normalizing-constants}.
\begin{lemma}[Normalizing sequences for generalized gamma distribution]
    \label{supp-lemma:generalized-gamma-normalizing-constants}
    Let $H_{(1)}, \dots, H_{(n)}$ be i.i.d.~random variables from the generalized gamma distribution $\operatorname{GG}(\sqrt{\lambda_{1}},\mul,2)$. Then, as $n\to\infty$,  
    \begin{align*}
        a_n\inv\plr{\max_{1\leq i\leq n} H_{(i)} - b_n} 
        \rightsquigarrow  
        \gumbelrv,
    \end{align*}
    where  
    $a_{n}
    =
    \dfrac{\sqrt{\lambda_{1}}}{2\sqrt{\log n}}$
    and
    $b_{n}
    = 
    \sqrt{\lambda_{1}\log n} 
    + \dfrac{\sqrt{\lambda_{1}}(\mul-2)\log\log n}{4\sqrt{\log n}}
    - \dfrac{\sqrt{\lambda_{1}}\log\Gamma(\mul/2)}{2\sqrt{\log n}}$.
    Furthermore, the sequence $(b_{n})_{n\geq 1}$ satisfies
    $n\overline{F}_{H}({b}_{n}) 
    =
    1+ O\plr{\frac{\log\log n}{\log n}}$. 
\end{lemma}

\cref{supp-lemma:first-order-convergence} establishes that the normalizing sequences for the first-order approximation $\tinorm{\m{EVS}\inv}$ computed in \cref{supp-lemma:generalized-gamma-normalizing-constants} can in fact be used for the desired target quantity $\tinorm{\mh{U}\RU-\m{U}}$.
The proof of \cref{supp-lemma:first-order-convergence} is provided in \cref{supp-pf:first-order-convergence}.

\begin{lemma}
    \label{supp-lemma:first-order-convergence}        
    Let $(a_n)_{n\geq 1}$ denote the sequence from \cref{supp-lemma:generalized-gamma-normalizing-constants}. Define
    \begin{align*}
        \gamma_{n} 
        \coloneqq
        a_{n}\inv
        \plr{
            \tinorm{\mh{U}\RU - \m{U}}
            -
            \tinorm{\m{EVS}\inv}
        },
    \end{align*}
    and the event 
    \begin{align*}
        \mathcal{C}_{n}
        \coloneqq
        \clr{
            |\gamma_{n}|
            \leq
            C_{\gamma}
            \blr{
                \sqrt{\rlog}
                \plr{
                    \frac{\sigma\sqrt{n}}
                    {s_{r}}
                    +
                    \sqrt{
                        \frac{\mu r}{n}
                    }
                }
                +
                \frac{\sigma\sqrt{\mu r n\log n}}
                {s_{r}}
            }
        },
    \end{align*}
    where $\rlog$ is defined in \cref{\mainlabel{body-defn:signal-matrix-quantities}}, and $C_{\gamma}>0$ is a universal constant.
    Under \cref{\mainlabel{body-assumption:noise},\mainlabel{body-assumption:matrix-size},\mainlabel{body-assumption:snr}}, it holds that $\P(\mathcal{C}_{n}) = 1-O(n^{-9})$ for $n$ sufficiently large.
\end{lemma}

\subsubsection{Proof of \cref{\mainlabel{body-theorem:Gumbel-convergence}}}
\begin{proof}
Let $(a_n)_{n\geq 1}$ and $(b_n)_{n\geq 1}$ denote the normalizing sequences of the generalized gamma distribution $\operatorname{GG}(\sqrt{\lambda_{1}},\mul,2)$ obtained in \cref{supp-lemma:generalized-gamma-normalizing-constants}. By \cref{\mainlabel{body-proposition:saddle-point-tail-equivalence},supp-lemma:tail-equivalence},  as $n\to\infty$,
\begin{align*}
    a_{n}^{-1}(\tinorm{\m{EVS}\inv}-b_{n}) +\log(\abias) 
    \rightsquigarrow
    \gumbelrv.
\end{align*}
Applying \cref{supp-lemma:first-order-convergence} under \cref{\mainlabel{body-assumption:noise},\mainlabel{body-assumption:matrix-size},\mainlabel{body-assumption:delocalization},\mainlabel{body-assumption:snr},\mainlabel{body-assumption:gap-lam1-lam2}} together with Slutsky's lemma yields
\begin{align*}
    &a_n^{-1}\plr{\|\mh{U}\RU- \m{U} \|_{2,\infty} -b_n} +\log(\abias) \\
    &\qquad
    =
    a_n\inv(\tinorm{\m{EVS}\inv}-b_n) +\log(\abias) 
    + \gamma_{n}\\
    &\qquad
    \rightsquigarrow 
    \gumbelrv,
\end{align*}
since $\gamma_{n} = o_{\P}(1)$. 
This completes the proof of \cref{\mainlabel{body-theorem:Gumbel-convergence}}.
\end{proof}

\subsection{Proof of \cref{\mainlabel{body-theorem:CDF-bounds}}}
\label{supp-pf:CDF-bounds}

\begin{proof}
    To avoid complications involving the $\sqrt{\lambda_{1}}$ scaling that changes the distribution of $H$ as $n$ grows, we work with the re-scaled random variables $\xp\coloneqq \lambda_{1}^{-1/2}X$ and $H_{1} = \lambda_{1}^{-1/2} H \sim \operatorname{GG}(1,\mul,2)$, which are stable in $n$.
    
    Let $\acdf(x)$ denote the CDF of $a_n^{-1}\clr{\|\m{EVS}\inv \|_{2,\infty} -b_n} +\log(\abias)$, and define the quantity
    $
    u_{n}(x)
    \coloneqq
    \lambda_{1}^{-1/2}\blr{a_{n}\clr{x - \log(\abias)} + b_{n}}
    $,
    The CDF $\acdf(x)$ satisfies
    \begin{align}
        \acdf(x) 
        &=
        \P\plr{a_n^{-1}\clr{\|\m{EVS}\inv \|_{2,\infty} -b_n} +\log(\abias) \leq x}
        \nonumber\\
        &=\P\plr{\lambda_{1}^{-1/2}\|\m{EVS}\inv \|_{2,\infty} \leq u_{n}(x)}
        \nonumber\\
        &=F_{\xp}^{n}\plr{u_{n}(x)}
        \nonumber\\
        &=\plr{1-\abias^{-1}\overline{F}_{H_{1}}\plr{u_{n}(x)}[1 + O(\clr{u_{n}(x)}^{-1})]}^{n},
        \label{supp-eq:CDF-survival-form}
    \end{align}
    where the last line is due to the intermediary non-asymptotic result of \cref{\mainlabel{body-proposition:saddle-point-tail-equivalence}} in \cref{supp-eq:tail-exact-r-1,supp-eq:tail-exact-r-eq-mul,supp-eq:tail-exact-r-gtr-1}. Importantly, \cref{supp-eq:CDF-survival-form} only holds if $u_{n}(x)\to \infty$ as $n\to \infty$, as $\overline{F}_{\xp}(x)$ and $\overline{F}_{H_{1}}(x)$ are proportional only in the large $x$ limit. Here, $x\equiv x_{n}$ is treated as a quantity that can change with $n$. 
    
    For the purpose of analysis, we partition the domain into four parts:
    \begin{align*}
        D_{1,n} \coloneqq (-\infty, x_{n,0}),
        \quad
        D_{2,n} \coloneqq [x_{n,0}, x_{n,1}],
        \quad
        D_{3,n} \coloneqq (x_{n,1}, x_{n,2}],
        \quad
        \text{and}
        \quad
        D_{4,n} \coloneqq (x_{n,2}, \infty),
    \end{align*}
    where 
    \begin{align*}
        x_{n,0}
        &\coloneqq
        (-1 + \log^{-1/4}n)a_{n}^{-1}b_{n}
        + 
        \log\abias,
        \\
        x_{n,1}
        &\coloneqq
        -\frac{1}{2}\log\log n
        + \log\Gamma(\mul/2)
        + \log \abias,
        \\
        x_{n,2}
        &\coloneqq
       \log \log n 
        + \log\Gamma(\mul/2)
        +\log \abias.
    \end{align*}
    Since $a_{n}^{-1}b_{n} = 2\log n + O(\log\log n)$,
    it holds that $x_{n,0}$ diverges to $-\infty$ at rate $O(\log n)$, $x_{n,1}$ diverges to $-\infty$ at rate $O(\log\log n)$, and $x_{n,2}$ diverges to $+\infty$ at rate $O(\log\log n)$ as $n\to\infty$ under \cref{\mainlabel{body-assumption:gap-lam1-lam2}}.

    As an overview of the proof, we proceed by uniformly upper bounding $|F_{n}(x)-F_{G}(x)|$ on each of the four sub-domains $D_{1,n}, D_{2,n}, D_{3,n}$, and $D_{4,n}$, where $F_{G}(x)\coloneqq \exp(-\exp(-x))$ is the CDF of the standard Gumbel distribution.
    
    The domain $D_{3,n}$ represents the bulk region, where a Taylor expansion of $\overline{F}_{H_{1}}(u_{n}(x))$ can be employed to obtain the desired leading term of $\acdf(x)$, and subsequently bound the error between $\acdf(x)$ and $F_{G}(x)$.
    The domain $D_{2,n}$ represents the region where $F_{G}(x)$ and $\acdf(x)$ are both small, yet the sequence $u_{n}(x_{n})$ still tends to infinity at a rate of $O(\log^{1/4}n)$ or faster so that \cref{supp-eq:CDF-survival-form} can be invoked.
    The domain $D_{4,n}$ represents the region where $1-F_{G}(x)$ and $1-\acdf(x)$ are both small due to the fast growth of $u_{n}(x_{n})$.
    In both $D_{2,n}$ and $D_{4,n}$, the absolute difference $|\acdf(x)- F_{G}(x)|$ is small due to the triangle inequality. 
    The domain $D_{1,n}$ represents the region including cases where $u_{n}(x_{n})$ grows to infinity at a rate of $O(\log^{1/4} n)$ or slower, or tends to $-\infty$ as $n\to\infty$.
    In this case, the values of $\acdf(x)$ and $F_{G}(x)$ are both negligibly small.
    Finally, we relate the CDFs $F_{n}(x)$ and $\acdf(x)$ via the concentration of the term $\gamma_{n}$ in \cref{supp-lemma:first-order-convergence} to get the final bound between $F_{n}(x)$ and $F_{G}(x)$.

\noindent
\textbf{(i)} (Bulk region)
    Suppose $x \in D_{3,n}$. By \cref{supp-lemma:survival-function},
    \begin{align}
        \overline{F}_{H_{1}}(u_{n}(x))
        =
        \frac{1}{\Gamma(\mul/2)}
        \clr{
            u_{n}(x)
        }^{\mul - 2}
        \exp
        \plr{
            -
            \clr{
                u_{n}(x)
            }^2
        }
        \clr{
            1 + O(\clr{u_{n}(x)}^{-2})
        }.
        \label{supp-eq:Survival-function-plugin}
    \end{align}
    Since $x=O(\log\log n)$, $a_{n} = \sqrt{\lambda_{1}}\plr{2\sqrt{\log n}}^{-1}$, and $b_{n} = O(\sqrt{\lambda_{1} \log n})$, it follows that the polynomial factor in \cref{supp-eq:Survival-function-plugin} satisfies
    \begin{align}
        \plr{
            \lambda_{1}^{-1/2}
            \blr{
                a_{n}\clr{x-\log(\abias)} + b_{n}
            }
        }^{\mul-2}
        &=
        (\lambda_{1}^{-1/2}b_{n})^{\mul - 2}
        \blr{
            1+ a_{n}b_{n}^{-1}
            \clr{
                x -\log(\abias)
            }
        }^{\mul - 2}
        \nonumber
        \\
        &=
        (\lambda_{1}^{-1/2}b_{n})^{\mul - 2}
        \clr{1+ O\plr{\frac{\log\log n}{\log n}}}.
        \label{supp-eq:Survival-function-polynomial-factor}
    \end{align}
    By the fact that
    \begin{align*}
        \frac{2a_{n}b_{n}}{\lambda_{1}} = 1 + \frac{(\mul - 2)\log \log n}{4\log n} - \frac{\log\Gamma(\mul/2)}{2\log n}, 
    \end{align*}
    and
    \begin{align*}
        \frac{a_{n}^{2} \clr{x-\log(\abias)}^{2}}{\lambda_{1}}
        =
        \frac{a_{n}^{2}x^2}{\lambda_{1}} + O\plr{\frac{\log\log n}{\log n}},
    \end{align*}
    the exponential factor in \cref{supp-eq:Survival-function-plugin} satisfies
    \begin{align}
        &\exp\plr{-\frac{\blr{a_{n}\clr{x-\log(\abias)} + b_{n}}^2}{\lambda_{1}}}
        \nonumber\\
        &=
        \exp
        \plr{-\frac{b_{n}^{2}}{\lambda_{1}}}
        \exp
        \plr{
            -\frac{2a_{n}b_{n}\clr{x-\log(\abias)}}
            {\lambda_{1}}
        } 
        \exp
        \plr{
            -\frac{a_{n}^{2} \clr{x-\log(\abias)}^{2}}{\lambda_{1}}
        }
        \nonumber\\
        &=
        \exp\plr{-\frac{b_{n}^{2}}{\lambda_{1}}}
        \exp(-x)
        \abias
        \exp
        \plr{
            -\frac{x(\mul -2)\log\log n}{4\log n} - \frac{a_{n}^2x^{2}}{\lambda_{1}}
            +
            O\plr{\frac{\log\log n}{\log n}}
        }
        \nonumber\\
        &=
        \exp
        \plr{
            -\frac{b_{n}^{2}}{\lambda_{1}}
        }
        \exp(-x)
        \abias
        \clr{
            1-\frac{x(\mul -2)\log\log n}{4\log n} - \frac{a_{n}^2x^{2}}{\lambda_{1}}
            + 
            O\plr{\frac{\log\log n}{\log n}}
        },
        \label{supp-eq:Survival-function-exponential-factor}
    \end{align}
    where in the final line a Taylor expansion of $\exp(y)$ is taken around $y=0$, which is justified because $x = O(\log\log n)$. 
    Plugging \cref{supp-eq:Survival-function-polynomial-factor,supp-eq:Survival-function-exponential-factor} into \cref{supp-eq:Survival-function-plugin} yields
    \begin{align}
        &\overline{F}_{H_{1}}
        \plr{u_{n}(x)} 
        \nonumber\\
        &=
        \overline{F}_{H_{1}}
        (\lambda_{1}^{-1/2}b_{n})
        \exp(-x)\abias
        \clr{
            1-\frac{x(\mul -2)\log\log n}{4\log n} - \frac{a_{n}^2x^{2}}{\lambda_{1}}
            + 
            O\plr{\frac{\log\log n}{\log n}}
        }
        \nonumber\\
        &= 
        \frac{1}{n}\exp(-x)\abias
        \clr{
            1-\frac{x(\mul -2)\log\log n}{4\log n} - \frac{a_{n}^2x^{2}}{\lambda_{1}}
            + 
            O\plr{\frac{\log\log n}{\log n}}
        },    
        \label{supp-eq:Survival-function-plugin-2}
    \end{align}
    where the last line is due to 
    \begin{align*}
        \overline{F}_{H_{1}}(\lambda_{1}^{-1/2}b_{n}) 
        =
        \frac{1}{n}
        \clr{
            1+ O\plr{\frac{\log\log n}{\log n}}
        },
    \end{align*} 
    which is a defining property of the sequence $b_{n}$, which is derived in \cref{supp-eq:quantile-equation-residual} in the proof of \cref{supp-lemma:generalized-gamma-normalizing-constants}.
    Plugging \cref{supp-eq:Survival-function-plugin-2} into \cref{supp-eq:CDF-survival-form} and taking logarithms yields
    \begin{align*}
        \log \acdf(x)
        &=
        n\log
        \plr{
            1
            -\frac{1}{n}
            \exp(-x) 
            \clr{
            1-\frac{x(\mul -2)\log\log n}{4\log n} - \frac{a_{n}^2x^{2}}{\lambda_{1}}
            + 
            O\plr{\frac{\log\log n}{\log n}}
        }
        }
        \\
        &=
        -\exp(-x) 
        \clr{
            1-\frac{x(\mul -2)\log\log n}{4\log n} - \frac{a_{n}^2x^{2}}{\lambda_{1}}
            + 
            O\plr{\frac{\log\log n}{\log n}}
        },
    \end{align*}
    where the last line is due to taking the Taylor expansion of $\log(1-y)$ around $y=0$. This is justified because $\frac{1}{n}\exp(-x)$, $\frac{x}{n}\exp(-x)\frac{(\mul -2)\log\log n}{4\log n}$, and $\frac{x^2}{n}\exp(-x)a_{n}$ are all $o(1)$ quantities since $x=O(\log \log n)$ and $\exp(-x) = O(\sqrt{\log n})$. 
    
    Exponentiating $\log \acdf(x)$ yields
    \begin{align}
        \acdf(x)
        &=
        \exp(-\exp(-x))
        \exp\plr{
            \exp(-x)
            \clr{
            \frac{x(\mul -2)\log\log n}{4\log n} 
            +
            \frac{a_{n}^2x^{2}}{\lambda_{1}}
            + 
            O\plr{\frac{\log\log n}{\log n}}
        }
        }
        \label{supp-eq:Fn-expansion-1}\\
        &=
        \exp(-\exp(-x))
        \left\{
            1
            + \frac{x\exp(-x) (\mul-2)\log\log n}{4\log n}
            + \frac{a_{n}^{2}x^{2}\exp(-x)}{\lambda_{1}}
            + O\plr{\frac{\exp(-x)\log\log n}{\log n}}
        \right\},
        \nonumber
    \end{align}
    where in \cref{supp-eq:Fn-expansion-1}, the Taylor expansion of $\exp(y)$ taken around $y=0$ is justified because $\frac{x\exp(-x)(\mul -2)\log\log n}{4\log n}$ and $a_{n}^{2}x^2\exp(-x)/\lambda_{1}$ are all $o(1)$ quantities since $x\in D_{3,n}$.
    Further using the fact that $F_{G}(x) = \exp(-\exp(-x))$ yields
    \begin{align}
        \acdf(x)
        &=
        F_{G}(x) 
        + \frac{x\exp(-x)F_{G}(x)(\mul -2)\log \log n}{4\log n} 
        + \frac{a_{n}^{2}x^2\exp(-x)F_{G}(x)}{\lambda_{1}}
        + O\plr{\frac{\exp(-x)F_{G}(x)\log\log n}{\log n}}.
        \label{supp-eq:Fn-expansion}
    \end{align}
    Since $\exp(-x)F_{G}(x)$, $x\exp(-x)F_{G}(x)$ and $x^{2}\exp(-x)F_{G}(x)$ are all bounded functions on the real line, \cref{supp-eq:Fn-expansion} implies that there exists a finite $n_{1} \in \mathbb{N}$ and a constant $C_{1}>0$ such that for $n \geq n_{1}$, it holds that
    \begin{align}
        \sup_{x\in D_{3,n}}|\acdf(x) - F_{G}(x)|
        \leq
        C_{1}
        {
            \frac{\mul \log\log n}{\log n}
        }.
        \label{supp-eq:KS-bound-D3}
    \end{align}

\noindent
\textbf{(ii)} (Moderately small $x$)
    Suppose $x \in D_{2,n}$. Taking logarithms on \cref{supp-eq:CDF-survival-form}, then applying the inequality $1-x \leq \exp(-x)$ for all $x\in\R$ yields
    \begin{align}
        \acdf(x) 
        &\leq
        \exp
        \plr{
            -n\abias^{-1}\overline{F}_{H_{1}}
            \plr{
                u_{n}(x)
            }
            \blr{1 + O(\{u_{n}(x)\}^{-1})}
        }.
        \label{supp-eq:CDF-exponential-upper-bound}
    \end{align}
    Since $x\in D_{2,n}$, it holds that 
    \begin{align*}
        u_{n}(x)
        &\leq
        \lambda_{1}^{-1/2}
        \blr{
            a_{n}\clr{x_{n,1}-\log(\abias)} + b_{n}
        }
        \\
        &=
        \lambda_{1}^{-1/2}
        \blr{
            a_{n}\clr{-\frac{1}{2}\log\log n + \log \Gamma(\mul/2)} + b_{n}
        }
        \\
        &=
        \sqrt{\log n} +\frac{(\mul -3)\log \log n}{4\sqrt{\log n}}.
    \end{align*}
    Further, by the monotonicity of the survival function $\overline{F}_{H_{1}}$, and the fact that $\abias^{-1} >1$, \cref{supp-eq:CDF-exponential-upper-bound} is further bounded as
    \begin{align}
        \acdf(x) 
        &\leq
        \exp
        \plr{
            -n\overline{F}_{H_{1}}
            \plr{
                \sqrt{\log n} +\frac{(\mul -3)\log \log n}{4\sqrt{\log n}}
            }
        }
        \nonumber\\
        &=
        \exp
        \left(
            -\frac{n}{\Gamma(\mul/2)}
            \clr{\sqrt{\log n} +\frac{(\mul -3)\log \log n}{4\sqrt{\log n}}}^{\mul - 2}
        \right.
        \nonumber\\
        &\qquad
        \times
        \left.
            \exp
            \plr{
                -\clr{\log n 
                + \frac{(\mul-3)\log\log n}{2} 
                + O\plr{\frac{(\log \log n)^{2}}{\log n}}}
            }
        \right)
        \label{supp-eq:CDF-exponential-further-bound-1}\\
        &\leq 
        \exp
        \left(
            -\frac{n}{4\Gamma(\mul/2)}
            \plr{
                {\log n}
            }^{\frac{\mul - 2}{2}}
            \exp
            \plr{
                -
                \clr{
                \log n 
                + \frac{(\mul-3)\log\log n}{2}
                }
            }
        \right)
        \label{supp-eq:CDF-exponential-further-bound-2}\\
        &=
        \exp
        \plr{
            -\frac{(\log n)^{1/2}}{4\Gamma(\mul/2)}
        }
        \nonumber \\
        &\ll (\log n)^{-1}
        \label{supp-eq:CDF-exponential-further-bound-3}
    \end{align}
    for $n$ sufficiently large,
    where \cref{supp-eq:CDF-exponential-further-bound-1} is due to \cref{supp-lemma:survival-function}, and \cref{supp-eq:CDF-exponential-further-bound-2} is due to
    \begin{align*}
        \plr{
            1 +  \frac{(\mul - 3) \log\log n}{4\log n}
        }^{\mul - 2}
        \geq 
        1/2
        \quad
        \text{ and }
        \quad
        \exp\plr{O\plr{\frac{(\log \log n)^{2}}{\log n}}}
        \geq
        1/2
    \end{align*}
    for $n$ sufficiently large. 
    
    On the other hand, since $x\in D_{2,n}$, it holds that
    \begin{align}
        F_{G}(x)
        \leq
        F_{G}(x_{n,1})
        &=
        \exp
        \plr{
            -\exp
            \plr{
                \frac{1}{2}\log\log n 
                -
                \log \Gamma(\mul/2)
                -\log\abias 
            }
        }
        \nonumber\\
        &\leq
        \exp\plr{-C_{2}^\prime\sqrt{\log n}},
        \label{supp-eq:G-upper-bound-D2}
    \end{align}
    where $C_{2}^\prime>0$ is a constant such that 
    \begin{align*}
        C_{2}^\prime
        <
        \exp
        \plr{
            -\log \Gamma(\mul/2)
            -\log \abias
        }
    \end{align*}    
    for all $n$.
    Combining \cref{supp-eq:CDF-exponential-further-bound-3,supp-eq:G-upper-bound-D2} yields
    \begin{align*}
        \sup_{x\in D_{2,n}} |F_{n}(x) - F_{G}(x)| 
        \leq
        \sup_{x\in D_{2,n}} \clr{|F_{n}(x)| + |F_{G}(x)|} 
        \ll
        (\log n)\inv.
    \end{align*} 
    It follows that there exists $n_{2} \in \mathbb{N}$ such that 
    \begin{align}
        \sup_{x\in D_{2,n}} |F_{n}(x) - F_{G}(x)| 
        \leq
        C_{1}
        \frac{\mul \log\log n}{\log n}
        \label{supp-eq:KS-bound-D2}
    \end{align}
    for $n\geq n_{2}$.

\noindent
\textbf{(iii)} (Right tail)
    Suppose $x\in  D_{4,n}$. Then, it holds that
    \begin{align*}
        u_{n}(x)
        >
        \sqrt{\log n}
        + \frac{\mul \log \log n}{4\sqrt{\log n}},
    \end{align*}
    and since $\overline{F}_{H_{1}}$ is non-increasing, there exists a finite $n_{3}\in \mathbb{N}$ such that $n\geq n_{3}$ implies
    \begin{align}
        &\overline{F}_{H_{1}}(u_{n}(x)) 
        \nonumber\\
        &\leq 
        \overline{F}_{H_{1}}
        \plr{
            \sqrt{\log n}
            + \frac{\mul \log \log n}{4\sqrt{\log n}}
        }
        \nonumber\\
        &=
        \frac{1}{\Gamma(\mul/2)}
        \plr{
            \sqrt{\log n}
            + \frac{\mul \log \log n}{4\sqrt{\log n}}
        }^{\mul -2} 
        \exp
        \plr{
            -
            \clr{
                \sqrt{\log n}
                + \frac{\mul \log \log n}{4\sqrt{\log n}}
            }^2
        }
        \clr{
            1 + O\plr{(\log n)^{-1}}
        }
        \nonumber\\
        &\leq
        \frac{2}{\Gamma(\mul/2)}
        \plr{
            2
            \sqrt{\log n}
        }^{\mul -2} 
        \exp
        \plr{
            -
            \clr{
                \log n
                + \frac{\mul \log \log n}{2}
            }
        }
        \label{supp-eq:survival-function-bound-D4-1}\\
        &\leq
        \frac{
        2^{\mul-1}}{\Gamma(\mul/2)n\log n},
        \label{supp-eq:survival-function-bound-D4-2}
    \end{align}
    where \cref{supp-eq:survival-function-bound-D4-1} follows from the inequality
    $
        \frac{\mul \log \log n}{4\sqrt{\log n}} \leq 2\sqrt{\log n}
    $
    and by bounding the terms satisfying $1+o(1)$ by the constant $2$ for $n \geq n_{3}$. Plugging in \cref{supp-eq:survival-function-bound-D4-2} into \cref{supp-eq:CDF-survival-form} yields
    \begin{align}
        1-\acdf(x)
        &= 
        1-
        \clr{
            1
            - \abias^{-1}
            \overline{F}_{H_{1}}
            \plr{
                \sqrt{\log n}
                + \frac{\mul \log \log n}{4\sqrt{\log n}}
            }
            (1+ o(1))
        }
        ^{n}
        \nonumber\\
        &\leq
        1
        -
        \clr{
        1
        -\frac{C_{3}^{\prime}}{n\log n}
        }^{n}
        \nonumber\\
        &=
        1
        -
        \exp
        \plr{
            n\log 
            \plr{
                1-\frac{C_{3}^{\prime}}{n\log n}
            }
        }
        \nonumber\\
        &=
        1
        -
        \exp
        \plr{
            -\frac{C_{3}^{\prime}}{\log n} 
            + 
            O\plr{
                \{n\log^{2}n\}^{-1}
            }
        }
        \label{supp-eq:cdf-bound-D4-1}
        \\
        &=
        \frac{C_{3}^{\prime}}{\log n} 
        + 
        O\plr{
            \{n\log^{2}n\}^{-1}
        }
        ,
        \label{supp-eq:cdf-bound-D4-2}
    \end{align}
    where $C_{3}^{\prime}>0$ is a constant such that $C_{3}^{\prime} \geq 2\abias\inv$ for $n$ large enough, and \cref{supp-eq:cdf-bound-D4-1} follows from taking the Taylor expansion of $\log(1-y)$ around $y=0$, and \cref{supp-eq:cdf-bound-D4-2} follows from taking the Taylor expansion of $1-\exp(y)$ around $y=0$.

    On the other hand, it holds for $x\in D_{4,n}$ that
    \begin{align}
        1-F_{G}(x)
        \leq
        1-G(x_{n,2})
        &=
        1
        -\exp
        \plr{
            -\exp
            \plr{
                -\log\log n
                -\log \Gamma(\mul/2)
                -\log\abias 
            }
        }
        \nonumber\\
        &\leq
        1
        -\exp
        \plr{
            -\frac{C_{4}^{\prime}}{\log n}
        }
        \nonumber\\
        &=
        \frac{C_{4}^{\prime}}{\log n}
        + O((\log n)^{-2}),
        \label{supp-eq:G-bound-D4}
    \end{align}
    where $C_{4}^{\prime}>0$ is a constant such that
    \begin{align*}
        C_{4}^{\prime} \geq \exp(-\log \Gamma(\mul/2) - \log \abias),
    \end{align*}
    for all $n$, and the last line follows from taking the Taylor expansion of $\exp(-y)$ around $y=0$.

    Finally, 
    combining \cref{supp-eq:cdf-bound-D4-2,supp-eq:G-bound-D4} yields that there exists an $n_{3}'\in\mathbb{N}$ such that 
    \begin{align}
        \sup_{x\in D_{4,n}}
        |F_{n}(x) - F_{G}(x)|
        \leq
        \sup_{x\in D_{4,n}}
        \clr{
            |1 - F_{n}(x)| + |1 - F_{G}(x)|
        }
        &\leq 
        (C_{3}^{\prime} + C_{4}^{\prime}) (\log n)^{-1}
        \nonumber\\
        &\leq
        C_{1}
        \frac{\mul \log\log n}{\log n},
        \label{supp-eq:KS-bound-D4}
    \end{align}
    for $n\geq n_{3}'$.
    
\noindent
\textbf{(iv)} (Left tail)
    Suppose $x \in D_{1,n}$.
    Then, it holds that 
    \begin{align*}
        u_{n}(x)
        \leq
        \lambda_{1}^{-1/2}b_{n}\log^{-1/4}n
        =
        \log^{1/4}n
        +
        O
        \plr{
            \frac{\log\log n}{\log^{3/4}n}
        }.
    \end{align*}    
    We use the fact that $X^{2} \overset{\dd}{=} H^{2} + \delta_{n}$ by the definition of $X$ in \cref{\mainlabel{body-lemma:MGF}}, where $\delta_{n}$ is a sum of squared non-i.i.d. Gaussian random variables, hence non-negative with probability one.
    Hence there exist constants $n_{4}^\prime\in \mathbb{N}$ and $C_{4}^\prime>0$ such that
    \begin{align}
        \acdf(x)
        &=
        [\P(\xp \leq u_{n}(x))]^{n}
        \nonumber
        \\
        &\leq 
        [\P(H_{1} \leq u_{n}(x))]^{n}
        \nonumber
        \\
        &\leq 
        [\P(H_{1} \leq \sqrt{2}\log^{1/4} n)]^{n}
        \nonumber
        \\
        &\leq
        [1-C_{4}^\prime \exp(-2\log^{1/2}n)]^{n}
        \label{supp-eq:left-tail-probability-decay-1}
        \\
        &\leq
        \exp(-C_{4}^\prime n\exp(-2\log^{1/2}n))
        \label{supp-eq:left-tail-probability-decay}
    \end{align}
    holds for $n\geq n_{4}^\prime$, where \cref{supp-eq:left-tail-probability-decay-1} follows from \cref{supp-lemma:survival-function}, and \cref{supp-eq:left-tail-probability-decay} is obtained by taking a log followed by an exponential on \cref{supp-eq:left-tail-probability-decay-1}, then a Taylor expansion of $\log(1-y)$ around $y = 0$.
    Furthermore, due to the monotonicity of $F_{G}$, and $x_{n,0} = -2\log n + O(\log^{3/4}n)$, there exists constants $n_{4}'' \in\mathbb{N}$ and $C_{4}'' >0$ such that
    \begin{align*}
        F_{G}(x) 
        \leq
        F_{G}(x_{n,0})
        =
        \exp(-\exp(-x_{n,0}))
        \leq
        \exp\plr{-n}
    \end{align*}
    for all $n\geq n_{4}''$.
    Since
    \begin{align*}
        \max\{
            \exp(-C_{4}^\prime n\exp(-2\log^{1/2}n))
            ,
            \;
            \exp(-n)
        \}
        \ll
        \frac{\mul \log\log n}{\log n},
    \end{align*}
    there exists $n_{4}\in\mathbb{N}$ such that
    \begin{align}
        \sup_{x\in D_{1,n}} |F_{n}(x) - F_{G}(x)| 
        \leq 
        \sup_{x\in D_{1,n}} \clr{|F_{n}(x)| + |F_{G}(x)|}
        \leq
        C_{1}
        \frac{\mul \log\log n}{\log n}
        \label{supp-eq:KS-bound-D1}
    \end{align}
    for $n \geq n_{4}$.

    Finally, combining \cref{supp-eq:KS-bound-D3,supp-eq:KS-bound-D2,supp-eq:KS-bound-D4,supp-eq:KS-bound-D1} yields that 
    \begin{align}
        \sup_{x\in \R} |\acdf(x) - F_{G}(x)|
        \leq
        C_{1} \frac{\mul \log\log n}{\log n}
        \label{supp-eq:KS-bound-final}
    \end{align}
    for all $n\geq n_{0}\coloneqq \max\{n_{1}, n_{2}, n_{3}', n_{4}\}$.

\noindent
\textbf{(v)} (Relating $\acdf$ with $F_{n}$)
    For the ease of presentation, define $\tstata\coloneqq a_{n}^{-1}(\tinorm{\m{EVS}\inv} -b_{n}) + \log \abias$ so that $\gamma_{n} = \tstat - \tstata$ matches the definition in \cref{supp-lemma:first-order-convergence}. Also, define the deterministic sequence $(k_{n})_{n\geq 1}$ such that
    \begin{align*}
        k_{n} 
        \coloneqq
        C_{\gamma}
        \clr{
            {\sqrt{\rlog}}
            \plr{
                \frac{\sigma\sqrt{n}}
                {s_{r}}
                +
                \sqrt{
                    \frac{\mu r}{n}
                }
            }
            +
            \frac{\sigma\sqrt{\mu r n\log n}}
            {s_{r}}
        }
    \end{align*}
    where $C_{\gamma}>0$ is a universal constant such that $\P(|\gamma_{n}| > k_{n}) \leq c n^{-9}$ for a constant $c>0$ per \cref{supp-lemma:first-order-convergence}.
    Since either $\gamma_{n} > -k_{n}$ or $\gamma_{n} \leq -k_{n}$ holds, it follows that for any $x\in \R$,
    \begin{align*}
        \{\tstata +\gamma_{n} \leq x\} 
        \implies
        \{\tstata \leq x + k_{n}\} 
        \cup
        \{\gamma_{n} \leq -k_{n}\}.
    \end{align*}
    On the other hand, it holds that 
    \begin{align*}
        \{\tstata \leq x - k_{n}\}
        \cap
        \{\gamma_{n} \leq k_{n}\}
        \implies
        \{\tstata + \gamma_{n} \leq x\}.
    \end{align*}
    By the preceding two relations, it holds that
    \begin{align*}
        F_{n}(x)
        =
        \P(\tstat \leq x) 
        &\leq
        \P(\tstata \leq x + k_{n}) 
        +
        \P(\gamma_{n} \leq -k_{n})
        \\
        &\leq 
        \acdf(x + k_{n}) 
        +
        cn^{-9},
    \end{align*}
    and
    \begin{align*}
        F_{n}(x)
        &\geq 
        \P(\tstata \leq x - k_{n})
        - 
        \P(\gamma_{n} > k_{n})
        \\
        &\geq
        \acdf(x - k_{n})
        -
        cn^{-9}.
    \end{align*}
    By \cref{supp-eq:KS-bound-final}, it follows that there exists a constant $C>0$ such that
    \begin{align*}
        \sup_{x\in\R} 
        \{F_{n}(x) - F_{G}(x)\}
        &\leq
        \sup_{x\in\R}
        \alr{
            \acdf(x + k_{n}) 
            -
            F_{G}(x + k_{n})
        }
        +
        \sup_{x\in\R}
        \alr{
            F_{G}(x + k_{n}) 
            -
            F_{G}(x)
        }
        +
        cn^{-9}
        \\
        &\leq
        C
        \blr{
            \frac{\mul \log\log n}{\log n}
            + 
            k_{n}
            +
            n^{-9}
        },
    \end{align*}
    and
    \begin{align*}
        \inf_{x\in\R}
        \{F_{n}(x) - F_{G}(x)\}
        &\geq
        \inf_{x\in \R}
        \clr{
            \acdf(x - k_{n}) 
            -
            F_{G}(x - k_{n})
        }
        -
        \sup_{x\in \R}
        \alr{
            F_{G}(x - k_{n}) 
            -
            F_{G}(x)
        }
        -
        cn^{-9}
        \\
        &\geq
        C
        \clr{
            -\frac{\mul \log\log n}{\log n}
            - 
            k_{n}
            -
            n^{-9}
        },
    \end{align*}
    for $n\geq n_{0}$, where we have used the fact that the Gumbel CDF is Lipschitz, which is due to the boundedness of the Gumbel density.
    Crucially, this implies that
    \begin{align*}
        \sup_{x\in\R}
        |F_{n}(x) - F_{G}(x)|
        \leq
        C
        \clr{
            {\sqrt{\rlog}}
            \plr{
                \frac{\sigma\sqrt{n}}
                {s_{r}}
                +
                \sqrt{
                    \frac{\mu r}{n}
                }
            }
            +
            \frac{\sigma\sqrt{\mu r n\log n}}
            {s_{r}}
            +
            \frac{\mul \log\log n}{\log n}
        }
    \end{align*}
    for all $n\geq n_{0}$, which completes the proof of \cref{\mainlabel{body-theorem:CDF-bounds}}. 
\end{proof}

\subsection{Towards proving \cref{\mainlabel{body-proposition:debiased-singular-values}}}
\label{supp-section:debiased-singular-values}

\cref{\mainlabel{body-proposition:debiased-singular-values},\mainlabel{body-theorem:Gumbel-convergence-plugin}} pertain to the Gumbel convergence of the plug-in version of the proposed two-to-infinity norm test statistic.

\subsubsection{Sample singular value locations and fluctuations}
\label{supp-section:biased-singular-values}
In order to compute data-driven estimators for the signal singular values $(s_{j})_{j\in \dpar{r}}$, \cref{supp-lemma:singular-value-displacement-specific} quantifies both the deterministic locations and deviations of the sample singular values $(\widehat{s}_j)_{j\in\dpar{r}}$. 

\begin{lemma}
    \label{supp-lemma:singular-value-displacement-specific}
    Under \cref{\mainlabel{body-assumption:noise},\mainlabel{body-assumption:snr},\mainlabel{body-assumption:plug-in-asms}}, for each $j\in\dpar{r}$, it holds that
    \begin{align}
        \alr{
            \widehat{s}_j 
            -
            \sqrt{
                \frac{(\sigma^2 N+s_{j}^2)(c\sigma^2 N+s_{j}^2)}
                {s_{j}^2}
            }
        } 
        \leq
        \sqrt{48\sigma^2 \log n},
    \end{align}
    with probability exceeding $1-O(n^{-6})$.
\end{lemma}

\begin{proof}
Suppose $n\leq m$, so that $N\coloneqq \max\{n,m\} = m$. Note that if $n > m$, the same argument holds by taking transposes and reversing the roles of $n$ and $m$.
It is convenient to consider the rescaled model $\mh{M}^{(w)} = \m{M}^{(w)} + \m{E}^{(w)}$, where
\begin{align*}
    \mh{M}^{(w)}
    \coloneqq
    \frac{1}{\sigma\sqrt{m}}\mh{M},
    \qquad 
    \m{M}^{(w)}
    \coloneqq 
    \frac{1}{\sigma\sqrt{m}}\m{M},
    \qquad
    \m{E}^{(w)} 
    \coloneqq 
    \frac{1}{\sigma\sqrt{m}}\m{E}.
\end{align*}
By \cref{\mainlabel{body-assumption:plug-in-asms}}, $n/m\to c$ as $n\to\infty$, and the matrix $\m{E}^{(w)}$ has i.i.d.~$\NN(0,1/m)$ entries per \cref{\mainlabel{body-assumption:noise},\mainlabel{body-assumption:plug-in-asms}}. Furthermore, the rescaled signal matrix $\m{M}^{(w)}$ has singular values $s_{j}/(\sigma\sqrt{m})\gg 1$ for all $j\in\dpar{r}$ due to \cref{\mainlabel{body-assumption:snr}} hence are all eventually above any fixed threshold as $n$ grows. The empirical singular value distribution of $\m{E}^{(w)}$, defined as $\widehat{\mu}_{\m{E}^{(w)}}\coloneqq \frac{1}{n}\sum_{i=1}^n \delta_{s_i\plr{\m{E}^{(w)}}}$, converges weakly almost surely to the limiting distribution $\mu_{{(w)}}$ with density
\begin{align}
    \label{supp-eq:MP-density}
    \frac{\dd \mu_{(w)}(x)}{\dd x}
    \coloneqq 
    \frac{\sqrt{((1+\sqrt{c})-x^2)(x^2-(1-\sqrt{c}))}}
    {\pi cx}
    \I{[1-\sqrt{c},1+\sqrt{c}]}(x),
\end{align}
which is equal to $2xf_{\text{MP}}(x^2)$, where $f_{\text{MP}}$ denotes the Marchenko--Pastur density \citep{marchenko_distribution_1967}. By \citet[Theorem 2.8]{benaych-georges_singular_2012}, for each $j\in\dpar{r}$ and for $\phi_{\mu_{(w)}}(z)\coloneqq \int_{\R_{\geq 0}}\frac{z}{z^2-t^2}\dd\mu_{(w)}(t)$, the solution to the deterministic equation 
\begin{align}
    \label{supp-eq:deterministic-equivalent}
    1-\frac{s_{j}^2}{\sigma^{2}m}
    \phi_{\mu_{(w)}}(z)
    \left(c\phi_{\mu_{(w)}}(z) + \frac{1-c}{z}\right) 
    = 
    0
\end{align}
is equal to the deterministic location of the $j$-th singular value of $\mh{M}^{(w)}$ for $j\in\dpar{r}$. Furthermore, since $\m{E}^{(w)}$ has i.i.d.~$\NN(0,1/m)$ entries, the solution of \cref{supp-eq:deterministic-equivalent} is well understood (see for example, \citealp[Section 3.1]{benaych-georges_singular_2012}), namely
\begin{align*}
    \theta^{(w)}_{j} 
    = 
    \sqrt{
        \frac{(\sigma^{2}m+s_{j}^{2})(c\sigma^2 m+s_{j}^{2})}
        {\sigma^{2} m s_{j}^{2}}
    },
    \quad j\in\dpar{r}.
\end{align*} 
The singular values of $\m{E}$ are simply those of  $\m{E}^{(w)}$ multiplied by $\sigma\sqrt{m}$, so the limiting distribution of $\widehat{\mu}_{\m{E}} \coloneqq \frac{1}{n}\sum_{i=1}^n \delta_{s_i(\m{E})}$, denoted by $\mu$, satisfies $\dd \mu(\sigma\sqrt{m} x) = \dd \mu_{(w)}(x)$. Thus, the function $\phi_{\mu}(z) \coloneqq \int \frac{z}{z^2-t^2} \dd\mu(t)$ and $\phi_{\mu_{(w)}}(z)$ are related in the manner
\begin{align*}
    \phi_{\mu_{(w)}}(z) 
    &=
    \int_{\R_{\geq 0}}
    \frac{z}{z^2-t^2} 
    \dd\mu_{(w)}(t)
    \overset{s = \sigma\sqrt{m} t}{=}
    \int_{\R_{\geq 0}}
    \frac{z}{z^2-s^2/(\sigma^2 m)}
    \dd\mu(s) 
    \overset{w = \sigma\sqrt{m} z}{=}
    \sigma\sqrt{m} 
    \int_{\R_{\geq 0}}
    \frac{w}{w^2-s^2}
    \dd\mu(s) \\
    &= 
    \sigma\sqrt{m} \phi_{\mu}(\sigma \sqrt{m}z).
\end{align*}
Hence, if $z$ is a solution to \cref{supp-eq:deterministic-equivalent}, then
\begin{align}
    \label{supp-eq:deterministic-equivalent-full}
    1-{s_{j}^2}\phi_{\mu}(\sigma\sqrt{m}z)\left(c\phi_{\mu}(\sigma\sqrt{m}z) + \frac{1-c}{\sigma\sqrt{m} z}\right) 
    =
    0,
\end{align}
which implies that $\sigma\sqrt{m} z$ is a solution of \cref{supp-eq:deterministic-equivalent-full}, the deterministic equation that gives rise to the locations of the top-$r$ singular values of $\mh{M}$. That is, the deterministic location of $\widehat{s}_{j}$ for $j\in\dpar{r}$ is
\begin{align}
    \label{supp-eq:singular-value-displacement}
    \theta_{j}
    \coloneqq
    \sqrt{
        \frac{(\sigma^2 m+s_{j}^2)(c\sigma^2 m+s_{j}^2)}{s_{j}^2}
        }.
\end{align} 

Next, we derive the deviations of $(\widehat{s}_j)_{j\in\dpar{r}}$ around their deterministic locations $(\theta_j)_{j\in\dpar{r}}$. Towards this end, we write the symmetric dilations of $\mh{M}$, $\m{M}$, and $\m{E}$ as    
\begin{align*}
    \hc{M} 
    = 
    \begin{bmatrix}
    \m{0} & \mh{M}\\ \mh{M}\T & \m{0}
    \end{bmatrix},
    \qquad 
    {\mathcal{M}} 
    = 
    \begin{bmatrix}
    \m{0} & \m{M}\\ \m{M}\T & \m{0}
    \end{bmatrix},
    \qquad 
    \mathcal{E} 
    = 
    \begin{bmatrix}
    \m{0} & \m{E}\\ \m{E}\T & \m{0}
    \end{bmatrix}.
\end{align*}
Per convention, write
\begin{align*}
    \mathcal{U}
    = \frac{1}{\sqrt{2}}\begin{bmatrix}
    \m{U} & \m{U}\\ \m{V} & -\m{V}
    \end{bmatrix}\in\R^{(n+m)\times 2r},
    \qquad
    \mathcal{D}
    = \begin{bmatrix}
    \m{S} & \m{0}\\ \m{0} & -\m{S}
    \end{bmatrix}\in\R^{2r\times 2r},
\end{align*}
from which it is clear that $\mc{U}$ has orthonormal columns $\mc{U}_{1},\dots,\mc{U}_{2r}$, hence $\mathcal{M} = \mathcal{U}\mathcal{D}\mathcal{U}\T$ is an eigenvalue decomposition. It follows that $\hc{M} = \mathcal{M} +\mathcal{E}$ is a symmetric system, and the dilated signal matrix $\mathcal{M}$ has $2r$ nonzero eigenvalues, where $r$ are equal to the singular values of $\m{M}$, and the remaining $r$ are the negatives of the singular values of $\m{M}$. Since all the nonzero singular values of $\mathcal{M}$ are distinct by \cref{\mainlabel{body-assumption:plug-in-asms}}, $\alpha_n \coloneqq \tnorm{\E[\mathcal{E}^2]}^{1/2} = \sigma\max\{\sqrt{n},\sqrt{m}\}$, and $s_r \gg \sigma\sqrt{\zeta_{r,n} \mu n}$ by \cref{\mainlabel{body-assumption:snr}}, it follows from \citet[Theorem 1]{fan_asymptotic_2022} and \cref{\mainlabel{body-assumption:snr}} that
\begin{align*}
    \widehat{s}_j 
    - 
    \theta_j 
    = 
    \mathcal{U}_j\T\mathcal{E}\mathcal{U}_j
    +
    \eta_{n,j},
\end{align*}
where $\eta_{n,j} = O(\alpha_n s_r^{-1}) = o((\sigma\log n)\inv)$ with probability exceeding $1-O(n^{-6})$.
Writing $\mathcal{U}_j = [\mathcal{U}_{j,(1)}\T, \mathcal{U}_{j,(2)}\T]\T$, where $\mathcal{U}_{j,(1)}\in\R^{n}$ and $\mathcal{U}_{j,(2)}\in\R^{m}$, it holds for all $j\in\dpar{r}$ that
\begin{align*}
    \mathcal{U}_j\T\mathcal{E}\mathcal{U}_j
    =
    2\mathcal{U}_{j,(1)}\T\m{E}\mathcal{U}_{j,(2)} 
    =
    2\sum_{i=1}^{n}\sum_{k=1}^{m}
    (\mathcal{U}_{j,(1)})_{i}
    (\mathcal{U}_{j,(2)})_{k}
    E_{i,k} 
    \sim
    \NN(0, \sigma^2),
\end{align*}
where we have used the fact that $\m{E}$ has i.i.d. $\NN(0,\sigma^{2})$ entries by \cref{\mainlabel{body-assumption:plug-in-asms}}, and 
$
    \sum_{i=1}^{n}\sum_{k=1}^{m}
    (\mathcal{U}_{j,(1)})_{i}^{2}
    (\mathcal{U}_{j,(2)})_{k}^{2} 
    =
    \sum_{i=1}^{n}
    (\mathcal{U}_{j,(1)})_{i}^{2}
    \sum_{k=1}^{m}
    (\mathcal{U}_{j,(2)})_{k}^{2}
    = 1/4
$.
Next, for $n\geq 1$ and $j\in\dpar{r}$, define the events $\mathcal{A}_{n,j} \coloneqq \{|\widehat{s}_j - \theta_j|\geq \sqrt{12\sigma^2 \log n}\}$ and $\mathcal{B}_{n,j} \coloneqq \{|\eta_{n,j}| \leq (\sigma\log n)^{-1}\}$. Applying the Gaussian tail bound yields
\begin{align*}
    \P(\mathcal{A}_{n,j} \mid \mathcal{B}_{n,j}) 
    &\le
    2\P\plr{\mathcal{U}_j\T\mathcal{E}\mathcal{U}_j \ge \sqrt{12\sigma^2 \log n} - |\eta_{n,j}| \middle\vert \mathcal{B}_{n,j}}
    \\
    &\le 2\P\plr{\mathcal{U}_j\T\mathcal{E}\mathcal{U}_j \ge \sqrt{12\sigma^2 \log n} - (\sigma\log n)^{-1} }
    \\
    &\le 
    2\exp\plr{-6\log n}
    \underbrace{
    \exp
    \plr{
        \frac{\sqrt{12\sigma^2\log n}}{\sigma^2}(\sigma\log n)^{-1}
    }
    \exp
    \plr{
        -\frac{(\sigma\log n)^{-2}}{2\sigma^2}
    }
    }_{
    \to 1 \; \text{ as } n\to\infty
    }
    \\
    &\le
    3n^{-6},
\end{align*}
for $n$ large enough. Since $\P(\mathcal{B}_{n,j})\ge 1-O(n^{-6})$, it follows that for all $j\in\dpar{r}$,
\begin{align*}
    \P(\mathcal{A}_{n,j})
    =
    \P(\mathcal{A}_{n,j}\mid \mathcal{B}_{n,j})\P(\mathcal{B}_{n,j}) 
    +
    \P(\mathcal{A}_{n,j}\mid \mathcal{B}_{n,j}^c)\P(\mathcal{B}_{n,j}^c) 
    =
    1 - O(n^{-6}),
\end{align*}
which proves \cref{supp-lemma:singular-value-displacement-specific}.
\end{proof}

\subsubsection{Proof of \cref{\mainlabel{body-proposition:debiased-singular-values}}}
\label{supp-pf:debiased-singular-values}
\begin{proof}
For each $j\in\dpar{r}$, define the shrinkage-type estimator $\widetilde{s}_j$ as the positive real-valued solution to the equation
\begin{equation*}
    \widehat{s}_j 
    =
    \sqrt{\frac{(\sigma^2 N+\widetilde{s}_j^2)(c\sigma^2 N+\widetilde{s}_j^2)}
    {\widetilde{s}_j^2}}.
\end{equation*}
Solving the above for $\widetilde{s}_j$ yields 
\begin{align}
    \label{supp-eq:corrected-singular-values}
    \widetilde{s}_j 
    = 
    \frac{1}{\sqrt{2}}
    \left(\widehat{s}_j^2 - (1+c)\sigma^2 N 
    +
    \sqrt{[(1+c)\sigma^2 N -\widehat{s}_j^2]^2 - 4c\sigma^4 N^2}\right)^{1/2}.
\end{align}
Since $s_{j}\gg \sqrt{\mu \rlog}(\sigma \sqrt{n})$ by \cref{\mainlabel{body-assumption:snr}} and $\tnorm{\m{E}}\lesssim \sigma\sqrt{n+m}$ with probability exceeding $1-O(n^{-7})$ \citep[Section 3.2]{chen_spectral_2021}, it holds that $\widehat{s}_j^2 - (1+c)\sigma^2 N \geq 0$ and $[(1+c)\sigma^2 N -\widehat{s}_j^2]^2 - 4c\sigma^4 N^2 \geq 0$ with probability exceeding $1-O(n^{-7})$ due to Weyl's inequality for singular values. Thus, the positive real-valued estimators $(\widetilde{s}_{j})_{j\in\dpar{r}}$ exist with probability exceeding $1-O(n^{-7})$.

To derive a bound for $|\widetilde{s}_{j}/s_{j}-1|$, we have by \cref{supp-lemma:singular-value-displacement-specific} that
\begin{align}
    \widehat{s}_j 
    =
    \sqrt{\frac{(\sigma^2 N+\widetilde{s}_j^2)(c\sigma^2 N+\widetilde{s}_j^2)}{\widetilde{s}_j^2}} 
    =
    \sqrt{\frac{(\sigma^2 N+s_{j}^2)(c\sigma^2 N+s_{j}^2)}{s_{j}^2}} 
    +
    O(\sigma\log^{1/2} n)
    \label{supp-eq:stil-s-relation}
\end{align}
with probability greater than $1-O(n^{-6})$.
Squaring and expanding each side of \cref{supp-eq:stil-s-relation} yields
\begin{align*}
    \widetilde{s}_{j}^2 
    + 
    \frac{c\sigma^4 N^2}{\widetilde{s}_{j}^{2}}
    =
    s_{j}^2 + \frac{c\sigma^4 N^2}{{s}_{j}^{2}}  
    +
    O\left(\sigma s_{j}\log^{1/2} n \right),
\end{align*}
where we have used that $\sigma^2 N\ll s_{j}^2$ and $\sigma^2 N\ll \widetilde{s}_j^2$, hence $\sqrt{{(\sigma^2 N+s_{j}^2)(c\sigma^2 N+s_{j}^2)}/{s_{j}^4}}=O(1)$. It follows that
\begin{align}
    \frac{\widetilde{s}_{j}^2}{s_{j}^2} 
    =
    1
    +
    \frac{c\sigma^4 N^2}{\widetilde{s}_{j}^2 s_{j}^2}  
    +
    \frac{c\sigma^4 N^2}{s_{j}^{4}}  
    +
    O\left(\frac{\sigma\log^{1/2} n}{s_{j}}\right).
    \label{supp-eq:stil-s-squared-ratio}
\end{align}
Under \cref{\mainlabel{body-assumption:noise},\mainlabel{body-assumption:matrix-size}}, we have
\begin{align}
    \label{supp-eq:corrected-singular-value-ratio}
    \frac{\widetilde{s}_{j}}{s_{j}} 
    =
    1
    +
    O\left(\frac{\sigma^4 n^2}{s_{j}^4}+\frac{\sigma\log^{1/2} n}{s_{j}}\right),
\end{align}
with probability greater than $1-O(n^{-6})$. This completes the proof of \cref{\mainlabel{body-proposition:debiased-singular-values}}.
\end{proof}

\begin{remark}
    By \cref{supp-eq:stil-s-relation} and \cref{\mainlabel{body-assumption:plug-in-asms}}, the uncorrected sample singular values satisfy 
    \begin{align*}
        \frac{\widehat{s}_{j}}{s_{j}} 
        &=
        \sqrt{\plr{\frac{\sigma^{2} N}{s_{j}^2} + 1}\plr{\frac{c\sigma^{2} N}{s_{j}^2} + 1}} 
        +
        O\plr{\frac{\sigma\log^{1/2} n}{s_{j}}}
        \\
        &= 1 + O\plr{\frac{\sigma^2 N}{s_{j}^2}},
    \end{align*} 
    with probability exceeding $1-O(n^{-6})$ for all $j\in\dpar{r}$. This shows that although each ratio $\widehat{s}_{j}/s_{j}$ tends to one in probability as $n\to\infty$ under the diverging spikes condition in \cref{\mainlabel{body-assumption:snr}}, the corrected ratios $\widetilde{s}_{j}/s_{j}$ obtained in \cref{supp-lemma:singular-value-displacement-specific} converge to one in probability potentially at a much faster rate. This is demonstrated in the finite-$n$ simulations in \cref{\mainlabel{body-fig:convergence-in-distribution}}, where the bias of the uncorrected test statistic is especially visible under weak signals.
\end{remark}

\subsection{Proof of \cref{\mainlabel{body-theorem:Gumbel-convergence-plugin}}}
\label{supp-pf:Gumbel-convergence-plugin}

\begin{proof}
    All the following statements are either deterministic or hold with probability greater than $1-O(n^{-6})$ unless otherwise specified. By \cref{\mainlabel{body-lemma:MGF}}, we have $\lambda_j \coloneqq 2\sigma^2/s^2_{r-j+1}$ for all $j\in\dpar{r}$. By \cref{\mainlabel{body-proposition:debiased-singular-values}}, we have 
    \begin{align}
        \frac{\lambda_j}{\widetilde{\lambda}_j} 
        =
        \frac{\widetilde{s}_{r-j+1}^2}{{s}_{r-j+1}^2}
        = 
        1+ O\left(\frac{\sigma^4 n^2}{s_{r-j+1}^4}+\frac{\sigma\log^{1/2} n}{s_{r-j+1}}\right). 
        \label{supp-eq:lambda-ratio}
    \end{align}
    Next, since $\widetilde{a}_n = a_n\widetilde{\lambda}^{1/2}_{1}/\lambda_{1}^{1/2}$ and $\widetilde{b}_n = b_n\widetilde{\lambda}_{1}^{1/2}/\lambda_{1}^{1/2}$, we have
    \begin{align}
        \tplug 
        &\coloneqq 
        \widetilde{a}_n^{-1}\plr{\tinorm{\mh{U}\RU -\m{U}}-\widetilde{b}_n} + \log(\abplug)
        \nonumber\\
        &= 
        \frac{\lambda_{1}^{1/2}}{\widetilde\lambda_{1}^{1/2}}a_n^{-1}\plr{\tinorm{\mh{U}\RU-\m{U}}-b_n} +\plr{\frac{\lambda_{1}^{1/2}}{\widetilde\lambda_{1}^{1/2}}-1}a_n\inv b_n  +\log\blr{\prod_{j=2}^r\plr{1-\frac{\lambda_j}{\lambda_{1}}\frac{\widetilde{\lambda}_j/\lambda_j}{\widetilde{\lambda}_{1}/\lambda_{1}}}^{1/2}}
        \nonumber\\
        &= 
        \underbrace{a_n^{-1}\plr{\tinorm{\mh{U}\RU-\m{U}}-b_n}+\log\blr{\prod_{j=2}^r\plr{1-\frac{\lambda_j}{\lambda_{1}}}^{1/2}}}_{= \tstat} 
        + 
        \underbrace{\plr{\frac{\lambda_{1}^{1/2}}{\widetilde\lambda_{1}^{1/2}}-1}a_n\inv b_n}_{\eqqcolon\text{(a)}}
        \nonumber\\
        &
        \qquad +\underbrace{O\clr{\left(\frac{\sigma^4 n^2}{s_{r-j+1}^4}+\frac{\sigma\log^{1/2} n}{s_{r-j+1}}\right)a_n^{-1}\plr{\tinorm{\mh{U}\RU-\m{U}}-b_n}}}_{\eqqcolon\text{(b)}}
        \nonumber\\
        &
        \qquad + \underbrace{\frac{1}{2}\sum_{j=2}^r \log\blr{{1+O\clr{\left(\frac{\sigma^4 n^2}{s_{r-j+1}^4}+\frac{\sigma\log^{1/2} n}{s_{r-j+1}}\right)\frac{\lambda_j}{\lambda_{1}-\lambda_j}}}}}_{\eqqcolon\text{(c)}}.
        \label{supp-eq:plug-in-convergence}
    \end{align}
    
    \noindent
    \textbf{(a)} By \cref{\mainlabel{body-theorem:Gumbel-convergence}}, $a_n\inv b_n = O(\log n)$. Hence, by applying \cref{supp-eq:lambda-ratio}, term (a) satisfies
    \begin{align}
        \label{supp-eq:plug-in-convergence-a}
        \text{(a)} 
        = 
        O\left(\frac{\sigma^4 n^2\log n}{s_{r}^4}+\frac{\log^{3/2} n}{s_{r}}\right).
    \end{align}

    \noindent
    \textbf{(b)} By \cref{\mainlabel{body-lemma:first-order-approximation}},  
    \begin{align*}
        \tinorm{\mh{U}\RU-\m{U}}
        \lesssim
        a_n \log n +b_n 
        = 
        O\plr{\frac{\sigma\sqrt{\log n}}{s_r}}
    \end{align*}
        with probability greater than $1-O(n^{-9})$. Thus, since $a_n\inv = O(s_r\sqrt{\log n}/\sigma)$, term (b) satisfies
    \begin{align}
        \label{supp-eq:plug-in-convergence-b}
        \text{(b)} 
        =
        O\left(\frac{\sigma^4 n^2\log n}{s_{r}^4}
        +
        \frac{\log^{3/2} n}{s_{r}}\right)
    \end{align}
    with probability greater than $1-O(n^{-9})$.

    \noindent
    \textbf{(c)} Using the Taylor series expansion of the function $\log(1+x)$ about $x=0$, we have $\log(1+x) = x + O(x^2)$. Furthermore, by \cref{\mainlabel{body-assumption:snr}}, we have that $\lambda_j/(\lambda_{1}-\lambda_j) = {s_r^2}/\plr{s_{r-j+1}^2-s_r^2} = O(1)$. Thus, term (c) satisfies
    \begin{align}
    \label{supp-eq:plug-in-convergence-c}
        \text{(c)}
        = 
        O\left(\frac{r\sigma^4 n^2}{s_{r}^4}+\frac{r\log^{1/2} n}{s_{r}}\right),
    \end{align}
    which is smaller in order than (a) and (b) under \cref{\mainlabel{body-assumption:snr}}.

    Combining \cref{supp-eq:plug-in-convergence,supp-eq:plug-in-convergence-a,supp-eq:plug-in-convergence-b,supp-eq:plug-in-convergence-c}, we have
    \begin{align*}
        \tplug 
        = 
        \tstat
        +
        O\left(\frac{\sigma^4 n^2\log n}{s_{r}^4}+\frac{\log^{3/2} n}{s_{r}}\right)
    \end{align*}
    with probability greater than $1-O(n^{-5})$.
    Hence, under \cref{\mainlabel{body-assumption:noise},\mainlabel{body-assumption:snr}}, $\tplug \rightsquigarrow \gumbelrv$
    as $n\to\infty$ by Slutsky's lemma. This completes the proof of \cref{\mainlabel{body-theorem:Gumbel-convergence-plugin}}.
\end{proof}

\subsection{Towards proving \cref{\mainlabel{body-theorem:power-analysis}}}
\label{supp-pf:power-analysis}

The key to the power analysis in \cref{\mainlabel{body-theorem:power-analysis}} is understanding how the term $\tinorm{\mh{U}\m{R}({\mh{U},\m{U}_0})-\m{U}_0}$ changes based on the magnitude of the discrepancy $d_{n}\coloneqq \tinorm{\m{U}_{1}-\m{U}_0}$ under $\hypa$ in \cref{\mainlabel{body-eq:T1}}. The remainder of this section assumes that the alternative hypothesis in \cref{\mainlabel{body-eq:T1}} holds, so that $\m{M} = \m{U}_{1}\m{S}\m{V}\T$.

\subsubsection{Technical lemmas}

We first introduce and prove supporting lemmas which will be used to prove \cref{\mainlabel{body-theorem:power-analysis}}. First, \cref{supp-lemma:switching-orthogonal-alignment-matrices} shows that the map $\sgn: \R^{n\times r}\times \R^{n\times r} \to \R^{r\times r}$ is Lipschitz in its arguments. More specifically, holding one argument fixed, the spectral norm difference of the sign functions when the other argument changes is bounded above by the spectral norm difference of the changed arguments. 
\begin{lemma}[Bounding the orthogonal alignment matrix difference]
    \label{supp-lemma:switching-orthogonal-alignment-matrices}
    Under \cref{\mainlabel{body-assumption:noise},\mainlabel{body-assumption:matrix-size},\mainlabel{body-assumption:snr},\mainlabel{body-assumption:aligned-alternative}},
    \begin{align*}
        \tinorm{\mh{U}\{\Rs{\mh{U}}{\m{U}_0} - \Rs{\mh{U}}{\m{U}_{1}}\}}
        \lesssim
        \plr{\frac{\mu r}{n}}^{1/2}\tnorm{\m{U}_0-\m{U}_{1}},
    \end{align*}
    with probability exceeding $1-O(n^{-9})$.
\end{lemma}
\begin{proof}
    By definition, the matrices $\Rs{\mh{U}}{\m{U}_0}$ and $\Rs{\mh{U}}{\m{U}_{1}}$ are the orthogonal factors of the polar decompositions of the matrices $\mh{U}\T\m{U}_0$ and $\mh{U}\T\m{U}_{1}$, respectively. By properties of the two-to-infinity norm along with perturbation results for the polar factorization of matrices \citep[Corollary 4.1]{bhatia_matrix_1994}, it holds that
    \begin{align}
        \tinorm{\mh{U}\{\Rs{\mh{U}}{\m{U}_{0}} - \Rs{\mh{U}}{\m{U}_{1}}\}}
        &
        \leq \tinorm{\mh{U}}\tnorm{\Rs{\mh{U}}{\m{U}_{0}} - \Rs{\mh{U}}{\m{U}_{1}}} 
        \nonumber\\
        &
        \leq \frac{\tinorm{\mh{U}}}{s_{r}(\mh{U}\T\m{U}_{1})}\tnorm{\m{U}_{0}-\m{U}_{1}}
        \label{supp-eq:Lipschitz-rotations-a}\\ 
        &\lesssim 
        \plr{\frac{\mu r}{n}}^{1/2}\tnorm{\m{U}_{0}-\m{U}_{1}},
        \label{supp-eq:Lipschitz-rotations-b}
    \end{align}
    where the first two inequalities hold with probability one, and the last asymptotic relation holds with probability exceeding $1-O(n^{-6})$. In more detail, \cref{supp-eq:Lipschitz-rotations-a} uses \citep[Corollary 4.1]{bhatia_matrix_1994} along with the submultiplicativity of the spectral norm. \cref{supp-eq:Lipschitz-rotations-b} results from the following two properties. First, by \cref{\mainlabel{body-lemma:first-order-approximation},\mainlabel{body-assumption:noise},\mainlabel{body-assumption:matrix-size},\mainlabel{body-assumption:snr}}, the term $\tinorm{\mh{U}}$ in \cref{supp-eq:Lipschitz-rotations-a} satisfies
    \begin{align}
        \tinorm{\mh{U}} 
        =
        \tinorm{\mh{U}\Rs{\mh{U}}{\m{U}_{1}}} 
        &=
        \tinorm{\m{U}_{1} +(\mh{U}\Rs{\mh{U}}{\m{U}_1}-\m{U}_1)}
        \nonumber \\
        &\leq
        \tinorm{\m{U}_{1}} + \tinorm{\mh{U}\Rs{\mh{U}}{\m{U}_1}-\m{U}_1} 
        \lesssim
        \plr{\frac{\mu r}{n}}^{1/2}
        \label{supp-eq:Uhat-delocalization}
    \end{align}
    with probability exceeding $1-O(n^{-9})$, where \cref{supp-eq:Uhat-delocalization} uses \cref{\mainlabel{body-eq:perturbation-bound}} in \cref{\mainlabel{body-lemma:first-order-approximation}}.
    On the other hand, by Wedin's $\sin\Theta$ theorem and \cref{\mainlabel{body-assumption:noise},\mainlabel{body-assumption:snr}}, it holds that
    \begin{align*}
        \sqrt{1-s_{r}(\mh{U}\T\m{U}_{1})^2}
        =
        \tnorm{\sin\Theta(\mh{U},\m{U}_{1})}
        =
        o(1)
    \end{align*}
    with probability exceeding $1-O(n^{-9})$. 
    Hence, the denominator of the fraction in \cref{supp-eq:Lipschitz-rotations-a} satisfies $s_{r}(\mh{U}\T\m{U}_{1}) \to 1$ with probability exceeding $1-O(n^{-9})$.
\end{proof}

Next, \cref{supp-lemma:misaligned-hypotheses} asserts that the two-to-infinity norm difference of a pair of hypothesis matrices $\m{U}_{1}$ and $\m{U}_{0}$ misaligned in a specific manner must be of asymptotic order greater than or equal to the properly aligned discrepancy $d_n$.

\begin{lemma}[Lower bound for two-to-infinity norm difference of misaligned hypotheses]
    \label{supp-lemma:misaligned-hypotheses}
    Under \cref{\mainlabel{body-assumption:noise},\mainlabel{body-assumption:matrix-size},\mainlabel{body-assumption:delocalization},\mainlabel{body-assumption:snr},\mainlabel{body-assumption:aligned-alternative}}, 
    \begin{align*}
        \tinorm{\m{U}_{1}\Rs{\m{U}_{1}}{\mh{U}}\Rs{\mh{U}}{\m{U}_{0}} - \m{U}_{0}} 
        \gtrsim
        (\mu r)^{-1/2} d_{n}
    \end{align*}
    with probability exceeding $1-O(n^{-6})$.
\end{lemma}

\begin{proof}
    The proof consists of two parts, depending on the delocalization of $\m{U}_{1} - \m{U}_{0}$. 
    
    \noindent
    \textbf{(i)} (Delocalized case). Suppose that $\tinorm{\m{U}_{1} - \m{U}_{0}} \lesssim \sqrt{\mu r/n} \fnorm{\m{U}_{1} - \m{U}_{0}}$. Then, defining the two-to-infinity norm optimal Procrustes problem solution $\m{R}_{2,\infty} \coloneqq \argmin_{\m{R}\in {O}_{r}} \tinorm{\m{U}_{1}\m{R} - \m{U}_{0}}$, where $O_{r}$ denotes the set of $r\times r$ orthogonal matrices, it follows that
    \begin{align*}
        d_n
        =
        \tinorm{\m{U}_{1} - \m{U}_{0}} 
        \lesssim
        \sqrt{\frac{\mu r}{n}} \fnorm{\m{U}_{1} - \m{U}_{0}} 
        \leq
        \sqrt{\frac{\mu r}{n}} \fnorm{\m{U}_{1}\m{R}_{2,\infty} - \m{U}_{0}}
        \leq \sqrt{\mu r} \tinorm{\m{U}_{1}\m{R}_{2,\infty} - \m{U}_{0}},
    \end{align*}
    where we have used the fact that $\m{I}_r$ is the Frobenius norm optimal Procrustes problem solution which is a consequence of \cref{\mainlabel{body-assumption:aligned-alternative}}. This result implies that
    \begin{align}
        \tinorm{\m{U}_{1}\Rs{\m{U}_{1}}{\mh{U}}\Rs{\mh{U}}{\m{U}_{0}} - \m{U}_{0}} 
        \geq
        \tinorm{\m{U}_{1}\m{R}_{2,\infty} - \m{U}_{0}}
        \gtrsim ({\mu r})^{-1/2} d_n.
        \label{supp-eq:misaligned-delocalized}
    \end{align}
    
    \noindent
    \textbf{(ii)} (Localized case). Suppose that $\tinorm{\m{U}_{1} - \m{U}_{0}} \gg \sqrt{\mu r/n} \fnorm{\m{U}_{1} - \m{U}_{0}}$. By the reverse triangle inequality, and the orthogonal invariance of the spectral norm, it holds with probability one that
    \begin{align}
        \tinorm{\m{U}_{1}\Rs{\m{U}_{1}}{\mh{U}}\Rs{\mh{U}}{\m{U}_{0}} - \m{U}_{0}} 
        \ge
        \tinorm{\m{U}_{1} - \m{U}_{0}} 
        -
        \tinorm{
            \m{U}_{1}
            \plr{
                \Rs{\m{U}_{1}}{\mh{U}}\Rs{\mh{U}}{\m{U}_{0}}
                -
                \m{I}_{r}
            }
        }.
        \label{supp-eq:misaligned-localized-1}
    \end{align}
    By \cref{supp-lemma:switching-orthogonal-alignment-matrices}, it holds with probability exceeding $1-O(n^{-9})$ that
    \begin{align*}
        \tinorm{
            \m{U}_{1}
            \{
                \Rs{\m{U}_{1}}{\mh{U}}\Rs{\mh{U}}{\m{U}_{0}} - \m{I}_{r}
            \}
        }
        &\le
        \tinorm{\m{U}_{1}}\tnorm{\Rs{\mh{U}}{\m{U}_{0}} - \Rs{\mh{U}}{\m{U}_{1}}}
        \\
        &\lesssim
        \sqrt{\frac{\mu r}{n}} \tnorm{\m{U}_{1} - \m{U}_0}
        \\
        &\le
        \sqrt{\frac{\mu r}{n}} \fnorm{\m{U}_{1} - \m{U}_0}
        \\
        &\ll 
        \tinorm{\m{U}_{1} - \m{U}_0},
    \end{align*}
    where the last line is by hypothesis. Thus, plugging back into \cref{supp-eq:misaligned-localized-1} yields that for a sufficiently large $n$,
    \begin{align}
        \tinorm{\m{U}_{1}\Rs{\m{U}_{1}}{\mh{U}}\Rs{\mh{U}}{\m{U}_{0}} - \m{U}_{0}} 
        \ge
        d_n/2
        \label{supp-eq:misaligned-localized-2}
    \end{align}
    holds with probability exceeding $1-O(n^{-9})$. Finally, combining \cref{supp-eq:misaligned-delocalized,supp-eq:misaligned-localized-2} finishes the proof of \cref{supp-lemma:misaligned-hypotheses}.
\end{proof}

\begin{remark}
 The proof of \cref{supp-lemma:misaligned-hypotheses} is divided into two cases due to the intricacies of the $\ell_{2,\infty}$ geometry, where the value of $\min_{\m{R}\in {O}_{r}} \tinorm{\m{U}_{1}\m{R} - \m{U}_{0}}$ may potentially be smaller in order than $\tinorm{\m{U}_{1} - \m{U}_{0}}$ when $\m{U}_{1} - \m{U}_{0}$ is localized in a few rows. In words, if $\m{U}_{1}-\m{U}_{0}$ is sufficiently delocalized, then $d_{n}$, which is aligned by the Frobenius norm optimal Procrustes solution, is also nearly optimal in the two-to-infinity norm, hence the desired relation directly holds. On the other hand, if $\m{U}_{1}-\m{U}_{0}$ is highly localized, the error induced by switching the orthogonal alignment matrices (\cref{supp-lemma:switching-orthogonal-alignment-matrices}) is strictly dominated in order by $d_{n}$, hence the reverse triangle inequality in \cref{supp-eq:misaligned-localized-1} reveals the desired relationship.
\end{remark}

\subsubsection{Proof of \cref{\mainlabel{body-theorem:power-analysis}}}

\begin{proof}
    \noindent
    \textbf{(i)} (Inconsistent regime: $d_{n} \ll \widetilde{a}_n(\mu r)^{-1/2}$).
    The following decomposition serves as the basis of the power analysis:
    \begin{align}
        {\mh{U}\m{R}({\mh{U},\m{U}_0})-\m{U}_0} 
        &=
        \underbrace{\m{U}_{1} - \m{U}_0}_{\eqqcolon\m{Q}_{1}} 
        +
        \underbrace{{\mh{U}\Rs{\mh{U}}{\m{U}_{1}} - \m{U}_{1}} }_{\eqqcolon \m{Q}_2} 
        +
        \underbrace{\mh{U}\{\Rs{\mh{U}}{\m{U}_0} - \Rs{\mh{U}}{\m{U}_{1}}\}}_{\eqqcolon \m{Q}_3}. \label{supp-eq:power-decomposition}
    \end{align}
    First, note that $\tinorm{\m{Q}_{1}} = d_{n}$ is a deterministic quantity that measures the distance between $\m{U}_{0}$ and $\m{U}_{1}$. Next, under $\hypa$, the rows of $\m{Q}_2$ are asymptotically i.i.d.~zero mean Gaussian vectors by \cref{\mainlabel{body-lemma:first-order-approximation}}, and \cref{\mainlabel{body-theorem:Gumbel-convergence-plugin}} gives the limiting distribution of $\tinorm{\m{Q}_{2}}$ properly shifted and scaled. Finally, the matrix $\m{Q}_3$ is a random matrix whose two-to-infinity norm is controlled by \cref{supp-lemma:switching-orthogonal-alignment-matrices}, namely
    \begin{align}
        \tinorm{\m{Q}_{3}}
        \lesssim
        \plr{\frac{\mu r}{n}}^{1/2} \tnorm{\m{U}_{0} - \m{U}_{1}}
        \leq (\mu r)^{1/2} d_n,
        \label{supp-eq:Q3-bound}
    \end{align}
    with probability exceeding $1-O(n^{-9})$. Putting all the ingredients together, the test statistic developed in \cref{\mainlabel{body-theorem:Gumbel-convergence-plugin}} satisfies
    \begin{align*}
        \tplug 
        &=\widetilde{a}_n\inv (\tinorm{\m{Q}_{1} + \m{Q}_2 + \m{Q}_3} -\widetilde{b}_n) 
        +
        \log(\abplug). 
    \end{align*}
    Applying the triangle inequality twice based on the decomposition in \cref{supp-eq:power-decomposition}, we obtain the upper bound
    \begin{align*}
        \tplug 
        &\leq \left[\widetilde{a}_n\inv (\tinorm{\m{Q}_2} -\widetilde{b}_n) 
        +
        \log(\abplug)\right] 
        +
        \widetilde{a}_n\inv (d_{n}+\tinorm{\m{Q}_3}),
    \end{align*}
    and the lower bound
    \begin{align*}
        \tplug 
        &\geq \left[\widetilde{a}_n\inv (\tinorm{\m{Q}_2} -\widetilde{b}_n) 
        +
        \log(\abplug)\right]
        -
        \widetilde{a}_n\inv (d_{n}+\tinorm{\m{Q}_3}).
    \end{align*}
    Thus, the test statistic has the form
    \begin{equation*}
        \tplug
        = \left[\widetilde{a}_n\inv (\tinorm{\m{Q}_2} -\widetilde{b}_n)
        +
        \log(\abplug)\right] 
        +
        \psi_n,
    \end{equation*}
    where $|\psi_n| \leq \widetilde{a}_n\inv (d_{n}+\tinorm{\m{Q}_3})$. By the hypothesis that $d_{n} \ll \widetilde{a}_n(\mu r)^{-1/2}$, alongside \cref{supp-eq:Q3-bound}, we have $\widetilde{a}_n\inv (d_{n}+\tinorm{\m{Q}_3})\to 0$ in probability as $n\to\infty$. Furthermore, \cref{\mainlabel{body-theorem:Gumbel-convergence-plugin}} states that under $\hypa$, we have $\widetilde{a}_n\inv (\tinorm{\m{Q}_2} -\widetilde{b}_n) + \log(\abplug)\rightsquigarrow \gumbelrv$ as $n\to\infty$. Hence, invoking Slutsky's lemma yields $\tplug\rightsquigarrow \gumbelrv$ under $\hypa$. That is, the test statistic still converges to the null distribution under the alternative hypothesis. Therefore, given any specified Type I error probability $\alpha \in (0,1)$, the power satisfies
    \begin{equation*}
        \E_{\hypa}[\Is{\tplug\geq F_G\inv(1-\alpha)}]
        =
        \P(\tplug \geq F_G\inv(1-\alpha)\mid \hypa) 
        \to
        \P(G \geq F_G\inv(1-\alpha)) 
        =
        \alpha
    \end{equation*}
    as $n\to\infty$.
    
    \noindent
    \textbf{(ii)} (Consistent regime: $d_{n} \gg \widetilde{b}_n(\mu r)^{1/2}$). Recall $\m{\Psi} \coloneqq \mh{U}\Rs{\mh{U}}{\m{U}_{1}}-\m{U}_{1} - \m{EVS}\inv$ by \cref{\mainlabel{body-lemma:first-order-approximation}}.
    It holds that
    \begin{align*}
        \mh{U}\m{R}({\mh{U},\m{U}_0})-\m{U}_0 
        & = (\m{U}_{1} +\m{EVS}\inv +\m{\Psi})\Rs{\m{U}_{1}}{\mh{U}}\Rs{\mh{U}}{\m{U}_{0}} - \m{U}_{0}
        \\
        &=
        \m{U}_{1}\Rs{\m{U}_{1}}{\mh{U}}\Rs{\mh{U}}{\m{U}_{0}} - \m{U}_0
        +
        (\m{EVS}\inv + \m{\Psi})\Rs{\m{U}_{1}}{\mh{U}}\Rs{\mh{U}}{\m{U}_{0}}.
    \end{align*}
    By applying the two-to-infinity norm, the reverse triangle inequality, and \cref{supp-lemma:misaligned-hypotheses}, there exist constants $N\in\mathbb{N}$ and $C>0$ such that for all $n\geq N$,
    \begin{align*}
       \tinorm{\mh{U}\m{R}({\mh{U},\m{U}_0})-\m{U}_0} 
       &\ge
       \tinorm{\m{U}_{1}\Rs{\m{U}_{1}}{\mh{U}}\Rs{\mh{U}}{\m{U}_{0}}  - \m{U}_0}
       -
       \tinorm{\m{EVS}\inv +\m{\Psi}}
       \\
       &\ge
       C(\mu r)^{-1/2} d_{n} - \tinorm{\mh{U}\Rs{\mh{U}}{\m{U}_{1}}-\m{U}_{1}},
    \end{align*}
    holds with probability exceeding $1-O(n^{-9})$. It follows that 
    \begin{align}
        \tplug 
        &=
        \widetilde{a}_{n}\inv(\tinorm{\mh{U}\m{R}({\mh{U},\m{U}_0})-\m{U}_0} -\widetilde{b}_n) +\log (\abplug)
        \nonumber\\
        &\geq
        \widetilde{a}_{n}\inv\left(C(\mu r)^{-1/2} d_{n} - \tinorm{\mh{U}\Rs{\mh{U}}{\m{U}_{1}}-\m{U}_{1}} -\widetilde{b}_n\right) +\log (\abplug)
        \nonumber\\
        &=
        \underbrace{
            C(\mu r)^{-1/2} \widetilde{a}_{n}\inv d_{n} - 2(\widetilde{a}_{n}\inv\widetilde{b}_n - \log (\abplug))
        }_{\eqqcolon P_{1}}
        \nonumber\\
        &\qquad
        - 
        \underbrace{
            \clr{
                \widetilde{a}_{n}\inv(\tinorm{\mh{U}\Rs{\mh{U}}{\m{U}_{1}}-\m{U}_{1}} -\widetilde{b}_n) 
                +
                \log (\abplug)
            }
        }_{\eqqcolon P_{2}}.
        \label{supp-eq:tstat-consistent-b}
    \end{align}
    By \cref{\mainlabel{body-theorem:Gumbel-convergence-plugin}}, the term $P_{2}$ in \cref{supp-eq:tstat-consistent-b} satisfies
    \begin{align}
        P_{2} 
        \rightsquigarrow 
        \gumbelrv
        \label{supp-eq:B-to-Gumbel}
    \end{align} 
    as $n\to\infty$. Furthermore, under \cref{\mainlabel{body-assumption:delocalization},\mainlabel{body-assumption:snr},\mainlabel{body-assumption:gap-lam1-lam2}}, the term $P_{1}$ in \cref{supp-eq:tstat-consistent-b} satisfies 
    \begin{align}
        P_{1}
        \to
        \infty,
        \label{supp-eq:A-to-infinity}
    \end{align}
    with probability exceeding $1-O(n^{-6})$ since $d_{n}(\mu r)^{-1/2}\gg \widetilde{b}_{n} \gg \widetilde{a}_{n}$ by hypothesis, and $\log (\abplug)$ is bounded below with probability exceeding $1-O(n^{-6})$ by \cref{\mainlabel{body-assumption:gap-lam1-lam2},\mainlabel{body-proposition:debiased-singular-values}}.
    Combining \cref{supp-eq:B-to-Gumbel,supp-eq:A-to-infinity}, the power satisfies
    \begin{align*}
        \E_{\hypa}[\Is{\tplug\geq F_G\inv(1-\alpha)}]  
        \geq 
        \P\left(P_{1}- P_{2}\geq  F_G\inv(1-\alpha)\mid \hypa \right)
        \to 
        \P(G\leq \infty) 
        =
        1
    \end{align*}
    as $n\to\infty$. This completes the proof of \cref{\mainlabel{body-theorem:power-analysis}}.
\end{proof}

\subsection{Towards proving \cref{\mainlabel{body-proposition:non-gaussian-noise-rank-1}}}
\label{supp-section:non-gaussian-noise-rank-1}

\subsubsection{A technical lemma}

First, we define the notion of a sub-exponential random variable \citep[Definition 2.7]{wainwright_high-dimensional_2019}.
\begin{definition}[Sub-exponential random variable]
\label{supp-definition:sub-exponential-random-variable}
    A mean-zero random variable $X$ is $(\nu, K)$-sub-exponential if there exist constants $\nu$ and $K$ such that $\E[\exp(sX)]\leq \exp(\nu^{2}s^{2}/2)$ for all $|s|<1/K$.
\end{definition}

By \citet[Theorem 2.13]{wainwright_high-dimensional_2019}, the existence of the MGF on an interval containing zero is an equivalent definition of a $(\nu, K)$-sub-exponential variable.
In other words, the noise condition in \cref{\mainlabel{body-proposition:non-gaussian-noise-rank-1}} requires that the entries of $\m{E}$ are $(\nu, K)$-sub-exponential with fixed parameters $\nu$ and $K$.
\cref{supp-lemma:E-concentration-subexponential} is an analogue of concentration results in \citet[Lemma 19]{yan_entrywise_2024} for sub-exponential random variables, which enables the extension of \cref{\mainlabel{body-lemma:first-order-approximation}} to the sub-exponential noise setting.

\begin{lemma}[Tail bounds for sub-exponential random variables]
    \label{supp-lemma:E-concentration-subexponential}
    Let $\m{E}= (E_{i,j})_{1\leq i,j \leq n}$ have i.i.d. zero mean, unit variance entries, and there exist non-negative constants $\nu, k$ such that $M_{E_{1,1}}(s)\leq \exp(\nu^{2}s^{2})$ for all $|s|< 1/k$.
    Also, let $\m{Q}\in\m{R}^{n\times r_{1}}$ and $\m{T}\in\m{R}^{m\times r_{2}}$ be matrices independent of $\m{E}$ such that $\m{Q}\T \m{Q} = \m{I}_{r_{1}}$ and $\m{T}\T\m{T} = \m{I}_{r_{2}}$ with probability one, where $r_{1}\in\dpar{n}$ and $r_{2}\in\dpar{m}$.
    Then, under the condition $n/m\simeq 1$,
    \begin{align}
        |E_{i,j}|
        &\lesssim
        k \log n
        \tag{a}
        \label{supp-eq:concentration-subexponential-a}
    \end{align}
    holds with probability exceeding $1-O(n^{-9})$,
    \begin{align}
        \tnorm{\m{E}}
        &\lesssim
        \sigma\sqrt{n}
        + 
        c k \log^{3/2}n
        \tag{b}
        \label{supp-eq:concentration-subexponential-b}
    \end{align}
    holds with probability exceeding $1-O(n^{-8})$, and
    \begin{align}
        \tnorm{\m{Q}\T \m{E}\m{T}}
        \lesssim
        \sigma\sqrt{\max\{r_{1},r_{2}\}\log{n}}
        +
        k \mu_{0}\sqrt{r_{1}r_{2}} 
        \plr{
            \frac{\log^{2}n}{n}
        }
        \tag{c}
        \label{supp-eq:concentration-subexponential-c}
    \end{align}
    holds with probability exceeding $1-O(n^{-7})$, where
    \begin{align*}
        \mu_{0} 
        \coloneqq 
        \max
        \clr{
            \frac{n\tinorm{\m{Q}}^{2}}{r_{1}},
            \frac{m\tinorm{\m{T}}^{2}}{r_{2}}
        }.
    \end{align*}
\end{lemma}
\begin{proof}
    See \cref{supp-pf:E-concentration-subexponential}    
\end{proof}

\subsubsection{Proof of \cref{\mainlabel{body-proposition:non-gaussian-noise-rank-1}}}
\label{supp-pf:non-gaussian-noise-rank-1}

The proof of \cref{\mainlabel{body-proposition:non-gaussian-noise-rank-1}} consists of two parts as outlined in the remark following the statement of \cref{\mainlabel{body-proposition:non-gaussian-noise-rank-1}}.

\begin{proof}

\textbf{(i)} (Perturbation analysis).
    By assumption, there exists a constant $k>0$ such that $M_{E_{1,1}}(s)$ is finite on $|s| < 1/k$.
    Thus, the Taylor series expansion $K_{E_{1,1}}(s) = s^{2}/2 + O(s^{3})$ around $s = 0$ implies there exists a constant $\nu$ such that $M_{E_{1,1}}(s)\leq \exp(\nu^{2} s^{2}/2)$ for all $|s|< 1/k$. Therefore, plugging in $r=\sigma =1$ into \cref{supp-lemma:E-concentration-subexponential}\eqref{supp-eq:concentration-subexponential-b} yields
    $
    \label{supp-eq:E-concentration-subexponential-simplified}
    \tnorm{\m{E}} 
    \lesssim
    \sqrt{n}
    $
    with probability exceeding $1-O(n^{-8})$.
    Furthermore, as long as $\mu_{0}\ll \sigma n (\log n)^{-3/2} \sqrt{\max\{r_{1},r_{2}\}/\min\{r_{1},r_{2}\}}$, it holds that $\tnorm{\m{Q}\T\m{E}\m{T}}\lesssim \sigma\sqrt{\max\{r_{1},r_{2}\}\log n}$ with probability exceeding $1-O(n^{-7})$.
    Thus, a modified version of \cref{\mainlabel{body-lemma:first-order-approximation}} can be derived by replacing the sub-Gaussian concentration results \citet[Lemma 19]{yan_entrywise_2024} in the proof of \citet[Proposition 2]{yan_entrywise_2024} with the sub-exponential concentration results for the spectral norm of $\m{E}$ and $\m{Q}\T\m{E}\m{T}$. 
    Namely, under the simplified model assumptions, it holds that
    \begin{align}
            \tinorm{\mh{u}\sgn(\mh{u}\T\m{u}) - \m{u} - \m{Ev}s_{r}^{-1}}
            \lesssim
            \frac{\sqrt{\log n}}{s_{r}}
            \plr{
                \frac{\sqrt{n\log n}}
                {s_{r}}
            }
            \label{supp-eq:first-order-approximation-subexponential}
        \end{align}
    with probability exceeding $1 - O(n^{-5})$.

\noindent
\textbf{(ii)} (Gumbel convergence of the first-order approximation)
    Here, define the $i$-th row of the first-order approximation as $X_{i,n} \coloneqq \m{E}_{i,\all}\T \m{v} s_{r}^{-1} = \sum_{j=1}^{n} \frac{1}{s_1 \sqrt{n}} \m{E}_{i,j}$, and denote $Y_{i,n} \coloneqq \m{Z}_{i,\all}\T \m{v}s_{r}^{-1}$, where $\m{Z}$ is a matrix which has i.i.d.~$\NN(0,1)$ entries independent of $\m{E}$. We proceed in two steps. First, we show the difference of rate functions (defined below) of $E_{1,1}$ and $Z_{1,1}$ converges to zero sufficiently quickly. Next, we show that the convergence in rate functions implies convergence of the suitably normalized maximum entry of the first-order approximation.

\noindent
\textbf{(ii-a)} (Convergence of rate functions).
    Let $I_{\xi}(x) \coloneqq \sup_{s\in\R}(sx - K_{\xi}(s))$ denote the \textit{rate function} of a random variable $\xi$ whose MGF is finite on an interval containing zero (note that $sx -K_{\xi}(s) = -\infty$ at $s$ for which the MGF is infinite).
    By this definition, it holds that $I_{Z_{1,1}}(y) = \sup_{s\in\R}(sy - s^{2}/2) = y^{2}/2$. On the other hand, the rate function of $E_{1,1}$ is analyzed via its Taylor series expansion around zero.
    For the constant term, we have $I_{E_{1,1}}(0) = \sup_{s\in\R}\{-K_{E_{1,1}}(s)\} = 0$, since $M_{E_{1,1}}(s) = \E[\exp(s E_{1,1})]\geq 1$ by Jensen's inequality, and the lower bound is attained at $s = 0$. 
    
    For the linear term, since $sy - K_{E_{1,1}}(s)$ is a concave function of $s$ due to the convexity of $K_{E_{1,1}}(s)$, it attains its maximum at the saddle point $\widehat{s}_{E_{1,1}}(y)$, which satisfies $y = K_{E_{1,1}}'(\widehat{s}_{E_{1,1}}(y))$. 
    Thus, the rate function can be expressed as $I_{E_{1,1}}(y) = \widehat{s}_{E_{1,1}}(y)y - K_{E_{1,1}}(\widehat{s}_{E_{1,1}}(y))$, hence it holds that 
    \begin{align}
    \label{supp-eq:rate-derivative}
        I'_{E_{1,1}}(y) 
        =
        \widehat{s}'_{E_{1,1}}(y)\{y-K'(\widehat{s}_{E_{1,1}}(y))\} + \widehat{s}_{E_{1,1}}(y) 
        =
        \widehat{s}_{E_{1,1}}(y).    
    \end{align}
    Since $K'_{E_{1,1}}(0) = \E[E_{1,1}] = 0$, it follows that $I'_{E_{1,1}}(0) = \widehat{s}_{E_{1,1}}(0) = 0$.
    
    For the quadratic term, \cref{supp-eq:rate-derivative} together with the saddle point equation yields the relationship $y = K'_{E_{1,1}}(I'_{E_{1,1}}(y))$. Thus, by the inverse function theorem, it holds that
    \begin{align}
        I''_{E_{1,1}}(y) 
        =
        \frac{1}
        {K''_{E_{1,1}}(I'_{E_{1,1}}(y))}.
        \label{supp-eq:rate-second-derivative}
    \end{align}
    Hence, it follows that $I''_{E_{1,1}}(0) = 1/K''_{E_{1,1}}(I'_{E_{1,1}}(0)) = 1/K''_{E_{1,1}}(0)=1$, since $K''_{E_{1,1}}(0) = \E[E_{1,1}^{2}]=1$.

    For the remainder of the proof, assume $x\in\R$ is fixed unless stated otherwise. Define 
    \begin{align}
        u_{n}(x) 
        \coloneqq
        \frac{1}{\sqrt{n}}(s_{r}a_{n}x + s_{r}b_{n}) 
        =
        \sqrt{\frac{2\log n}{n}} 
        +
        O\plr{\frac{\log\log n}{\sqrt{n\log n}}},
        \label{supp-eq:un-definition}
    \end{align} 
    where the last relation is due to the fact that $s_{r}a_{n} = 1/\sqrt{2\log n}$ and $s_{r}b_{n} = \sqrt{2\log n} + O(\log\log n / \sqrt{n})$.
    Under the assumption that $x\in\R$ is fixed, the notation $u_{n}$ may be used in the place of $u_{n}(x)$ at times to simplify notation.
    By \cref{supp-eq:rate-derivative,supp-eq:rate-second-derivative}, the Taylor series expansion of the rate function of $E_{1,1}$ around zero is $I_{E_{1,1}}(y) = y^{2}/2 + O(y^{3})$ for $|y| < 1$, which yields
    \begin{align}
        n \alr{ 
            I_{E_{1,1}}\plr{u_{n}(x)}
            -
            I_{Z_{1,1}}\plr{u_{n}(x)}
            }
        &=
        O(n \{u_n(x)\}^{3})
        \nonumber
        \\
        &= 
        O(n^{-1/2}\{\log n\}^{3/2})
        \to
        0
        \label{supp-eq:rate-difference-convergence}
    \end{align}
    as $n\to\infty$. Furthermore, \cref{supp-eq:rate-derivative,supp-eq:rate-second-derivative} with Taylor expansions of $I'_{E_{1,1}}(y)$ and $I''_{E_{1,1}}(y)$ around zero yield
    \begin{align}
        \lim_{n\to\infty}
        \frac{\widehat{s}_{Z_{1,1}}\plr{u_{n}(x)}}
        {\widehat{s}_{E_{1,1}}\plr{u_{n}(x)}}
        =
        \lim_{n\to\infty}
        \frac{I''_{Z_{1,1}}\plr{u_{n}(x)}}
        {I''_{E_{1,1}}\plr{u_{n}(x)}}
        =
        1,
        \label{supp-eq:saddlepoint-second-derivative-ratio-convergence}
    \end{align}
    since $u_{n}(x) \to 0$ as $n\to\infty$ for any fixed $x\in\R$.

    Notably, since $-E_{1,1}$ also has zero mean, unit variance, and finite higher moments, applying the above procedure to $-E_{1,1}$ yields
    \begin{align}
        n \alr{ 
            I_{-E_{1,1}}\plr{u_{n}(x)}
            -
            I_{Z_{1,1}}\plr{u_{n}(x)}
            }
        \to
        0,
        \label{supp-eq:rate-difference-convergence-negative}
    \end{align}
    and
    \begin{align}
        \lim_{n\to\infty}
        \frac{\widehat{s}_{Z_{1,1}}\plr{u_{n}(x)}}
        {\widehat{s}_{-E_{1,1}}\plr{u_{n}(x)}}
        =
        \lim_{n\to\infty}
        \frac{I''_{Z_{1,1}}\plr{u_{n}(x)}}
        {I''_{-E_{1,1}}\plr{u_{n}(x)}}
        =
        1,
        \label{supp-eq:saddlepoint-second-derivative-ratio-convergence-negative}
    \end{align}
    for any fixed $x\in\R$.
    
\noindent
\textbf{(ii-b)} (Weak convergence of the maximum).
    The Bahadur--Rao theorem \citep{bahadur_deviations_1960} yields 
    \begin{align} 
        1-F_{X_{n}}(a_{n}x + b_{n})
        &=
        \P(s_{r}X_{n} \geq s_{r}a_{n}x + s_{r}b_{n})
        \nonumber
        \\ 
        &=
        \P\plr{
            \frac{1}{n}\sum_{j=1}^n E_{i,j}
            \geq
            u_{n}(x)
            }
        \nonumber
        \\
        &=
        \frac{1 + o(1)}
        {\widehat{s}_{E_{1,1}}(u_n)\sqrt{2\pi n K''_{E_{1,1}}(\widehat{s}_{E_{1,1}}(u_n))}}
        \exp\plr{
            -n I_{E_{1,1}}\plr{u_n(x)}
            },
        \label{supp-eq:bahadur-rao-Xn}
    \end{align}
    \begin{align} 
        1-F_{-X_{n}}(a_{n}x + b_{n})
        &=
        \frac{1 + o(1)}
        {\widehat{s}_{-E_{1,1}}(u_n)\sqrt{2\pi n K''_{-E_{1,1}}(\widehat{s}_{-E_{1,1}}(u_n))}}
        \exp\plr{
            -n I_{-E_{1,1}}\plr{u_n(x)}
            },
        \label{supp-eq:bahadur-rao-Xn-negative}
    \end{align}
    and
    \begin{align} 
        1-F_{Y_{n}}(a_{n}x + b_{n})
        =
        \frac{1 + o(1)}
        {\widehat{s}_{Z_{1,1}}(u_{n})\sqrt{2\pi n K''_{Z_{1,1}}(\widehat{s}_{Z_{1,1}}(u_{n}))}}
        \exp\plr{
            -n I_{Z_{1,1}}\plr{u_{n}(x)}
            }
        \label{supp-eq:bahadur-rao-Yn}
    \end{align}
    as $n\to\infty$.
    By \cref{supp-eq:bahadur-rao-Xn,supp-eq:bahadur-rao-Xn-negative,supp-eq:bahadur-rao-Yn,supp-eq:rate-difference-convergence,supp-eq:saddlepoint-second-derivative-ratio-convergence,supp-eq:rate-difference-convergence-negative,supp-eq:saddlepoint-second-derivative-ratio-convergence-negative}, it holds that
    \begin{align}
        &\lim_{n\to\infty}
        \frac{1-F_{|X_{i,n}|}(a_{n}x + b_{n})}
        {1-F_{|Y_{i,n}|}(a_{n}x + b_{n})}
        \nonumber
        \\
        &\qquad
        =
        \lim_{n\to\infty}
        \frac{\{1-F_{X_{i,n}}(a_{n}x + b_{n})\} + \{1-F_{-X_{i,n}}(a_{n}x + b_{n})\}}
        {2\{1-F_{Y_{i,n}}(a_{n}x + b_{n})\}}
        \nonumber
        \\
        &\qquad
        =
        \frac{1}{2}
        \lim_{n\to\infty} 
        \plr{
            \frac{\widehat{s}_{Z_{1,1}}(u_{n})\sqrt{K''_{Z_{1,1}}(\widehat{s}_{Z_{1,1}}(u_{n}))}}
            {\widehat{s}_{E_{1,1}}(u_{n})\sqrt{K''_{E_{1,1}}(\widehat{s}_{E_{1,1}}(u_{n}))}}
        }
        (1 + o(1))
        \exp
        \plr{
            -n\clr{
                I_{E_{1,1}}\plr{u_{n}(x)} 
                -
                I_{Z_{1,1}}\plr{u_{n}(x)} 
                }
        }
        \nonumber
        \\
        &\qquad\qquad
        +
        \frac{1}{2}
        \lim_{n\to\infty} 
        \plr{
            \frac{\widehat{s}_{Z_{1,1}}(u_{n})\sqrt{K''_{Z_{1,1}}(\widehat{s}_{Z_{1,1}}(u_{n}))}}
            {\widehat{s}_{-E_{1,1}}(u_{n})\sqrt{K''_{-E_{1,1}}(\widehat{s}_{-E_{1,1}}(u_{n}))}}
        }
        (1 + o(1))
        \exp
        \plr{
            -n\clr{
                I_{-E_{1,1}}\plr{u_{n}(x)} 
                -
                I_{Z_{1,1}}\plr{u_{n}(x)} 
                }
        }
        \nonumber
        \\
        &\qquad 
        =
        1.
        \label{supp-eq:tail-ratio-to-1}
    \end{align}
    
    Finally, we use \cref{supp-eq:tail-ratio-to-1} to establish $[F_{|X_{i,n}|}(a_{n} x + b_{n})]^{n} \to \exp(-\exp(-x))$. Let $\xi_{n}$ denote a placeholder which can be either $|X_{i,n}|$ or $|Y_{i,n}|$. Observe that $[F_{\xi_{n}}(a_{n} x + b_{n})]^{n} \to \exp(-\exp(-x))$ as $n\to\infty$ is equivalent to $n\log (F_{\xi_{n}}(a_{n} x + b_{n})) \to -\exp(-x)$ as $n\to\infty$. Further, taking the Taylor series expansion of $\log(1-y)$ around zero yields 
    \begin{align}
        n\log (1-\{1-F_{\xi_{n}}(a_{n} x + b_{n})\}) 
        =
        -n\{1-F_{\xi_{n}}(a_{n} x + b_{n})\} 
        +
        O(n\{1-F_{\xi_{n}}(a_{n} x + b_{n})\}^{2}).
        \label{supp-eq:log-cdf-taylor-prelim}
    \end{align}
    To bound the residual of the Taylor series expansion, first, recalling that $s_{r}Y_{i,n}\sim \mathcal{N}(0,1)$, and $s_{r}(a_{n}x + b_{n}) = \sqrt{2\log n} + o(1)$, a Gaussian tail bound \citep[Proposition 2.1.2]{vershynin_high-dimensional_2018} yields that for $n$ large enough,
    \begin{align}
    \label{supp-eq:gaussian-tail-bound-Y}
        1-F_{|Y_{i,n}|}(a_{n}x + b_{n}) 
        =
        2\{1-F_{s_{r}Y_{i,n}}(s_{r}a_{n}x 
        +
        s_{r}b_{n})\} 
        \leq 
        2\P
        \plr{
            s_{r}Y_{i,n} \geq \sqrt{\frac{3}{8}}\times\sqrt{2\log n}
        }
        \lesssim
        n^{-3/4}.
    \end{align}
    On the other hand, applying the Chernoff inequality yields that for $n$ large enough and for any $t>0$,
    \begin{align}
        \P(
            X_{i,n}
            \geq
            a_{n}x + b_{n}
        )
        &=
        \P(
            n^{1/2}s_{r}X_{i,n}
            \geq
            n^{1/2}s_{r}(a_{n}x + b_{n})
        )
        \nonumber
        \\
        &
        =
        \P
        \plr{
                \sum_{j=1}^{n} E_{i,j}
                \geq
                n^{1/2}s_{r}(a_{n}x + b_{n})
        }
        \nonumber
        \\
        &
        \leq
        \P
        \plr{
                \sum_{j=1}^{n} E_{i,j}
                \geq
                \sqrt{\frac{10}{11}} \times \sqrt{2n\log n}
        }
        \nonumber
        \\
        &\leq
        \exp
        \plr{
            -t
            \sqrt{
                \frac{20n\log n}{11} 
            }
            +
            n K_{E_{1,1}} (t)
        }.
        \label{supp-eq:Chernoff-bound-1}
    \end{align}
    Since $E_{1,1}$ is mean zero, unit variance, and has finite higher moments on an interval containing zero, a Taylor expansion of the cumulant generating function around $t=0$ yields $K_{E_{1,1}}(t) = \frac{1}{2}t^{2} + O(t^{3})\leq \frac{5}{8}t^{2}$ for $t$ sufficiently small. Thus, plugging in a diminishing quantity such as $t=\sqrt{\frac{5\log n}{11n}}$ into \cref{supp-eq:Chernoff-bound-1} yields that for $n$ large enough,
    \begin{align*}
        \P(
            X_{i,n}
            \geq
            a_{n}x + b_{n}
        )
        \leq
        \exp
        \plr{
            -\frac{10}{11}\log n + \frac{5}{8}\times \frac{5}{11}\log n
        }
        = 
        n^{-5/8}.
    \end{align*}
    Also, since $-E_{1,1}$ also has zero mean, unit variance, and finite higher moments, applying the above procedure to $-X_{i,n}$ analogously yields
    \begin{align*}
        \P(
            -X_{i,n}
            \geq
            a_{n}x + b_{n}
        )
        \leq
        n^{-5/8}.
    \end{align*}
    Hence, for $n$ large enough
    \begin{align}
    \label{supp-eq:Chernoff-bound-2}
        1-F_{|X_{i,n}|}(a_{n}x + b_{n}) 
        =
        \P(X_{i,n}\geq a_{n}x + b_{n})
        +
        \P(-X_{i,n}\geq a_{n}x + b_{n})
        \lesssim
        n^{-5/8}.
    \end{align}
    Thus, plugging in \cref{supp-eq:gaussian-tail-bound-Y,supp-eq:Chernoff-bound-2} to the Taylor expansion residual in \cref{supp-eq:log-cdf-taylor-prelim} yields
    \begin{align*}
        n\log (1-\{1-F_{\xi_{n}}(a_{n} x + b_{n})\}) 
        =
        -n\{1-F_{\xi_{n}}(a_{n} x + b_{n})\} 
        +
        O(n^{-1/4}),
    \end{align*}
    which implies that
    \begin{align}
        \lim_{n\to\infty}
        n \{1-F_{\xi_{n}}(a_{n}x + b_{n})\}
        \to
        \exp(-x)
        \iff
        \lim_{n\to\infty} [F_{\xi_{n}}(a_{n}x + b_{n})]^{n}
        \to
        \exp(-\exp(-x)).
        \label{supp-eq:gumbel-convergence-equivalent-condition}
    \end{align} 
    By \cref{\mainlabel{body-proposition:saddle-point-tail-equivalence},supp-lemma:tail-equivalence}, it holds that
    \begin{align}
    \lim_{n\to\infty}[F_{|Y_{i,n}|}(a_{n}x + b_{n})]^{n} 
    =
    \exp(\exp(-x)). 
    \end{align}
    Thus, \cref{supp-eq:gumbel-convergence-equivalent-condition} yields $\lim_{n\to\infty}n\{1-F_{|Y_{i,n}|}(a_{n}x + b_{n})\} =  \exp(-x)$.
    Furthermore, applying \cref{supp-eq:tail-ratio-to-1} yields
    \begin{align*}
        \lim_{n\to\infty} n\{1 - F_{|X_{i,n}|}(a_{n}x + b_{n})\}
        =
        \lim_{n\to\infty} 
        \frac{1-F_{|X_{n}|}(a_{n}x + b_{n})}
        {1-F_{|Y_{n}|}(a_{n}x + b_{n})} 
        n\{1 - F_{|Y_{i,n}|}(a_{n}x + b_{n})\}
        =
        \exp(-x).
    \end{align*}
    Finally, by \cref{supp-eq:gumbel-convergence-equivalent-condition}, it holds that $\lim_{n\to\infty} [F_{|X_{i,n}|}(a_{n}x + b_{n})]^{n}\to \exp(-\exp(-x))$ for all $x\in\R$, or equivalently, $a_{n}\inv(\tinorm{\m{Ev}s_{r}\inv} -b_{n}) \rightsquigarrow \gumbelrv$ as desired. 
    
    To conclude the proof of \cref{\mainlabel{body-proposition:non-gaussian-noise-rank-1}},
    defining $\gamma_{n}\coloneqq a_{n}\inv (\tinorm{\mh{u}\sgn({\mh{u}\T\m{u}})-\m{u}} - \tinorm{\m{Ev}s_{r}\inv})$ as in
    \cref{supp-lemma:first-order-convergence}, and applying the same proof technique therein, replacing \cref{\mainlabel{body-lemma:first-order-approximation}} with \cref{supp-eq:first-order-approximation-subexponential}, there exists a constant $C_{\gamma}' > 0$ such that
    \begin{align}
        \P\plr{
            |\gamma_{n}| 
            \leq 
            C_{\gamma}'
            \clr{
                \log n
                \plr{
                    \frac{\sqrt{n\log n}}
                    {s_{r}}
                }
            }
        }
        \geq 
        1- O(n^{-5}),
    \end{align}
    where we have used the fact that $a_{n}= O(\{s_{r}\log^{1/2} n\}^{-1})$.
    Therefore, since $\gamma_{n} = o_{\P}(1)$ holds under the assumed condition $s_{r}\gg \sqrt{n \log^{3} n}$, applying the Gumbel convergence of the first-order approximation along with Slutsky's lemma yields
    \begin{align}
    \label{supp-eq:first-order-convergence-non-Gaussian}
        a_{n}\inv
        \plr{\tinorm{\mh{u}\sgn(\mh{u}\T\m{u}) -\m{u}} - b_{n}}
        =
        a_{n}\inv
        \plr{\tinorm{\m{Ev}s_{r}\inv} - b_{n}} 
        +
        \gamma_{n} \rightsquigarrow \gumbelrv.
    \end{align}
    This completes the proof of \cref{\mainlabel{body-proposition:non-gaussian-noise-rank-1}}.
\end{proof}

\subsubsection{Proof of \cref{supp-lemma:E-concentration-subexponential}}
\label{supp-pf:E-concentration-subexponential}
\begin{proof}
Here, assume for all $i,j\in\dpar{n}$ that $E_{i,j}$ are sub-exponential with parameters $(\nu, K)$ (See \cref{supp-definition:sub-exponential-random-variable}).

\noindent 
\eqref{supp-eq:concentration-subexponential-a}: 
By the Chernoff bound, it holds that
\begin{align}
\label{supp-eq:chernoff-bound-subexponential}
    \P(|E_{i,j}| > t) 
    \leq
    \begin{cases}
        2\exp
        \plr{
        -\frac{t^{2}}{2\nu^{2}}
        }
        & \text{if} 
        \quad
        0\leq t \leq \frac{\nu^{2}}{k},
        \\
        2\exp
        \plr{
        -\frac{t}{2k} 
        }
        & \text{if}
        \quad
        t > \frac{\nu^{2}}{k}
        .
    \end{cases}
\end{align}
See for example \citet[Proposition 2.9]{wainwright_high-dimensional_2019} for the proof. For a threshold $t$ growing at rate $K\log n$, applying the second bound in \cref{supp-eq:chernoff-bound-subexponential} yields
\begin{align*}
    \P(|E_{i,j}| > 18k\log n) \leq 2n^{-9},
\end{align*}
which completes the proof of \eqref{supp-eq:concentration-subexponential-a}.

\noindent 
\eqref{supp-eq:concentration-subexponential-b}: 
Let $\widetilde{\m{E}}\coloneqq [\widetilde{E}_{i,j}]_{1\leq i,j\leq n}$, where $\widetilde{E}_{i,j} \coloneqq E_{i,j} \Is{|E_{i,j}|\leq 18k \log n}$.
By \eqref{supp-eq:concentration-subexponential-a}, it holds that $\P\{E_{i,j} = \widetilde{E}_{i,j}\}\geq 1- O(n^{-9})$, hence $\P\{\m{E} = \widetilde{\m{E}}\}\geq 1-O(n^{-7})$ by the union bound.
Applying \citet[Corollary 3.12]{bandeira_sharp_2016} (more specifically, the version in Remark 3.13 for square matrices of bounded random variables applied to the symmetric dilation of the matrix $\m{E}$), it holds that for some universal constant $c>0$,
\begin{align*}
    \tnorm{\widetilde{\m{E}}} 
    \leq
    4\sqrt{2}\sigma\sqrt{n} 
    + c k \log^{3/2}n
\end{align*}
for sufficiently large $n$ with probability exceeding $1-O(n^{-7})$.

\noindent
\eqref{supp-eq:concentration-subexponential-c}: 
The concentration of the term $\tnorm{\m{Q}\T\m{E}\m{T}}$ is derived via the truncated matrix Bernstein inequality \citep[Theorem 3.1]{chen_spectral_2021} using the observation that 
\begin{align*}
    \m{Q}\T\m{E}\m{T}
    = 
    \sum_{i=1}^{n}\sum_{j=1}^{m}E_{i,j}\m{Q}_{i,\all }\m{T}_{j,\all }\T
\end{align*}
where each summand is independent and mean zero. 
By the definition of $\mu_{0}$, it holds that
\begin{align}
\label{supp-eq:matrix-bernstein-summand}
    \tnorm{E_{i,j}\m{Q}_{i,\all }\m{T}_{j,\all }\T}
    \leq
    |E_{i,j}|\tnorm{\m{Q}_{i,\all }}\tnorm{\m{T}_{i,\all }}
    \leq
    \mu_{0}\sqrt{\frac{r_{1} r_{2}}{nm}}|E_{i,j}|.
\end{align}
Applying \eqref{supp-eq:concentration-subexponential-a} yields
\begin{align*}
    \P\plr{
        \tnorm{E_{i,j}\m{Q}_{i,\all }\m{T}_{j,\all }\T} 
        \geq
        18 k \mu_{0} \log n
        \sqrt{\frac{r_{1} r_{2}}{nm}}
    }
    &\leq
    \P\plr{|E_{i,j}| \geq 18 k\log n} 
    \\
    &\leq
    2n^{-9} \eqqcolon q_0.
\end{align*}
Further, setting $L \coloneqq 18k\mu_0 \log n \sqrt{r_{1}r_{2} (nm)^{-1}}$, we invoke Jensen's inequality and integration by parts to show that for a constant $C_B>0$,
\begin{align*}
    &\tnorm{
        \E[
            E_{i,j}\m{Q}_{i,\all}\m{T}\T_{j,\all}
            \Is{
                \tnorm{E_{i,j}\m{Q}_{i,\all}\m{T}\T_{j,\all}}
                \geq
                L
            }
        ]
    }
    \\
    &\qquad
    \leq
    \E[
        \tnorm{E_{i,j}\m{Q}_{i,\all}\m{T}\T_{j,\all}}
        \Is{
            \tnorm{E_{i,j}\m{Q}_{i,\all}\m{T}\T_{j,\all}}
            \geq
            L
        }
    ]
    \\
    &\qquad
    = \int_{L}^{\infty} 
    \P(\tnorm{E_{i,j}\m{Q}_{i,\all}\m{T}\T_{j,\all}} > x) 
    \dd x
    +
    L\P(\tnorm{E_{i,j}\m{Q}_{i,\all}\m{T}\T_{j,\all}} > L)    
    \\
    &\qquad
    \leq
    \int_{L}^{\infty} 
    \P\plr{
        |E_{i,j}| 
        >
        x\sqrt{nm(r_{1}r_{2}\mu_{0}^{2})^{-1}} 
    } 
    \dd x
    +
    L
    \P\plr{
        |E_{i,j}| 
        >
        L\sqrt{nm(r_{1}r_{2}\mu_{0}^{2})^{-1}} 
    }
    \\ 
    &\qquad
    \leq
    C_{B} \mu_{0} (r_{1}r_{2})^{1/2} k n^{-10}\log n 
    \eqqcolon
    q_{1},
\end{align*}
where the penultimate line is due to \cref{supp-eq:matrix-bernstein-summand}, and the last line uses \cref{supp-eq:chernoff-bound-subexponential,supp-eq:concentration-subexponential-a}.
Finally, write the matrix variance statistic as
\begin{align*}
    v
    &\coloneqq 
    \max
    \clr{
        \Norm{\sum_{i,j}  \E[E_{i,j}^2]\Norm{\m{T}_{j,\all }}{2}^2\m{Q}_{i,\all }\m{Q}_{i,\all }\T}{2}, 
        \Norm{\sum_{i,j}  \E[E_{i,j}^2]\Norm{\m{Q}_{i,\all }}{2}^2\m{T}_{j,\all }\m{T}_{j,\all }\T}{2}
    }
    \\
    &=
    \sigma^2 
    \max
    \clr{
        \Norm{\sum_{j}\Norm{\m{T}_{j,\all }}{2}^2\sum_i \m{Q}_{i,\all }\m{Q}_{i,\all }\T}{2}, \Norm{\sum_{i}\Norm{\m{Q}_{i,\all }}{2}^2\sum_j \m{T}_{j,\all }\m{T}_{j,\all }\T}{2}
    }
    \\
    &=
    \sigma^2
    \max
    \clr{
        \fnorm{\m{T}}^2 \tnorm{\m{Q}\T\m{Q}}, \fnorm{\m{Q}}^2   \tnorm{\m{T}\T\m{T}}
    } 
    \\
    &=
    \sigma^2 \max\{r_{1},r_{2}\}.
\end{align*}
Thus, the truncated matrix Bernstein inequality paired with \cref{\mainlabel{body-assumption:noise},\mainlabel{body-assumption:matrix-size}} results in
 \begin{align*}
    \tnorm{\m{Q}\T\m{ET}}
    &\leq
    \sqrt{16v\log(\max\{n,m\})} 
    +
    \frac{16}{3}L\log (\max\{n,m\})
    +
    n^{2}q_{1}
    \\
    &\lesssim
    \sigma\sqrt{\max\{r_{1},r_{2}\}\log{n}}
    +
    k \mu_{0}\sqrt{r_{1}r_{2}}
    \plr{
        \frac{\log^{2}n}{n}
    }
 \end{align*}
 with probability greater than $1-2(\max\{n,m\})^{-7} -nmq_0 = 1-O({n}^{-7})$.
\end{proof}

\section{Proofs of auxiliary lemmas for \cref{\mainlabel{body-proposition:saddle-point-tail-equivalence}}}
\label{supp-pf:saddle-point-approximation}

\subsection{Proof of \cref{supp-lemma:saddle-point}}
\label{supp-pf:saddle-point}

\begin{proof}
    From $K(t)=-\frac{1}{2}\sum_{j=1}^r \log(1- t\lambda_j/\lambda_{1} )$, we have $K'(t) = \frac{1}{2}\sum_{j=1}^r \frac{\lambda_j}{\lambda_{1}- \lambda_j t}$. The saddle point $\spt_{x}$ is the solution to the equation
    $K'(t)= x$. The function $K'(t)$ has poles at the points $1, \lambda_{1}/\lambda_{2},\dots, \lambda{1}/\lambda_{r}$ on the positive real line, the smallest of which is $1$.

    \noindent
    \textbf{(i)} ({For $x>0$, $K'(t)=x$ has a unique solution on $t<1$}). 
    Since $K''(t) = \frac{1}{2}\sum_{j=1}^r \frac{\lambda_j^2}{(\lambda_{1}- \lambda_j t)^2}>0$ for all $t<1$, the map $K':(-\infty,\lambda_{1})\to(0,\infty)$ is strictly increasing and hence injective.
    Furthermore, $K'$ is continuous with $\lim_{t\to-\infty} K'(t)= 0$, and $\lim_{t\to 1/\lambda_{1}}K'(t) = \infty$. 
    Hence, $K'$ is a surjective function by the intermediate value theorem.
    Consequently, $K'$ is a bijection, and \cref{supp-fig:saddle-point-equation} provides an illustration. Since $K'(0)=\E[\xp^2]$, it follows that $\widehat t>0$ when $x>\E[\xp^2]$. 
    Moreover, $K'$ is convex since $K^{(3)}(t) = \sum_{j=1}^r\frac{\lambda_j^3}{(\lambda_{1}- \lambda_j t)^3}>0$ for $t<1$ by \cref{\mainlabel{body-lemma:MGF}}.
     
     \noindent
     \textbf{(ii)} ({Expressing $\spt_{x}$ as a function of $x$}). Suppose that $x>\E[\xp^2]$. Since $\lambda_{1}=\lambda_{2} =\dots = \lambda_{\mul}$, the sum in $K'(t)$ can be decomposed into two parts:
    \begin{align}
    \label{supp-eq:K'-decomposition}
        K'(t)
        = 
        \frac{\mul}{2}
        \plr{
            \frac{1}
            {1- t} 
        }
        + 
        \frac{1}{2}
        \sum_{j=\mul +1}^r
        \frac{\lambda_j}{\lambda_{1}- \lambda_j t}.
    \end{align}  
    The first term on the right-hand side tends to infinity as $t\uparrow 1$, while the second term is bounded above by $\frac{(r-\mul)}{2} \frac{\lambda_2}{\lambda_{1}-\lambda_2}$ and below by $0$. Hence, the first term is the dominant term in \cref{supp-eq:K'-decomposition} as $t$ approaches $1$ from below. By multiplying $(1- \spt_{x})$ to both sides of the equation $K'(\spt_{x})=x$, we have
    \begin{align*}
        \frac{\mul}{2} 
        +
        \frac{(1- \spt_{x})}{2}
        \sum_{j=\mul + 1}^r \frac{\lambda_j}{\lambda_{1}-\lambda_j \spt_{x}} 
        = 
        x-x\spt_{x}.
    \end{align*}
    Organizing the terms and dividing both sides by $x$ results in
    \begin{align*}
        \spt_{x} 
        = 
        1 
        - 
        \frac{\mul}{2x}
        - 
        \frac{(1-\spt_{x})}{2x}
        \sum_{j=\mul + 1}^r 
        \frac{\lambda_j}{\lambda_{1} - \lambda_j \spt_{x}}.
    \end{align*}
    Now, plugging in the right-hand side of the equation above into each $\spt_{x}$ in the right-hand side yields
    \begin{align*}
        \spt_{x} 
        = 
        1 
        - 
        \frac{\mul}{2x} 
        - 
        \clr{
        \frac{\mul-  (1- \spt_{x}) 
        \sum_{j=\mul + 1}^r \frac{\lambda_j}{\lambda_{1}-\lambda_j \spt_{x}}}
        {4x^2}
        }
        \sum_{j=\mul + 1}^r \frac{\lambda_j}{\lambda_{1}-\lambda_j \spt_{x}}
        =
        1
        -
        \frac{\mul}{2x}
        +
        g(x),
    \end{align*}
    where 
    \begin{equation*}
        g(x)
        \coloneqq
        -
        \frac{1}{4x^2}
        \left\{
            \mul
            -
            (1- \spt_{x})
            \sum_{j=\mul + 1}^r \frac{\lambda_j}{\lambda_{1}-\lambda_j \spt_{x}}
        \right\}
        \sum_{j=\mul + 1}^r 
        \frac{\lambda_j}
        {\lambda_{1}-\lambda_j \spt_{x}}.
    \end{equation*}
    Since $\spt_{x} \in (0,1)$ when $x>\E[\xp^2]$, it holds that
    \begin{equation*}
        \left\{
            \mul
            -
            (1- \spt_{x})
            \sum_{j=\mul + 1}^r \frac{\lambda_j}{\lambda_{1}-\lambda_j \spt_{x}}
        \right\}
        \sum_{j=\mul + 1}^r 
        \frac{\lambda_j}
        {\lambda_{1}-\lambda_j \spt_{x}}
        =
        O(1)
    \end{equation*}
    as $x\to\infty$. It follows that 
    $g(x)=O(x^{-2})$ as $x\to\infty$. 
    
    In particular, if $r=1$, then the second term in the decomposition in \cref{supp-eq:K'-decomposition} vanishes, and we simply have $\spt_{x} = 1 -\frac{1}{2x}$.
\end{proof}

\begin{figure}[ht]
    \centering
        \includegraphics[width=7cm]{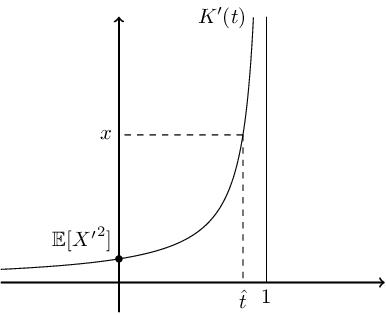}
    \caption{An illustration of the function $K'$. The saddle point $\spt_{x}$ is the value of $t$ at which $K'(t) = x$. The value of $\spt_{x}$ is uniquely defined on $t<1$, and approaches $1$ from below as $x$ grows to infinity.}
    \label{supp-fig:saddle-point-equation}
\end{figure}

\subsection{Proof of \cref{supp-lemma:saddle-point-approx-explicit}}
\label{supp-pf:saddle-point-approx-explicit}
\begin{proof}
    By \cref{\mainlabel{body-lemma:MGF}} and the fact that $M_{\xp^{2}}(t) = M_{X^2}(t/\lambda_{1}) $, we have
    \begin{align}
    \label{supp-eq:mgf-cgf-forms-r-gtr-1}
    M_{\xp^2}(t) 
    =
    \prod_{j=1}^r (1- t\lambda_j / \lambda_{1})^{-1/2}, 
    \quad 
    K_{\xp^2}(t) 
    =
    -\frac{1}{2}\sum_{j=1}^r\log(1- t \lambda_j / \lambda_{1}),
    \quad
    K_{\xp^2}''(t)
    =
    \frac{1}{2}\sum_{j=1}^r\frac{\lambda_j^2}{(\lambda_{1} - \lambda_j t)^2}.
    \end{align}
    Plugging in $\spt_{x} = 1 - \frac{\mul}{2x}+g(x)$ from \cref{supp-lemma:saddle-point} to \cref{supp-eq:mgf-cgf-forms-r-gtr-1} yields
    \begin{align*}
        M_{\xp^2}(\spt_{x})
        =
        \left[
        \frac{1}{2x}\{\mul-2xg(x)\} 
        \right]^{-\mul/2}
        \prod_{j=\mul + 1}^r (1-\spt_{x} \lambda_j / \lambda_{1} )^{-1/2},
    \end{align*}
    and
    \begin{align}
        K_{\xp^2}(\spt_{x})
        &=
        \log
        \left(
            \left[
                {2x\frac{x}{\{\mul x -  2x^2 g(x)\}}}
            \right]^{\mul / 2}
        \right)
        -
        \frac{1}{2}
        \sum_{j=\mul + 1}^r 
        \log(1-\spt_{x}\lambda_j/\lambda_{1})
        \nonumber\\
        &=
        \frac{\mul}{2}
        \log
        \left(
            2x
        \right)
        +\frac{\mul}{2}
        \log
        \left(
            \frac{1}{\mul}
            +\frac{2x^2g(x)}{\mul^{2} x - 2\mul x^2g(x)}
        \right)
        -\frac{1}{2}
        \sum_{j=\mul + 1}^r 
        \log(1- \spt_{x}\lambda_j/\lambda_{1}),
        \label{supp-eq:cgf-r-gtr-1}
    \end{align}
    and
    \begin{align}
        K_{\xp^2}''(\spt_{x}) 
        =
        \frac{\mul}{2}
        \frac{1}{\left\{\frac{\mul}{2x}-g(x)\right\}^2} 
        +
        \frac{1}{2}
        \sum_{j=\mul + 1}^r \frac{\lambda_j^2}{{(\lambda_{1}-\lambda_j \spt_{x})^2}}
        =
        \frac{2 x^2}{\mul \left\{1 - 2xg(x)/\mul\right\}^2}
        +
        \frac{1}{2}
        \sum_{j=\mul + 1}^r
        \frac{\lambda_j^2}{{(\lambda_{1}-\lambda_j \spt_{x})^2}}. 
        \label{supp-eq:cgf2-r-gtr-1}
    \end{align}
    By \cref{supp-eq:cgf-r-gtr-1,supp-eq:cgf2-r-gtr-1}, we have
    \begin{align*}
        \frac{\exp(K_{X^2}(\spt_{x})-\spt_{x}x)}
        {\sqrt{K''_{X^2}(\spt_{x})}} 
        &=
        \frac{
            \plr{2x}^{\mul/2}
            \left(
                \frac{1}{\mul}
                +
                \frac{2x^2 g(x)}{\mul^{2} x- 2\mul x^2 g(x)}
            \right)^{\mul/2}
            \plr{
                \prod_{j=\mul + 1}^r (1-\spt_{x}\lambda_j /\lambda_{1})^{-1/2}
            }
            \exp
            \left(-x+\frac{\mul}{2}-xg(x)\right)
        }
        {
            \sqrt{
                \frac{2x^2}{\mul\left\{1- 2xg(x)/\mul\right\}^2} 
                +
                \frac{1}{2}
                \sum_{j=\mul + 1}^r 
                \frac{\lambda_j^2}{{(\lambda_{1}-\lambda_j \spt_{x})^2}}
            }
        }
        \\
        &=
        \frac{
            2^{\frac{\mul-1}{2}}
            x^{\frac{\mul - 2}{2}}
            \left(
                1
                +
                \frac{2x^2 g(x)}{\mul x- 2 x^2 g(x)}
            \right)^{\mul/2}
            \plr{
                \prod_{j=\mul + 1}^r (1- \spt_{x}\lambda_j /\lambda_{1})^{-1/2}
            }
            \exp
            \left(-x+\frac{\mul}{2}-xg(x)\right)
        }
        {
            \mul^{\frac{\mul-1}{2}}
            \sqrt{
                \frac{1}{\left\{1- 2xg(x)/\mul\right\}^2} 
                +
                \frac{\mul}{4x^{2}}
                \sum_{j=\mul + 1}^r 
                \frac{\lambda_j^2}{{(1-\lambda_j \spt_{x})^2}}
            }
        },
    \end{align*}
    as desired. This completes the proof of \cref{supp-lemma:saddle-point-approx-explicit}.
\end{proof}

\subsection{Towards proving \cref{supp-lemma:fW-convergence}}
\label{supp-pf:fW-convergence}

\subsubsection{Technical lemmas}
It will be seen shortly that the random variable $W_{x}$ is distributed as a sum of $r$ non-identical gamma random variables with shared shape parameter $1/2$. 
Although the density of a gamma random variable with shape parameter $1/2$ has a singularity at zero, \cref{supp-lemma:gamma-Lipschitz} shows that it is still bounded and Lipschitz on a closed interval bounded away from zero. Crucially, the upper bound and Lipschitz constant only depends on the left end point of the interval, and does not depend on the scale parameter.

\begin{lemma}[Bounded and Lipschitz family of gamma densities on a closed interval]
\label{supp-lemma:gamma-Lipschitz}
    Let $g \sim \operatorname{Gamma}(1/2, \theta)$ where $\theta >0$. For any interval $[p_{1}, p_{2}]$ such that $0<p_{1}<p_{2} <\infty$, there exist finite positive constants $C_1(p_{1})$ and $C_1(p_{1})$ only depending on the left end point of the interval such that
    \begin{align*}
        f_{g}(x;\theta) \leq C_{1}(p_{1})
        \qquad
        \text{and}
        \qquad
        |f'_{g}(x;\theta)| \leq C_{2}(p_{1})
    \end{align*}
    for all $\theta>0$ and $x\in[p_{1},p_{2}]$.
\end{lemma}

\begin{proof}
    Applying the chain rule to the $\operatorname{Gamma}(1/2, \theta)$ density 
    \begin{align}
        f_{g}(x;\theta) 
        =
        \frac{1}
        {\Gamma(1/2)\theta^{1/2}}
        x^{-1/2}
        \exp(-x/\theta)
        \label{supp-eq:gamma-pdf}
    \end{align}
    yields
    \begin{align}
        f'_{g}(x;\theta)
        = 
        -
        \frac{f_{g}(x;\theta)}{2x}
        -
        \frac{f_{g}(x;\theta)}{\theta}.
        \label{supp-eq:gamma-derivative}
    \end{align}
    Differentiating $f_{g}(x;\theta)$ and $f_{g}(x;\theta)/\theta$ over $\theta>0$, respectively, and solving for zero and checking the second derivative yields that $f_{g}(x;\theta)$ and $f_{g}(x;\theta)/\theta$ are maximized at $\theta = 2x$ and $\theta = 2x/3$ respectively. Hence, plugging in the maximizers into \cref{supp-eq:gamma-pdf,supp-eq:gamma-derivative} and using the fact that $x\geq p_{1}$ yields
    \begin{align*}
        f_{g}(x;\theta)
        \leq
        \frac{\exp(-1/2)}{\Gamma(1/2)\sqrt{2}x}
        \leq \frac{\exp(-1/2)}{\Gamma(1/2)\sqrt{2}p_{1}} 
        \eqqcolon C_{1}(p_{1})
    \end{align*}
    and
    \begin{align*}
        |f'_{g}(x;\theta)|
        \leq 
        \alr{
            -
            \frac{f_{g}(x;\theta)}{2x}
        }
        +
        \alr{
            \frac{f_{g}(x;\theta)}{\theta}
        }
        &\leq
        \frac{\exp(-1/2)}{\Gamma(1/2)2^{3/2}x^2}
        +
        \frac{3^{3/2}\exp(-3/2)}{\Gamma(1/2)2^{3/2}x^2}
        \\
        &\leq
        \frac{\exp(-1/2)+3^{3/2}\exp(-3/2)}
        {\Gamma(1/2)2^{3/2}p_{1}^2}
        \\
        &\eqqcolon C_{2}(p_{1}),
    \end{align*}
    for all $x\in [p_{1},p_{2}]$ and $\theta > 0$.
    This completes the proof of \cref{supp-lemma:gamma-Lipschitz}.
\end{proof}

\cref{supp-lemma:Fourier-inversion-formula} provides a tool for expressing probability density functions in terms of their characteristic functions. Importantly, a generalization is made to incorporate the characteristic function of $W_{x}$, which is not absolutely integrable in general.

\begin{lemma}[Fourier inversion formula for distributions with Lipschitz densities]
\label{supp-lemma:Fourier-inversion-formula}
    Let $F(z)$ be a CDF with absolutely continuous probability density function $f(z)$ which is Lipschitz continuous on a closed interval $[a, b]$. Let $\varphi(s)$ denote the characteristic function associated with $F$. Then,
    \begin{align*}
        f(z) 
        =
        \lim_{T\to\infty}
        \frac{1}{2\pi}
        \int_{-T}^{T}
        \exp(-isz) \varphi(s) 
        \dd s
    \end{align*}
    for $z\in (a, b)$.
\end{lemma}

\begin{proof}
    Define the integral
    \begin{align*}
        I_{T}(z)
        \coloneqq 
        \frac{1}{2\pi}
        \int_{-T}^{T}
        \exp(-isz)
        \varphi(s)
        \dd s.
    \end{align*}
    Plugging in $\varphi(s) = \int_{-\infty}^{\infty}\exp(isy)f(y)\dd y$ to the above yields
    \begin{align}
        I_{T}(z)
        &=
        \frac{1}{2\pi}
        \int_{-T}^{T}
        \exp(-isz)
            \int_{-\infty}^{\infty}
            \exp(isy)f(y)
            \dd y
        \dd s
        \nonumber
        \\
        &=
        \frac{1}{2\pi}
        \int_{-\infty}^{\infty}
        f(y)
            \int_{-T}^{T}
            \exp(-is(z-y))
            \dd s
        \dd y
        \nonumber
        \\
        &=
        \frac{1}{\pi}
        \int_{-\infty}^{\infty}
        f(y)
        \clr{
            \frac{\exp(iT(z-y)) - \exp(-iT(z-y))}
            {2i(z-y)}
        }
        \dd y
        \nonumber
        \\
        &=
        \frac{1}{\pi}
        \int_{-\infty}^{\infty}
        f(y)
        \clr{
            \frac{\sin(T(z-y))}{z-y}
        }
        \dd y
        \label{supp-eq:inversion-integral-T}
    \end{align}
    On the other hand, using the Dirichlet integral identity $\int_{-\infty}^{\infty} \frac{\sin(T(z-y))}{(z-y)} \dd y = \pi$, the density function $f(z)$ can be expressed as 
    \begin{align}
        f(z)
        =
        \frac{1}{\pi}
        \int_{-\infty}^{\infty}
        f(z)
        \clr{
            \frac{\sin(T(z-y))}{z-y}
        }
        \dd y.
        \label{supp-eq:Dirichlet-integral-form}
    \end{align}
    By \cref{supp-eq:inversion-integral-T,supp-eq:Dirichlet-integral-form} it follows that
    \begin{align}
        f(z) - I_{T}(z)
        &= 
        \frac{1}{\pi}
        \int_{-\infty}^{\infty}
        \clr{
            \frac{f(z) - f(y)}{z-y}
        }
        \sin(T(z-y))
        \dd y
        \nonumber
        \\
        &=
        \frac{1}{\pi}
        \int_{a}^{b}
        \clr{
            \frac{f(z) - f(y)}{z-y}
        }
        \sin(T(z-y))
        \dd y
        \nonumber
        \\
        &\qquad
        +
        \frac{1}{\pi}
        \int_{\R\setminus[a,b]}
        \frac{f(z)}{z-y}
        \sin(T(z-y))
        \dd y
        -
        \frac{1}{\pi}
        \int_{\R\setminus[a,b]}
        \frac{f(y)}{z-y}
        \sin(T(z-y))
        \dd y.
        \label{supp-eq:integral-error}
    \end{align}
    We show that each of the three integrals in \cref{supp-eq:integral-error} converges to zero as $T\to\infty$. 
    First, since $f$ is Lipschitz on $[a,b]$ by definition, the function 
    $
        h_{1}(y)
        \coloneqq
        \Is{[a,b]}
        \clr{f(z) - f(y)}/\plr{z-y}
    $ 
    is bounded, and moreover, supported on a finite interval, hence absolutely integrable.
    Viewing the first integral as the imaginary part of the Fourier transformation of 
    $h_{1}(x)$, the Riemann-Lebesgue lemma yields
    \begin{align*}
        \lim_{T\to\infty}
        \frac{1}{\pi}
        \int_{-\infty}^{\infty}
        h_{1}(y)
        \sin(T(z-y))
        \dd y
        =
        0.
    \end{align*}
    
    The second integral can be written as
    \begin{align*}
        \frac{f(z)}{\pi}
        \int_{\R\setminus[a,b]}
        \frac{\sin(T(z-y))}{z-y}
        \dd y
        = 
        \frac{f(z)}{\pi}
        \int_{T(z-b)}^{\infty}
        \frac{\sin(u)}{u}
        \dd u
        + 
        \frac{f(z)}{\pi}
        \int^{-T(a-z)}_{-\infty}
        \frac{\sin(u)}{u}
        \dd u
        \to
        0
    \end{align*}
    for $z\in (a,b)$ as $T\to\infty$, since $\int_{-\infty}^{\infty} \frac{\sin(u)}{u} \dd u = \pi$ is a finite integral.

    In the third integral, the function $h_{2}(y)\coloneqq \I{(-\infty,b)\cup (a,\infty)} f(y)/(z-y)$ is bounded and absolutely integrable, since the denominator $z-y$ is bounded away from zero, and the probability density $f(y)$ is absolutely integrable. 
    Therefore, viewing the third integral as the imaginary part of the Fourier transform of $h_{2}(y)$, it vanishes as $T\to\infty$ by the Riemann-Lebesgue lemma.
    
    Combining the above results yields
    $
       \lim_{T\to \infty} \clr{f(z) - I_{T}(z)}
       = 
       0
    $ for $z \in (a,b)$, 
    which completes the proof of \cref{supp-lemma:Fourier-inversion-formula}.
\end{proof}

\subsubsection{Proof of \cref{supp-lemma:fW-convergence}}
    \begin{proof} 
    First, applying \cref{supp-eq:cgf2-r-gtr-1} to the definition of $w_{x}(x)$ in \cref{supp-eq:density-change-of-variables} yields
    \begin{align}
        w_{x}(x)
        \coloneqq
        \frac{x}{\sqrt{K_{\xp^2}''(\spt_{x})}} 
        &=
        \left[
            \frac{2}{\mul\left\{1- 2xg(x)/\mul\right\}^2}
            +
            \frac{1}{2x^2}
            \sum_{j=\mul + 1}^r 
            \frac{\lambda_j^2}{{(\lambda_{1}-\lambda_j \spt_{x})^2}}
        \right]^{-1/2}
        \nonumber \\
        &=
        \frac{\sqrt{\mul}\{1-2xg(x)/\mul\}}{\sqrt{2}}
        \left[
            {1 +\frac{\{1-2xg(x)/\mul\}^2}{4x^2}
            \sum_{j=\mul + 1}^r \frac{\lambda_j^2}{{(\lambda_{1}-\lambda_j \spt_{x})^2}}}
        \right]^{-1/2}
        \nonumber
        \\
        &= 
        \frac{\sqrt{\mul}\{1-2xg(x)/\mul \}}
        {\sqrt{2}\{1+q(x)\}^{1/2}}
        \label{supp-eq:w-reciprocal}
        \\
        &=
        \sqrt{\frac{\mul}{2}}
        \clr{1 + O(x^{-1})}
        ,
        \label{supp-eq:w-reciprocal-simplified}
    \end{align}
    where we have used the fact that $q(x) \coloneqq \frac{\{1-2xg(x)/\mul\}^2}{4x^2}\sum_{j=\mul + 1}^r \frac{\lambda_j^2}{{(\lambda_{1}-\lambda_j \spt_{x})^2}}$ satisfies $q(x)=O(x^{-2})$, since $g(x)=O(x^{-2})$ from \cref{supp-lemma:saddle-point}. Next, the explicit form of the characteristic function of $W_x$ is obtained via \cref{supp-eq:transformed-cf}. By \cref{supp-eq:mgf-cgf-forms-r-gtr-1}, it holds that
    \begin{align*}
        M_{\xp^2}
        \left(
        {\spt_{x}}+\frac{\ir s}{\sqrt{K_{\xp^2}''(\spt_{x})}}
        \right) 
        =
        \left[
            \frac{1}{2x}\{\mul-2xg(x)\}
            -
            \frac{\ir s}{\sqrt{K_{\xp^2}''(\spt_{x})}} 
        \right]^{-\mul /2}
        \prod_{j=\mul + 1}^r
        \left(
            1-\frac{\lambda_j}{\lambda_{1}} \spt_{x}-\frac{\ir\lambda_j s}{\lambda_{1}\sqrt{K_{X^2}''(\spt_{x})}}
        \right)^{-1/2}.
    \end{align*}
    Hence, by \cref{supp-eq:transformed-cf,supp-eq:w-reciprocal} it follows that
    \begin{align}
        \varphi_{W_x}(s)
        &=
        \frac{
            M_{\xp^2}
            \left(
                {\spt_{x}}+\frac{\ir s}{\sqrt{K_{\xp^2}''(\spt_{x})}}
            \right)
        }
        {
        M_{\xp^2}(\spt_{x})
        }
        \nonumber\\
        &=
        \left[
            1 
            -
            \frac{2\ir sx}{\sqrt{K_{\xp^2}''(\spt_{x})}}
            \frac{1}{\{\mul - 2xg(x)\}}
        \right]^{-\mul/2}
        \frac{
        \prod_{j=\mul + 1}^r
        \left(
            1-\frac{\lambda_j}{\lambda_{1}} \spt_{x}-\frac{\ir \lambda_j s}{\lambda_{1}\sqrt{K_{\xp^2}''(\spt_{x})}}
        \right)^{-1/2}
        }
        {
        \prod_{j=\mul + 1}^r 
        \plr{1-\frac{\lambda_j}{\lambda_{1}} \spt_{x}}^{-1/2}
        }
        \nonumber \\
        &= 
        \blr{
            1 - \frac{\sqrt{2}\ir s}{\sqrt{\mul}\{1+q(x)\}^{1/2}}
        }^{-\mul/2}
        \prod_{j=\mul + 1}^{r}
        \left(
            1
            -
            \frac{\ir \lambda_j s}
            {(\lambda_{1}- \lambda_j \spt_{x}) 
            \sqrt{K_{\xp^2}''(\spt_{x})}}
        \right)^{-1/2}.
        \label{supp-eq:cf-W}
    \end{align}
    Furthermore, since $q(x)= O(x^{-2})$, $\sqrt{K''_{X^2}(\spt_{x})}= O(x)$, and ${\lambda_j s}/{(\lambda_{1}-\lambda_j \spt_{x})} = O(1)$ for any fixed $s\in\R$ and $j= \mul +1,\dots, r$ as $x$ grows, \cref{supp-eq:cf-W} satisfies
    \begin{align}
        \varphi_{W_x}(s)
        =
        \left(1 - {\ir s\sqrt{2/\mul}}\right)^{-\mul/2}
        (1 + O(x^{-1})) 
        =
        \varphi_{W}(s)
        (1 + O(x^{-1}))
        \label{supp-eq:cf-delta}
    \end{align} 
    for all $s\in\R$, where $W$ is a limiting random variable distributed as $\operatorname{Gamma}(\mul/2, \sqrt{2/\mul})$.
    
    Shifting the attention to density functions, note that the form of $\varphi_{W_{x}}(s)$ in \cref{supp-eq:cf-W} implies $W_{x} = \sum_{j=1}^r{\omega}_{j,x}$, where the random variables $\omega_{1,x},\dots, \omega_{r,x}$ are independently distributed as 
    \begin{align*}
        \omega_{j,x} 
        \sim
        \begin{cases}
            \operatorname{Gamma}
            \left(
                \frac{1}{2}
                ,
                \frac{\sqrt{2}}{\sqrt{\mul}\{1+q(x)\}^{1/2}}
            \right)
            & j=1,\dots, \mul
            \\
            \operatorname{Gamma}
            \left(
                \frac{1}{2}
                ,
                \frac{\lambda_j}{(1-\lambda_j\spt_{x})\sqrt{K''_{X^2}(\spt_{x})}}
            \right) 
            & j=\mul + 1,\dots,r.
        \end{cases} 
    \end{align*}
    Since a finite convolution of absolutely continuous density functions is absolutely continuous, the random variable $W_{x}$ has an absolutely continuous density function $f_{W_{x}}$. 
    Furthermore, define $S_{x} \coloneqq \sum_{j=2}^r{\omega}_{j,x}$, and the interval $[p_{1}, p_{2}]$, where $0 < p_{1} < \sqrt{2/\mul} < p_{2} < \infty$.
    Then, $f_{W_{x}} = f_{\omega_{1,x}} * f_{S_{x}}$, where $*$ denotes the convolution operation defined by 
    \begin{align}
        (f_{\omega_{1,x}}*f_{S_{x}})(s)
        \coloneqq
        \int_{-\infty}^{\infty} f_{\omega_{1,x}}(t)f_{S_{x}}(s-t)\dd t. 
        \label{supp-eq:convolution}
    \end{align}
    By \cref{supp-lemma:gamma-Lipschitz}, it holds that
    \begin{align*}
        f_{W_{x}}(z) 
        =
        (f_{\omega_{1,x}} * f_{S_{x}}) (z)
        =
        \int_{-\infty}^{\infty} f_{\omega_{1,x}}(t)f_{S_{x}}(z-t)\dd t
        \leq
        C_{1}(p_{1})\int_{-\infty}^{\infty} f_{S_{x}}(z-t)\dd t
        =
        C_{1}(p_{1}),
    \end{align*}
    for all $z\in \R_{\geq 0}$ and $x \in \R$ since $C_{1}(p_{1})$ does not depend on $x$.
    
    In addition, since both functions in the convolution in \cref{supp-eq:convolution} are density functions that integrate to one over $\R$, by the differentiation property of convolutions along with \cref{supp-lemma:gamma-Lipschitz} and Jensen's inequality, it holds that
    \begin{align*}
        |f'_{W_{x}}(z)|
        =
        |(f'_{\omega_{1,x}}*f_{S_{x}})(z)|
        =
        \alr{
            \int_{-\infty}^{\infty} f'_{\omega_{1,x}}(t)f_{S_x}(z-t)\dd t
        }
        \leq
        C_{2}(p_{1})
        \int_{-\infty}^{\infty} f_{S_x}(z-t)\dd t
        = C_{2}(p_{1})
    \end{align*}
    for $z\in[p_{1}, p_{2}]$.

    Since $f_{W_{x}}$ and $f_{W}$ are absolutely continuous, bounded, and Lipschitz on the interval $[p_{1}, p_{2}]$, with the upper bound and Lipschitz constant independent of $x$, \cref{supp-lemma:Fourier-inversion-formula} yields
    \begin{align*}
        f_{W_{x}}(z) = \frac{1}{2\pi}\lim_{T\to\infty}\int_{-T}^{T}\exp(-isz)\varphi_{W_{x}}(s) \dd s
        ,
        \quad
        \text{ and }
        \quad
        f_{W}(z) = \frac{1}{2\pi}\lim_{T\to\infty}\int_{-T}^{T}\exp(-isz)\varphi_{W}(s) \dd s
    \end{align*}
    for $z\in (p_{1}, p_{2})$. 
    Thus, by \cref{supp-eq:cf-delta}, there exists a constant $C>0$ such that
    \begin{align}
        f_{W_{x}}(z) - f_{W}(z)
        &= 
        \lim_{T\to\infty}
        \frac{1}{2\pi}
        \int_{-T}^{T}
        \exp(-isz)
        \clr{
            \varphi_{W_{x}}(s) - \varphi_{W}(s)
        }
        \dd s
        \nonumber
        \\
        &=
        \lim_{T\to\infty}
        \frac{1}{2\pi}
        \int_{-T}^{T}
        \exp(-isz)
        \varphi_{W}(s)
        \frac{C}{x}
        \dd s
        \nonumber
        \\
        &=
        \frac{C}{x}
        f_{W}(z)
        \label{supp-eq:pdf-converge-rate-pointwise}
    \end{align}
    for $z\in (p_{1}, p_{2})$.
    By \cref{supp-eq:w-reciprocal,supp-eq:pdf-converge-rate-pointwise}, there exists a constant $C_{0}>0$ such that for $x$ large enough, 
    \begin{align}
        f_{W_{x}}(w_{x}(x)) - f_{W}(\sqrt{\mul/2}) 
        =
        \clr{
            f_{W_{x}}(w_{x}(x)) - f_{W}(w_{x}(x))
        }
        +
        \clr{
            f_{W}(w_{x}(x)) - f_{W}(\sqrt{\mul/2})
        }
        \leq
        \frac{C_{0}}{x},
        \label{supp-eq:pdf-convergence-rate-final}
    \end{align}
    where we have used the fact $w_{x}(x) \in (p_{1}, p_{2})$ for $x$ large enough, and that $|f_{W}(w_{x}(x)) - f_{W}(\sqrt{\mul/2})|\leq C_{1} |w_{x}(x)- \sqrt{\mul/2}| = O(x^{-1})$ for some constant $C_{1}$, since the $\operatorname{Gamma}(\mul/2, \sqrt{2/\mul})$ density $f_{W}$ is Lipschitz on $[p_{1},p_{2}]$.
    
    By scrutinizing the $\operatorname{Gamma}(\mul/2, \sqrt{2/\mul})$ density, it holds that
    $f_{W }
        \left(
            \sqrt{\mul / 2}
        \right)
        =
        \frac{
            (\mul/2)^{\frac{\mul-1}{2}}
            \exp(-\mul/2)
        }
        {
        \Gamma(\mul/2)
        }
    $, which alongside \cref{supp-eq:pdf-convergence-rate-final}, implies that 
    \begin{align*}
        f_{W_{x}}(w_{x}(x)) 
        =
        \clr{
            \frac{
            (\mul/2)^{\frac{\mul-1}{2}}
            \exp(-\mul/2)
            }
            {
            \Gamma(\mul/2)
            }
        }
        \clr{
            1 + O(x^{-1})   
        }
    \end{align*}
    for $x$ large enough.
    This completes the proof of \cref{supp-lemma:fW-convergence}.
\end{proof}

\section{Proofs of auxiliary lemmas for \cref{\mainlabel{body-theorem:Gumbel-convergence}}}

\subsection{Proof of \cref{supp-lemma:survival-function}}
\label{supp-pf:survival-function}

\begin{proof}
    First, suppose $x >0$. By integration by parts, $\overline{F}_{H}(x)$ has the form
    \begin{align}
        \overline{F}_{H}(x) 
        &=
        \int_{x}^{\infty}\frac{2}{\lambda_{1}^{\mul/2}\Gamma(\mul/2)}u^{\mul-1}\exp(-u^{2}/\lambda_{1}) \dd u
        \nonumber\\
        &=
        \frac{2}{\lambda_{1}^{\mul/2}\Gamma(\mul/2)}
        \clr{\frac{\lambda_{1}}{2} x^{\mul-2} \exp(-x^2/\lambda_{1}) 
        +
        \frac{\lambda_{1}(\mul-2)}{2} \int_{x}^{\infty} u^{\mul-3}\exp(-u^{2} /\lambda_{1})\dd u}.
        \label{supp-eq:IBP-1}
    \end{align}
    Using the fact that $u\geq x$ in the integral in \cref{supp-eq:IBP-1}, and applying integration by parts once more yields
    \begin{align}
        \int_{x}^{\infty} u^{\mul-3}\exp(-u^{2}/\lambda_{1}) \dd u
        &\leq
        \frac{1}{x^{2}}\int_{x}^{\infty} u^{\mul-1}\exp(-u^{2}/\lambda_{1})\dd u
        \nonumber\\
        &= \frac{1}{x^2} \clr{\frac{\lambda_{1}}{2}x^{\mul-2} \exp(-x^2/\lambda_{1}) +\frac{\lambda_{1}(\mul-2)}{2} \int_{x}^{\infty} u^{\mul-3} \exp(-u^2/\lambda_{1})\dd u}.
        \nonumber 
    \end{align}
    Combining the integrals yields
    \begin{align*}
        \clr{1-\frac{\lambda_{1}(\mul-2)}{2x^{2}}} \int_{x}^{\infty} u^{\mul-3}\exp(-u^{2}/\lambda_{1}) \dd u
        \leq
        \frac{1}{x^2} \clr{\frac{\lambda_{1}}{2}x^{\mul-2} \exp(-x^2/\lambda_{1})},
    \end{align*}
    that is,
    \begin{align*}
        \frac{\int_{x}^{\infty} u^{\mul-3}\exp(-u^{2}/\lambda_{1}) \dd u}
        {\frac{\lambda_{1}}{2}x^{\mul-2} \exp(-x^2/\lambda_{1})}
        \leq
        \frac{1}{x^2 - \frac{\lambda_{1} (\mul -2)}{2}}.
    \end{align*}
    Therefore, taking out the factor $\frac{\lambda_{1}}{2}x^{\mul-2} \exp(-x^2/\lambda_{1})$ in \cref{supp-eq:IBP-1} yields
    \begin{align}
        \overline{F}_{H}(x) 
        &=
        \frac{\lambda_{1}^{1-\mul/2}}{\Gamma(\mul/2)}x^{\mul-2}\exp(-x^2/\lambda_{1})(1 + r(x)),
        \label{supp-eq:survival-function}
    \end{align}
    where 
    \begin{align*}
        r(x) 
        \coloneqq
        \frac{\lambda_{1}(\mul - 2)}{2}
        \frac{\int_{x}^{\infty} u^{\mul-3}\exp(-u^{2}/\lambda_{1}) \dd u}
        {\frac{\lambda_{1}}{2}x^{\mul-2} \exp(-x^2/\lambda_{1})}
        \in
        \left(
            0,
            \;
            \frac{\lambda_{1}(\mul -2)}{2x^2 - \lambda_{1}(\mul -2)}
        \right].        
    \end{align*}
    If $x\leq 0$, then $\overline{F}_{H}(x) = 1$, since $H$ is non-negative with probability one.
\end{proof}

\subsection{Proof of \cref{supp-lemma:generalized-gamma-normalizing-constants}}
\label{supp-pf:generalized-gamma-normalizing-constants}

\begin{proof}
    As discussed in \cref{\mainlabel{body-remark:scaled-chi-properties}}, the generalized gamma distribution has a scaling property with respect to the scale parameter, namely, $H \eqd \sqrt{\lambda_{1}} H_{1}$. Thus, for i.i.d.~random variables $(H_{1,i})_{i\in\dpar{n}}$ equal in distribution to $H_{1}$ and i.i.d.~random variables $(H_{i})_{i\in\dpar{n}}$ equal in distribution to $H$, it holds that
    \begin{align}
        \alpha_{n}\inv\plr{\max_{i\in\dpar{n}}H_{1,i} -\beta_{n}}
        \eqd
        \plr{\sqrt{\lambda_{1}}\alpha_{n}}\inv\plr{\max_{i\in\dpar{n}}H_{i} -\sqrt{\lambda_{1}}\beta_{n}}.
        \label{supp-eq:generalized-gamma-scaling-property}
    \end{align}
    Thus, it suffices to first derive the normalizing sequences for $H_{1}$, and then to scale the sequences by $\sqrt{\lambda_{1}}$ to obtain the normalizing sequences for $H$.
    
    \noindent
    \textbf{(i)} (Gumbel domain of attraction). 
    The maximum domain of attraction of $H_{1}$ (more generally, of a continuous random variable a with a differentiable probability density function) is closely related to the ratio 
    \begin{align}
        \frac{\overline{F}_{H_{1}}(x)}{f_{H_{1}}(x)} 
        =
        \frac{\frac{1}{\Gamma(\mul/2)}x^{\mul-2}\exp(-x^2)(1 + r(x))}{\frac{2}{\Gamma(\mul/2)}x^{\mul-1}\exp(-x^{2})}
        =
        \frac{1}{2x}(1 + r(x)),
        \label{supp-eq:auxiliary-function}
    \end{align}
    which is due to \cref{supp-lemma:survival-function} with $\lambda_{1}$ replaced with $1$.
    The derivative with respect to $x$ of this ratio is
    \begin{align*}
       \plr{\frac{\overline{F}_{H_{1}}(x)}{f_{H_{1}}(x)}}' 
        = 
        -\frac{1}{2x^{2}} + O(x^{-4})\to 0, 
        \; \text{ as } \; x\to\infty,
    \end{align*}
    hence, by \citet[Theorem 1.1.8]{de_haan_extreme_2010}, $H_{1}$ is in the maximum domain attraction of the Gumbel distribution.

    \noindent
    \textbf{(ii)} (Normalizing sequences). By \citet[Theorem 1.1.8]{de_haan_extreme_2010}, the normalizing sequences for $\max_{i\in\dpar{n}}H_{1,i}$ are obtained by solving the equation
    \begin{align}
        n\overline{F}_{H_{1}}(b_{n,*}) 
        =
        1
        \label{supp-eq:quantile-equation}
    \end{align}
    and then computing $a_{n,*} = {\overline{F}_{H_{1}}(b_{n,*})}/{f_{H_{1}}(b_{n,*})}$.
    The exact solution of \cref{supp-eq:quantile-equation} is difficult to solve analytically, hence we aim to find an approximate solution $b_{n}$ that satisfies $(b_{n}-b_{n,*})/a_{n,*}\to 0$ as $n\to\infty$.
    First, the equation $n\overline{F}_{H_{1}}(b_{n,*}) =1$ can be written as
    \begin{align*}
        &b_{n,*}^{p-2}\exp(-b_{n,*}^{2})(1+r(b_{n,*})) 
        = 
        n^{-1}\Gamma(\mul/2).
    \end{align*}
    Taking logarithms on each side yields
    \begin{align}
        b_{n,*}^{2}- \frac{\mul-2}{2}\log b_{n,*}^{2} 
        =
        \log n -\log\Gamma(\mul/2) + \log(1 + r(b_{n,*})).
        \label{supp-eq:quantile-equation-log}
    \end{align}
    For convenience of notation, let the difference of the first two terms on the right-hand side of \cref{supp-eq:quantile-equation-log} be denoted as $L_{n} \coloneqq \log n -\log\Gamma(\mul/2)$. Replacing $\log b_{n,*}^{2}$ on the left-hand side with $\log L_{n}$, and dropping the term $\log(1-r(b_{n,*}))$ on the right hand side yields an approximation of $b_{n,*}$ given by
    \begin{align}
        \widehat{b}_{n}^{2} 
        =
        L_{n} + \frac{\mul-2}{2}\log L_{n}.
        \label{supp-eq:quantile-equation-solution}
    \end{align}
    Taking square roots yields
    \begin{align}
        \widehat{b}_{n} 
        &=
        \plr{L_{n} + \frac{\mul-2}{2}\log L_{n}}^{1/2}
        \nonumber\\
        &=
        \sqrt{L_{n}}\plr{1+\frac{\mul-2}{2}\frac{\log L_{n}}{L_{n}}}^{1/2}
        \nonumber\\
        &=
        \sqrt{L_{n}} + \frac{\mul-2}{4} \frac{\log L_{n}}{\sqrt{L_{n}}} + O\plr{\frac{\log^2 L_{n}}{L_{n}^{3/2}}}
        \nonumber\\
        &= 
        \sqrt{\log n} 
        - \frac{\log\Gamma(\mul/2)}{2\sqrt{\log n}} 
        + \frac{(\mul-2)\log\log n}{4\sqrt{\log n}}
        + O\plr{\frac{(\log\log n)^2}{(\log n)^{3/2}}},
        \label{supp-eq:bn-hat}
    \end{align}
    where the penultimate line is obtained by taking the Taylor series expansion of $\sqrt{1+y}$ around $y=0$. The last line is from 
    \begin{align*}
        \sqrt{L_{n}} 
        &=
        \sqrt{\log n - \log \Gamma(\mul/2)}
        \\
        &=
        \sqrt{\log n}\plr{1 - \frac{\log\Gamma(\mul/2)}{\log n}}^{1/2}
        \\
        &=
        \sqrt{\log n} - \frac{\log\Gamma(\mul/2)}{2\sqrt{\log n}} + O\plr{(\log n)^{-3/2}},
    \end{align*}
    which is obtained by taking the Taylor series expansion of $\sqrt{1-y}$ around $y=0$, and
    \begin{align*}
        \log L_{n} 
        &=
        \log(\log n - \log \Gamma(\mul/2))
        \\
        &= 
        \log\log n + \log \plr{1-\frac{\log\Gamma(\mul/2)}{\log n}}
        \\
        &= 
        \log\log n -\frac{\log\Gamma(\mul/2)}{\log n} + O\plr{(\log n)^{-2}},
    \end{align*}
    which is obtained by taking the Taylor series expansion of $\log(1-x)$ around $x=0$.

    For the scaling sequence $a_{n}$, an approximation of $a_{n,*}$ is obtained by plugging $\widehat{b}_{n}$ into \cref{supp-eq:auxiliary-function}, namely,
    \begin{align}
        \widehat{a}_{n} 
        = 
        \frac{\overline{F}_{H_{1}}(\widehat{b}_{n})}{f_{H_{1}}(\widehat{b}_{n})} 
        =
        \frac{1}{2\widehat{b}_{n}}(1+r(\widehat{b}_{n}))
        =
        \frac{1}{2\sqrt{\log n}} + O((\log n)^{-1}).
        \label{supp-eq:an-hat}
    \end{align}

    \noindent
    \textbf{(iii)} (Approximation error).
    To study the approximation error of $\widehat{b}_{n}$, define the function 
    \begin{align*}
        g_{n}(b)
        \coloneqq
        b^2 -\frac{p-2}{2}\log b^2 - L_{n} - \log(1+r(b)), 
    \end{align*}
    which is the residual of \cref{supp-eq:quantile-equation-log} at a given value $b$. Since $r(x) = O(x^{-2})$, it can be seen that $\log(1 + r(b)) = O(b^{-2})$ by taking the Taylor series expansion of $\log (1+x)$ around $x=0$. The residual for $\widehat{b}_{n}$ satisfies
    \begin{align}
        g_{n}(\widehat{b}_{n}) 
        &=
        \frac{\mul-2}{2}\clr{\log L_{n} - \log\plr{L_{n} - \frac{\mul-2}{2}\log L_{n}}}
        \nonumber\\
        &=\frac{\mul-2}{2}\clr{-\log\plr{1 - \frac{\mul-2}{2}\frac{\log L_{n}}{L_{n}}}}
        \nonumber\\
        &= -\frac{(\mul-2)^2}{4} \frac{\log L_{n}}{L_{n}} + O\plr{\plr{\frac{\log L_{n}}{L_{n}}}^{2}}
        \nonumber\\
        &= O\plr{\frac{\log \log n}{\log n}}
        \label{supp-eq:quantile-equation-residual-log}
    \end{align}
    since $L_{n} = O(\log n)$. Here $g_{n}(b_{n,*})=0$, hence the mean value theorem yields
    \begin{align*}
        g_{n}(\widehat{b}_{n}) 
        =
        g'_{n}(\xi) (\widehat{b}_{n} - b_{n,*}), \quad \xi\in (\min\{\widehat{b}_{n}, b_{n,*}\},\max\{\widehat{b}_{n}, b_{n,*}\})
    \end{align*}
    which implies that
    \begin{align*}
        \delta_{b}
        \coloneqq
        \widehat{b}_{n} - b_{n,*} 
        =
        O\plr{\frac{\log \log n}{(\log n)^{3/2}}},
    \end{align*}
    since $b'(\xi)=O(\sqrt{\log n})$.
    Furthermore, \cref{supp-eq:quantile-equation-residual-log} yields $\log(n\overline{F}_{H_{1}}(\widehat{b}_{n})) = O(\log\log n/ \log n)$, which implies that
    \begin{align}
        n\overline{F}_{H_{1}}(\widehat{b}_{n}) 
        =
        \exp\plr{O\plr{\frac{\log\log n}{\log n}}} 
        =
        1 + O\plr{\frac{\log\log n}{\log n}}.
        \label{supp-eq:quantile-equation-residual}
    \end{align}
    For the approximation error of $\widehat{a}_{n}$, the function
    \begin{equation*}
        A(x) 
        \coloneqq
        \frac{\overline{F}_{H_{1}}(x)}{f_{H_{1}}(x)} 
        =
        \frac{1}{2x}(1+r(x))
    \end{equation*}
    has bounded slope on a domain bounded away from zero. Since both $\widetilde{b}_{n}$ and $b_{n,*}$ are growing at rate $\log n$, $A(\widehat{b}_{n}) = \widehat{a}_{n}$ and $A(b_{n,*}) = a_{n,*}$, it holds that for $n$ large enough, 
    \begin{equation*}
        \frac{|\widehat{a}_{n} - a_{n,*}|}{|\widehat{b}_{n}-b_{n,*}|}\leq 1,
    \end{equation*}
    which implies that
    \begin{equation*}
        \delta_{a}
        \coloneqq
        \widehat{a}_{n} - a_{n,*} 
        =
        O\plr{\frac{\log \log n}{(\log n)^{3/2}}},
    \end{equation*}
    as $n\to\infty$. In conclusion, by Slutsky's lemma,
    \begin{align*}
        \widehat{a}_{n}\inv
        \plr{
            \max_{i\in\dpar{n}}H_{1,i} - \widehat{b}_{n}
        } 
        =
        a_{n,*}\inv
        \plr{
            \max_{i\in\dpar{n}}H_{1,i} - {b}_{n,*}
        } 
        \plr{
            \frac{1}{1+\frac{\delta_{a}}{a_{n,*}}}
        } 
        +\frac{\delta_{b}}{\widehat{a}_{n}} \rightsquigarrow \gumbelrv,
    \end{align*}
    since ${\delta_{a}}/{a_{n,*}} = O(\log\log n / \log n)$ and ${\delta_{b}}/{\widehat{a}_{n}} = O(\log\log n / \log n)$. 
    By the same reasoning, keeping only the largest order term in $\widehat{a}_{n}$ and discarding the $o(\widehat{a}_{n})$ terms in $\widehat{b}_{n}$ does not affect the convergence in distribution. Thus, scaling both sequences by $\sqrt{\lambda_{1}}$ yields the final sequences
    \begin{align}
        a_{n}
        =
        \frac{\sqrt{\lambda_{1}}}{2\sqrt{\log n}},
        \qquad
        b_{n}
        = 
        \sqrt{\lambda_{1}\log n} 
        + \frac{\sqrt{\lambda_{1}}(\mul-2)\log\log n}{4\sqrt{\log n}}
        - \frac{\sqrt{\lambda_{1}}\log\Gamma(\mul/2)}{2\sqrt{\log n}}.
    \end{align}
    Lastly, using the fact that
    \begin{align*}
        \overline{F}_{H}(x) 
        =
        \P(H > x)
        =
        \P
        \plr{
            \frac{H}{\sqrt{\lambda_{1}}} > \frac{x}{\sqrt{\lambda_{1}}}
        }
        =
        \overline{F}_{H_{1}}
        \plr{
            \frac{x}{\sqrt{\lambda_{1}}}
        },
    \end{align*}
    \cref{supp-eq:quantile-equation-residual} directly yields
    \begin{align}
        n\overline{F}_{H}({b}_{n})
        =
        n\overline{F}_{H_{1,\mul}}(\widehat{b}_{n})
        =
        1 + O\plr{\frac{\log\log n}{\log n}},
        \label{supp-eq:quantile-equation-residual-scaled}
    \end{align}
    which completes the proof of \cref{supp-lemma:generalized-gamma-normalizing-constants}.
\end{proof}

\subsection{Proof of \cref{supp-lemma:first-order-convergence}}
\label{supp-pf:first-order-convergence}
\begin{proof} 
    Let $\m{\Psi}\coloneqq \mh{U}\RU- \m{U} - \m{EVS}\inv$.
    By the triangle inequality and the reverse triangle inequality,
    \begin{align}
    \label{supp-eq:squeeze}
       \| \m{EVS}\inv\|_{2,\infty} - \|\m{\Psi}\|_{2,\infty}
       \leq
       \tinorm{\m{EVS}\inv + \m{\Psi}} 
       \leq 
       \| \m{EVS}\inv\|_{2,\infty} + \|\m{\Psi}\|_{2,\infty}.
    \end{align}
    In particular,
    \begin{gather*}
        \tinorm{\mh{U}\RU- \m{U} }
        = 
        \tinorm{\m{EVS}\inv + \m{\Psi}} 
        = 
        \| \m{EVS}\inv\|_{2,\infty}+\widetilde{\gamma}_{n}, 
    \end{gather*}
    where $|\widetilde{\gamma}_{n}|\leq \|\m{\Psi}\|_{2,\infty}$.
    Letting $\gamma_{n} \coloneqq a_{n}\inv \widetilde{\gamma}_{n}$, it holds that
    \begin{align*}
        a_{n}^{-1}\plr{\|\mh{U}\RU- \m{U} \|_{2,\infty} -b_{n}}
        &=
        a_{n}^{-1}\plr{\|\m{EVS}\inv\|_{2,\infty} +\widetilde{\gamma_{n}}-b_{n}}
        \\
        &= 
        a_{n}^{-1}\plr{\|\m{EVS}\inv\|_{2,\infty}-b_{n}} 
        +
        \gamma_{n}.
    \end{align*}
    By \cref{\mainlabel{body-lemma:first-order-approximation}} and recalling $a_{n} = O(\sigma/(s_r\sqrt{\log n}))$, it holds under \cref{\mainlabel{body-assumption:noise},\mainlabel{body-assumption:matrix-size},\mainlabel{body-assumption:delocalization},\mainlabel{body-assumption:snr}} that there exists a constant $C_{\gamma}>0$ such that
    \begin{align*}
        |\gamma_{n}|
        =
        a_{n}\inv |\widetilde{\gamma}_{n}|
        \leq
        a_{n}\inv \tinorm{\m{\Psi}} 
        \leq
        C_{\gamma}
        \clr{
            \sqrt{\rlog}
            \plr{
                \frac{\sigma\sqrt{n}}
                {s_{r}}
                +
                \sqrt{
                    \frac{\mu r}{n}
                }
            }
            +
            \frac{\sigma\sqrt{\mu r n\log n}}
            {s_{r}}
        },
    \end{align*}
    with probability exceeding $1-O(n^{-9})$.
    This completes the proof of \cref{supp-lemma:first-order-convergence}.
\end{proof}

\section{Additional material for non-Gaussian noise distributions}
\label{supp-section:additional-experiments-non-Gaussian}

\subsection{Balanced two-block stochastic blockmodel}
\label{supp-section:balanced-two-block-SBM}

In this section, the behavior of the statistic $\tstat$ is investigated when the underlying low-rank signal matrix has balanced two-block signal structure, namely
\begin{align*}
    \m{M} 
    =
    \m{Z}\m{B}\m{Z}\T
    \in
    \R^{n \times n},
    \qquad
    \text{where}
    \qquad
    \m{Z}
    = 
    \begin{bmatrix}
        \1_{n/2} & \m{0}
        \\ 
        \m{0} & \1_{n/2}
    \end{bmatrix}
    \in
    \R^{n \times 2},
    \qquad
    \m{B} 
    =
    \begin{bmatrix}
        p & q
        \\ 
        q & p
    \end{bmatrix}
    \in \R^{2 \times 2}
    .
\end{align*}
In the context of stochastic blockmodels (SBMs) \citep{holland_stochastic_1983}, the matrix $\m{Z}$ represents the block membership matrix, where each row of $\m{Z}$ indicates the membership of the corresponding node in the network, while the entries $p$ and $q$ in $\m{B}$ denote the within and between block connectivity parameters, respectively. 
In \cref{supp-fig:SBM1,supp-fig:SBM2,supp-fig:SBM3}, the values of $p$ are chosen as $p\in\{\log n/ \sqrt{n}, \sqrt{\log / n}, 1/\sqrt{n}\}$, respectively, and the value of $q$ is set as $q = p/\log n$.
The three choices of $p$ are labeled as ``strong signal strength'', ``medium signal strength'', and ``weak signal strength'', respectively.

For $\m{M}$ as above, three different types of data matrices are simulated: a Bernoulli SBM adjacency matrix, a Poisson SBM adjacency matrix, and a Gaussian SBM adjacency matrix.
The Bernoulli SBM adjacency matrix $\mh{M}_{1}$ is obtained by simulating within-group entries from $\ber{p}$, and between-group entries from $\ber{q}$. 
The corresponding error matrix is $\m{E}_{1} = \mh{M}_{1} - \m{M}$, whose within-group entries have variance $p(1-p)$, and between-group entries have variance $q(1-q)$.
The Poisson SBM adjacency matrix $\mh{M}_{2}$ is obtained by simulating within-group entries from $\pois{p}$, and between-group entries from $\pois{q}$.
The corresponding error matrix is $\m{E}_{2} = \mh{M}_{2} - \m{M}$, whose within-group entries have variance $p$, and between-group entries have variance $q$.
The Gaussian SBM error matrix $\m{E}_{3}$ is generated so that the first two moments of the entries match those of $\m{E}_{1}$, namely the within-group entries have distribution $\NN(0,p(1-p))$ and the between-group entries have distribution $\NN(0,q(1-q))$.
To be clear, the noise matrices $\m{E}_{1}$, $\m{E}_{2}$, and $\m{E}_{3}$ here do not satisfy \cref{\mainlabel{body-assumption:noise}} in the main text.

In \cref{supp-fig:SBM1,supp-fig:SBM2,supp-fig:SBM3}, each column presents a histogram and a corresponding quantile-quantile (QQ) plot of $\tstat$.
Here, $n = 3000$ across all figures, while the signal strength varies based on the choice of parameters $p$ and $q$.
\cref{supp-fig:SBM1,supp-fig:SBM2,supp-fig:SBM3} represent strong, medium, and weak signals, respectively.
In each setting, the noise matrix $\m{E}$ is independently simulated 1800 times to create a random sample of $\mh{M}$, keeping $\m{M}$ fixed.

The simulations show that $\tstat$ behaves similarly under the Bernoulli and Poisson SBMs.
In both cases, $\tstat$ aligns reasonably well with the standard Gumbel distribution when the signal is strong, as shown in \cref{supp-fig:SBM1}.
However, for moderate or weak signal strength (\cref{supp-fig:SBM2,supp-fig:SBM3}), there is a dramatic discrepancy compared to the standard Gumbel distribution.
In contrast, under the Gaussian SBM (right-most column of \cref{supp-fig:SBM1,supp-fig:SBM2,supp-fig:SBM3}), $\tstat$ exhibits a less dramatic departure from the standard Gumbel distribution, even in the moderate and weak signal regimes.

\begin{figure}[t]
    \centering
    \begin{minipage}{.33\textwidth}
        \includegraphics[width=5.9cm]{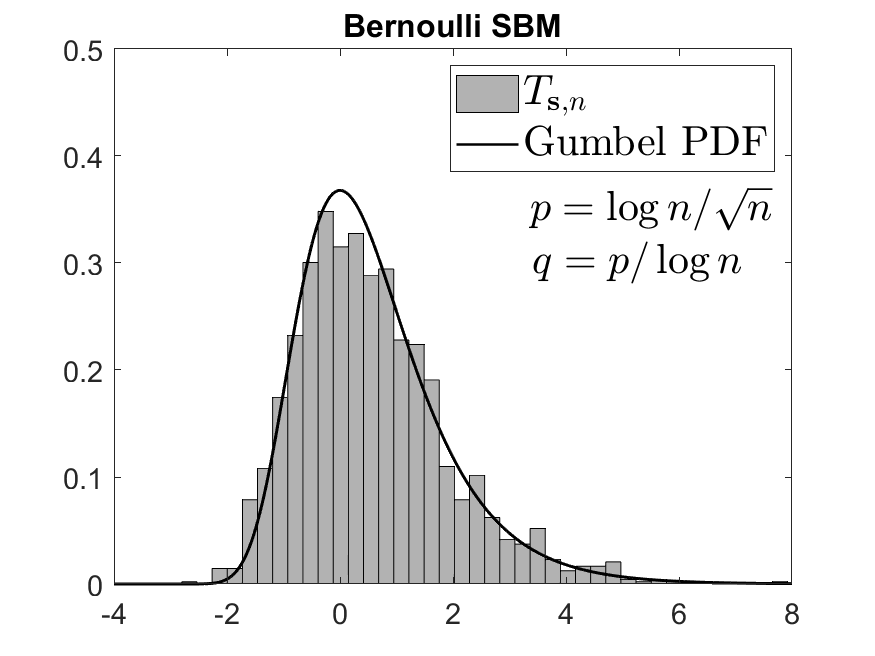}
    \end{minipage}%
    \begin{minipage}{.33\textwidth}
        \includegraphics[width=5.9cm]{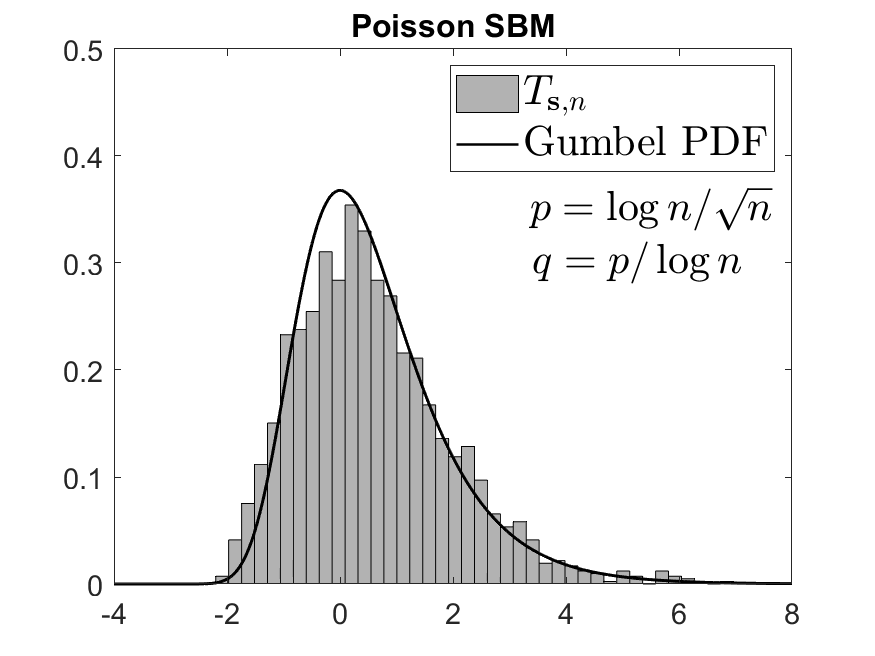}
    \end{minipage}%
    \begin{minipage}{.33\textwidth}
        \includegraphics[width=5.9cm]{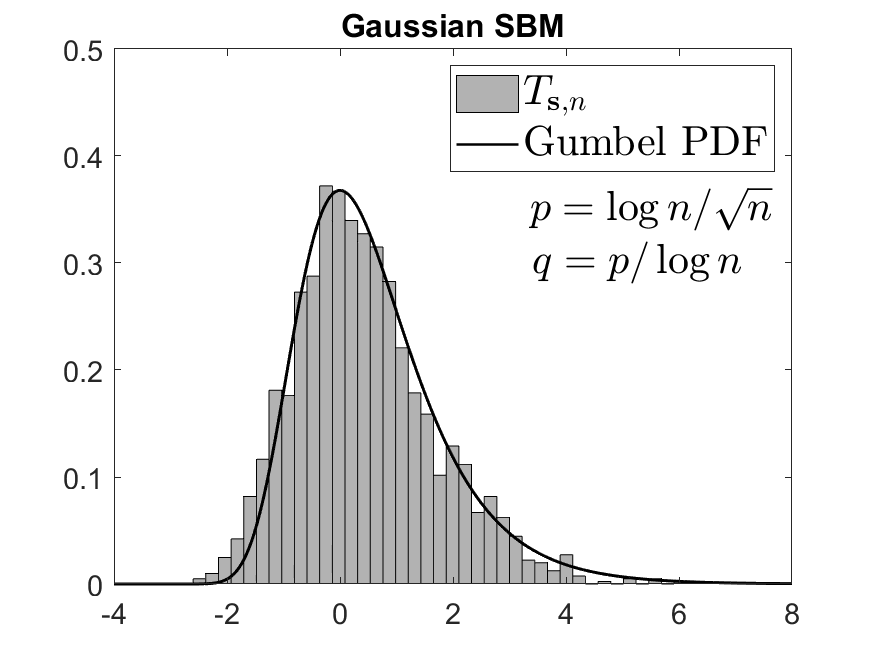}
    \end{minipage}
    \begin{minipage}{.33\textwidth}
        \includegraphics[width=5.9cm]{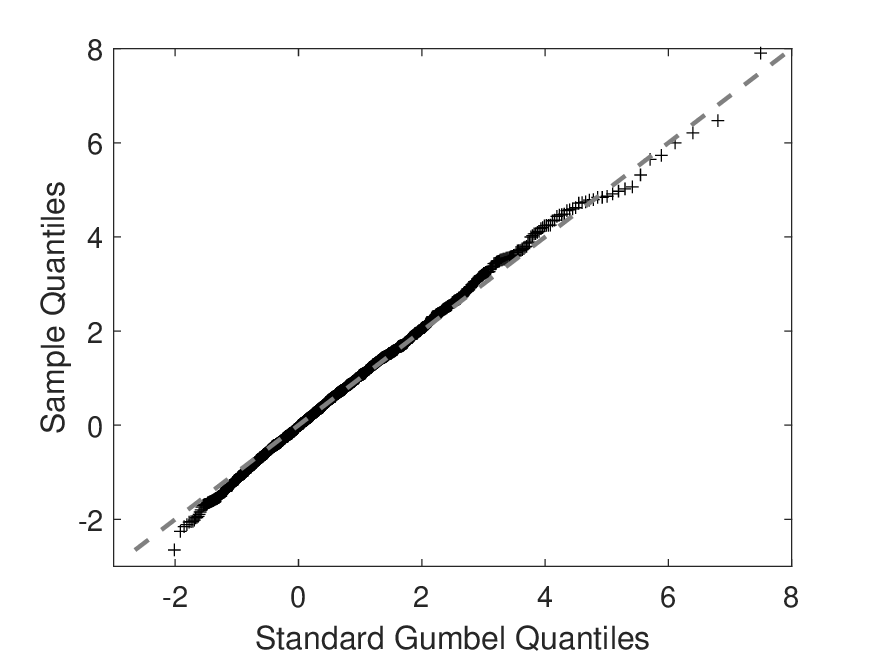}
    \end{minipage}%
    \begin{minipage}{.33\textwidth}
            \includegraphics[width=5.9cm]{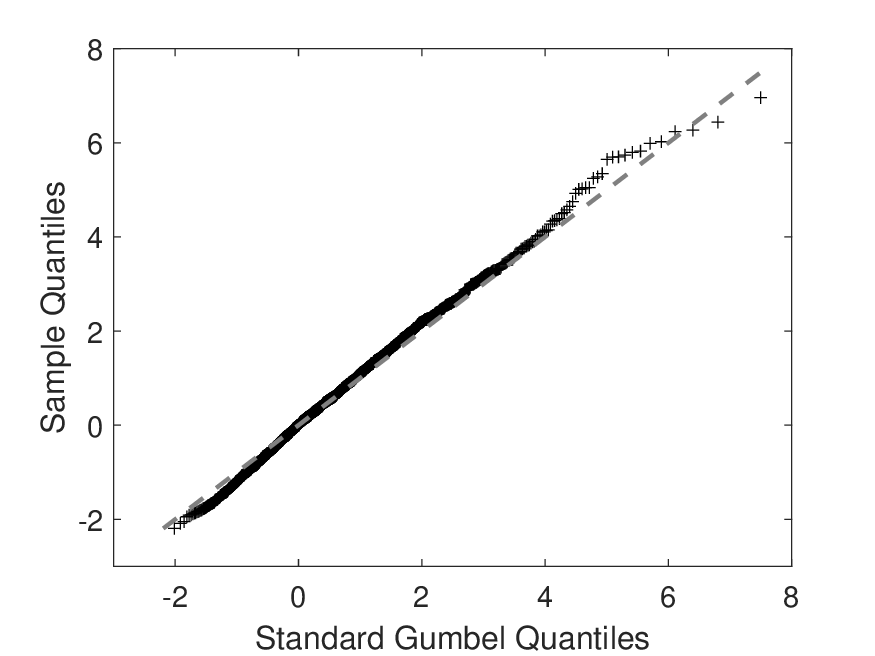}
    \end{minipage}%
    \begin{minipage}{.33\textwidth}
        \includegraphics[width=5.9cm]{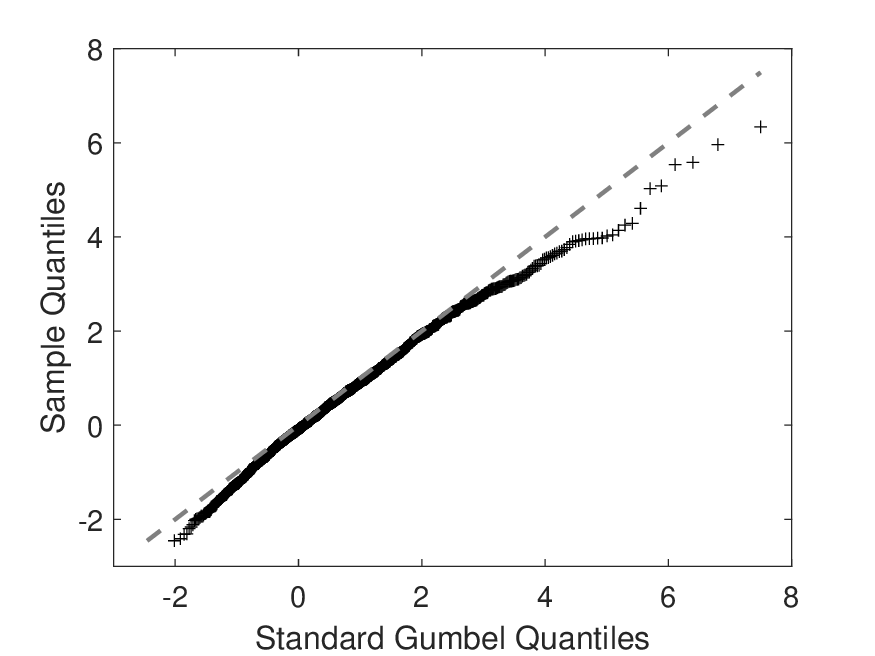}
    \end{minipage}
    \caption{Simulation under strong signal strength. See \cref{supp-section:balanced-two-block-SBM}.} 
    \label{supp-fig:SBM1}
\end{figure}

\begin{figure}[t]
    \centering
    \begin{minipage}{.33\textwidth}
        \includegraphics[width=5.9cm]{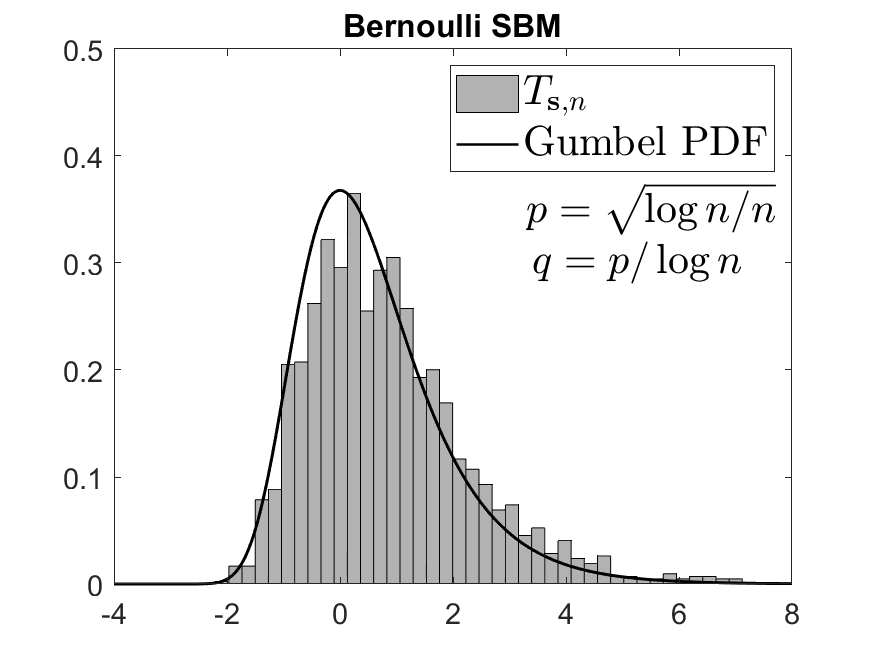}
    \end{minipage}%
    \begin{minipage}{.33\textwidth}
        \includegraphics[width=5.9cm]{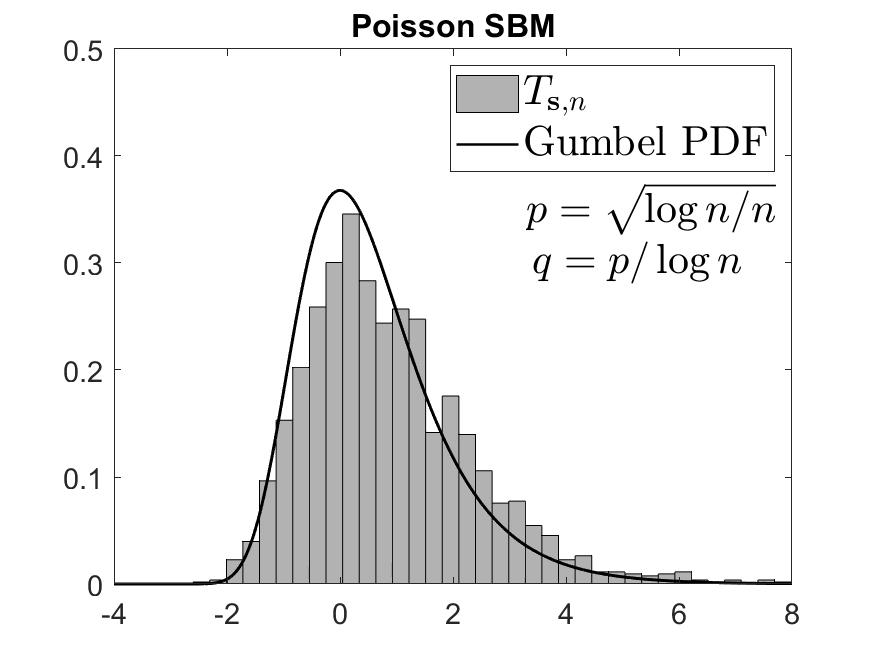}
    \end{minipage}%
    \begin{minipage}{.33\textwidth}
        \includegraphics[width=5.9cm]{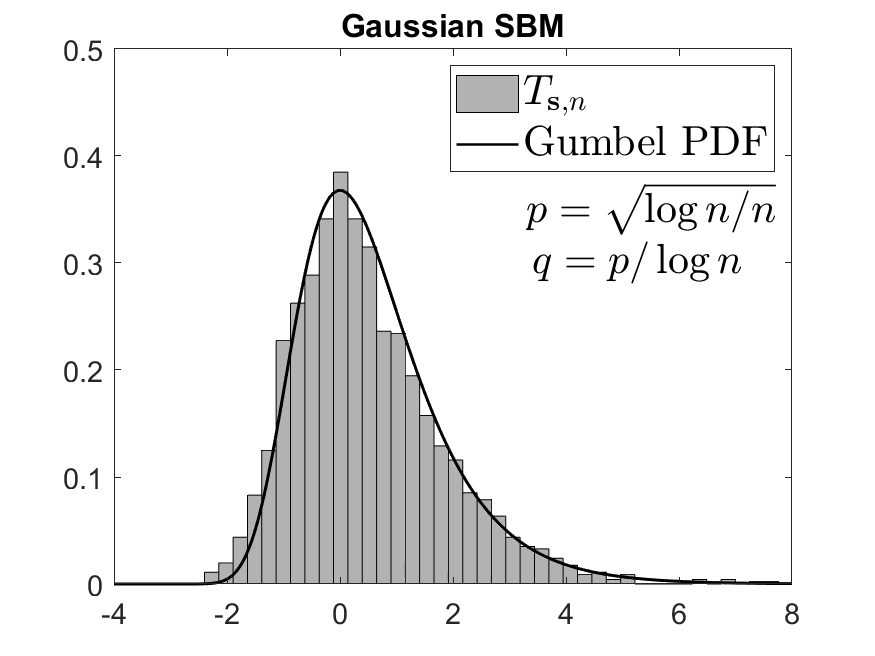}
    \end{minipage}
    \begin{minipage}{.33\textwidth}
        \includegraphics[width=5.9cm]{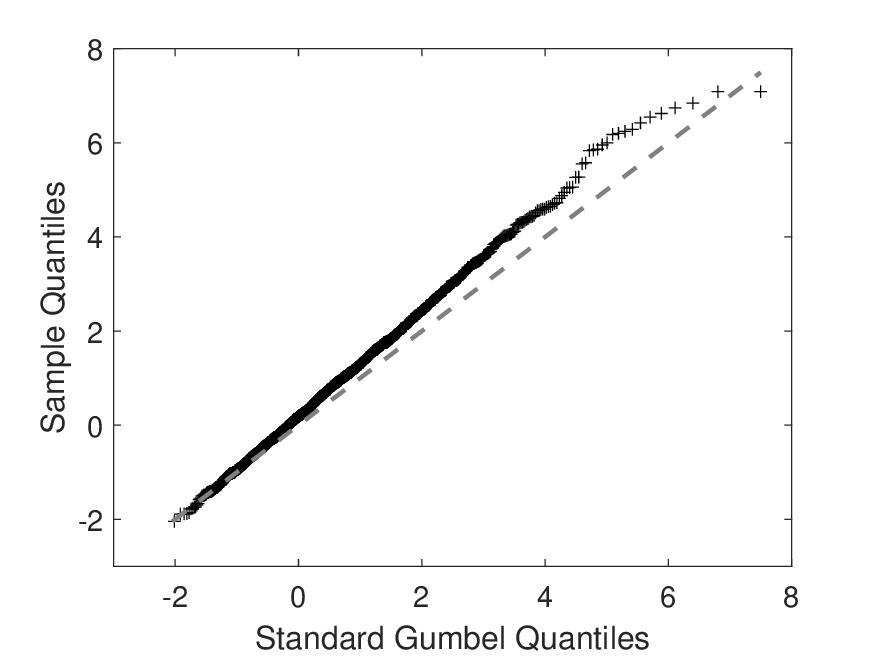}
    \end{minipage}%
    \begin{minipage}{.33\textwidth}
        \includegraphics[width=5.9cm]{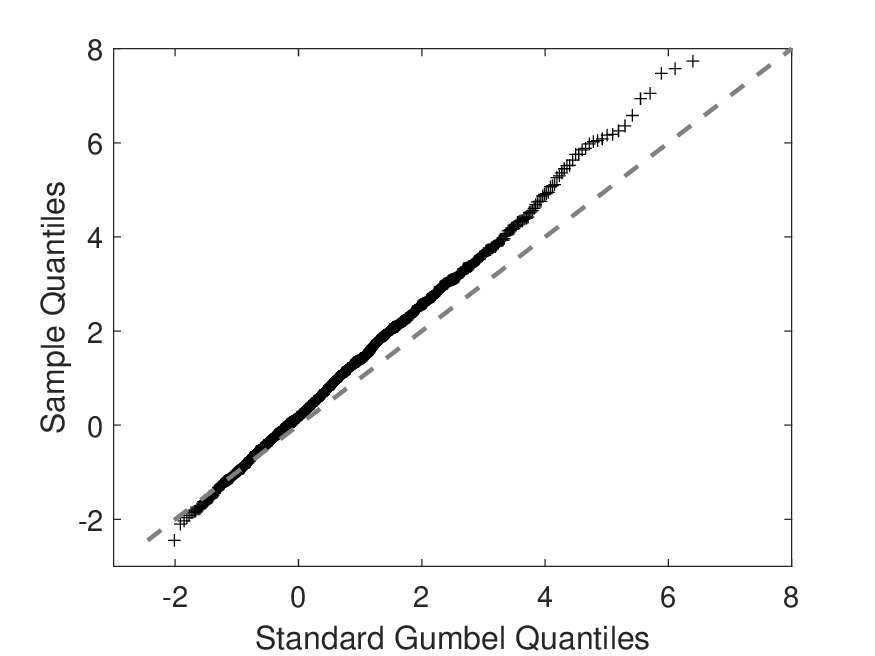}
    \end{minipage}%
    \begin{minipage}{.33\textwidth}
        \includegraphics[width=5.9cm]{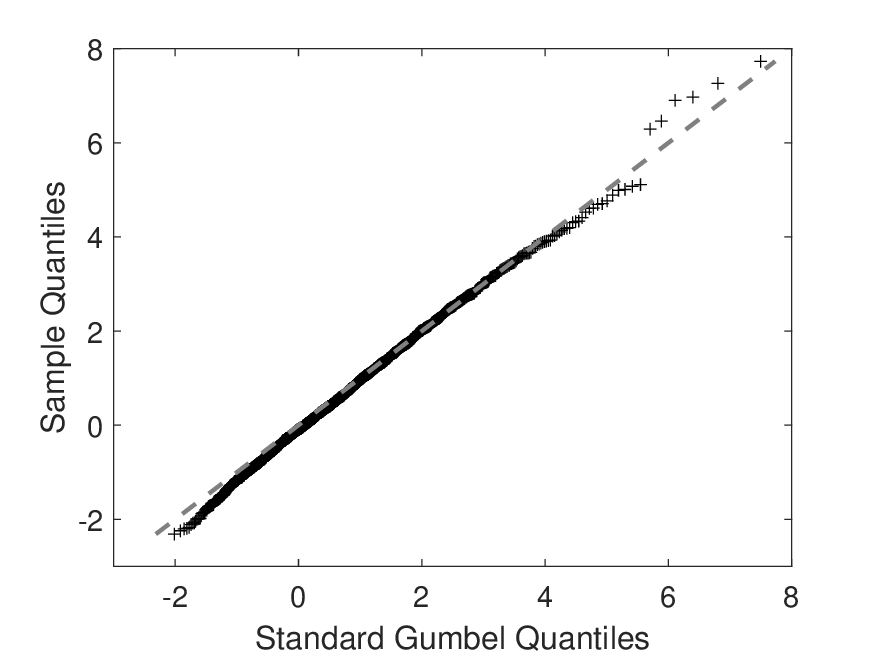}
    \end{minipage}
    \caption{Simulation under medium signal strength. See \cref{supp-section:balanced-two-block-SBM}.} 
    \label{supp-fig:SBM2}
\end{figure}

\begin{figure}[t]
    \centering
    \begin{minipage}{.33\textwidth}
        \includegraphics[width=5.9cm]{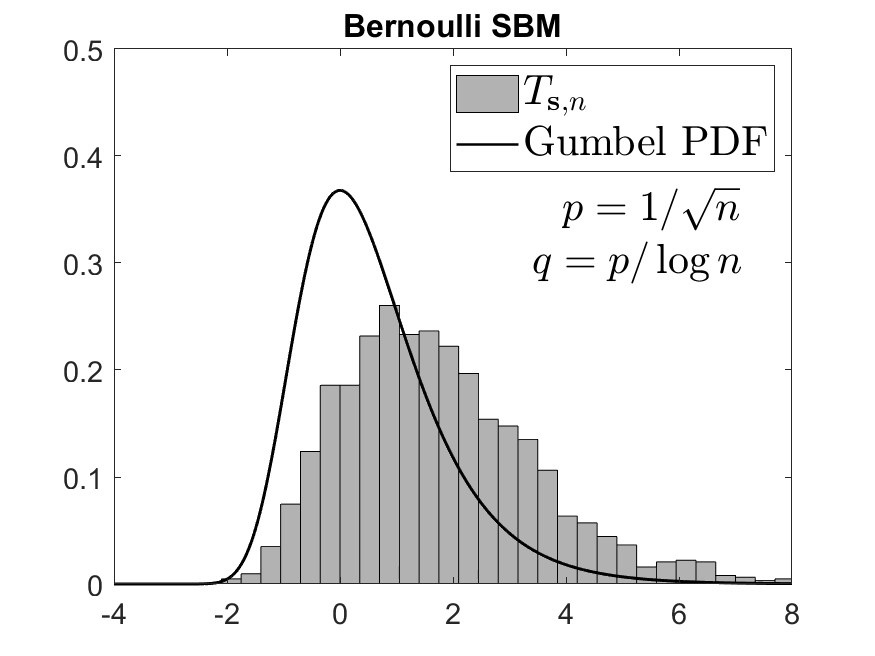}
    \end{minipage}%
    \begin{minipage}{.33\textwidth}
        \includegraphics[width=5.9cm]{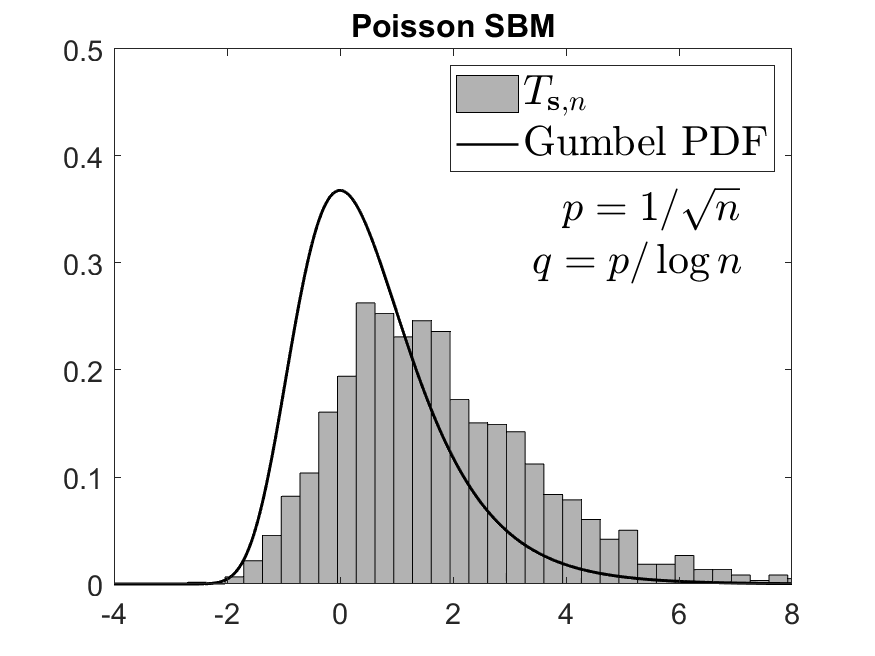}
    \end{minipage}%
    \begin{minipage}{.33\textwidth}
        \includegraphics[width=5.9cm]{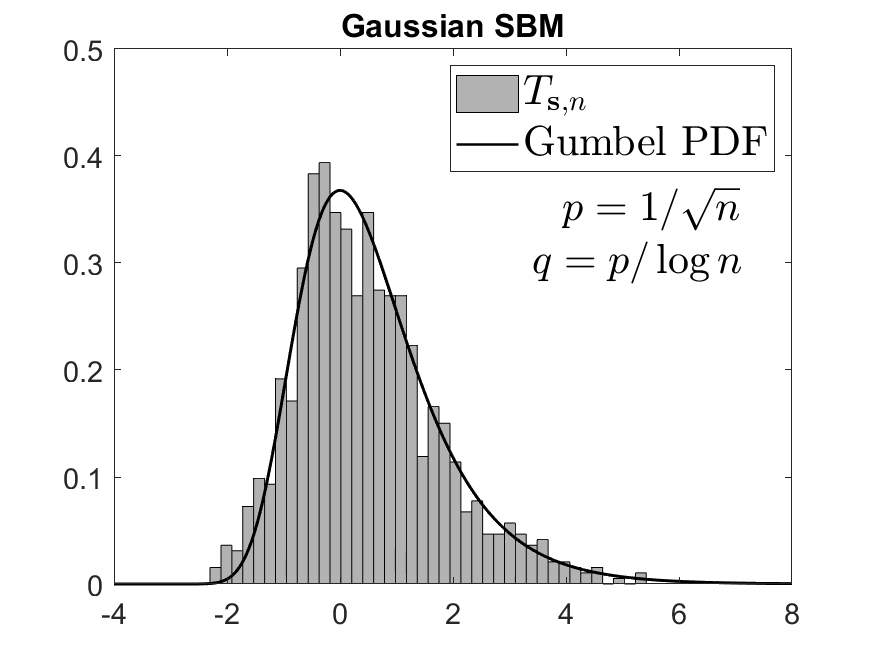}
    \end{minipage}
    \begin{minipage}{.33\textwidth}
        \includegraphics[width=5.9cm]{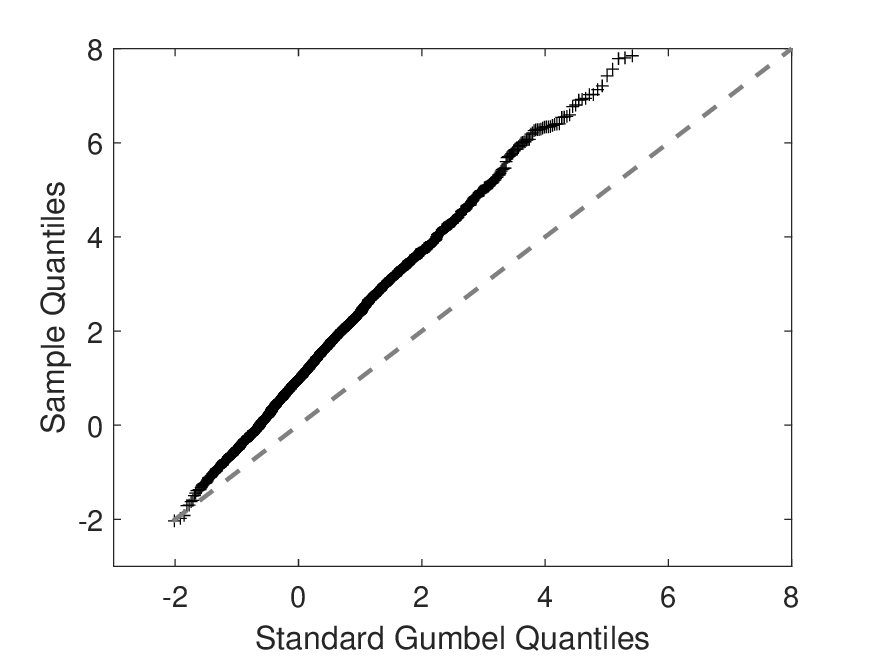}
    \end{minipage}%
    \begin{minipage}{.33\textwidth}
        \includegraphics[width=5.9cm]{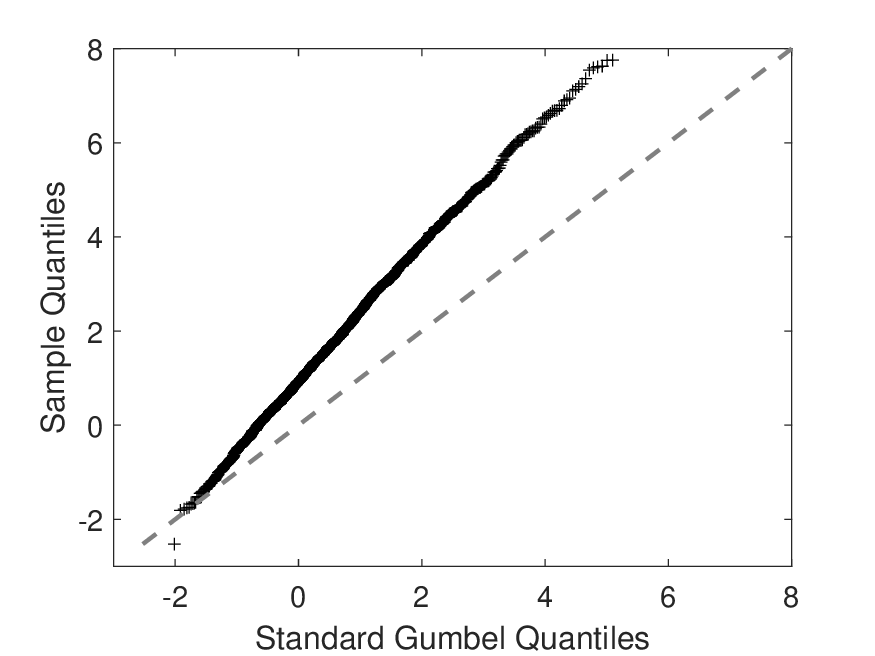}
    \end{minipage}%
    \begin{minipage}{.33\textwidth}
        \includegraphics[width=5.9cm]{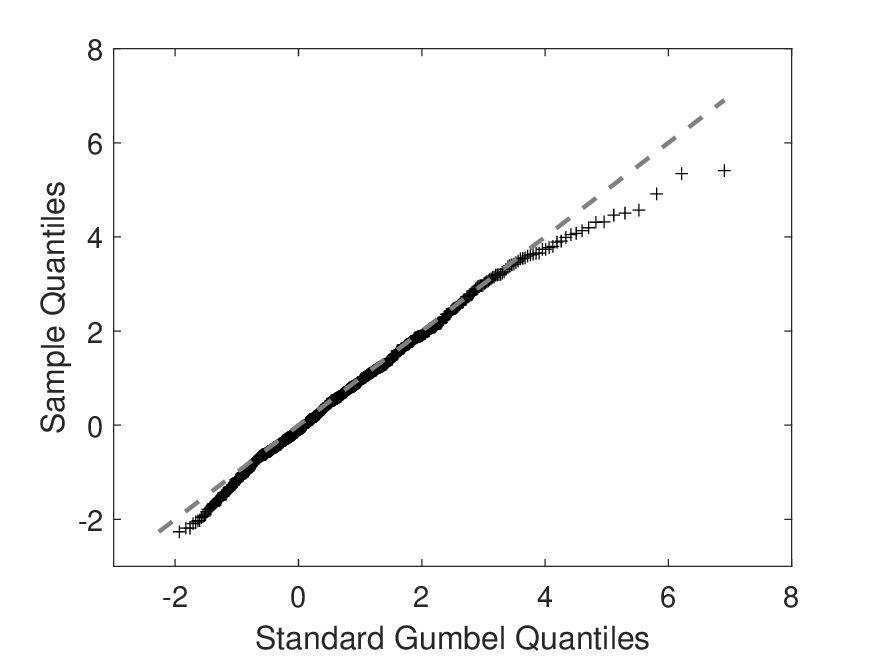}
    \end{minipage}
    \caption{Simulation under weak signal strength. See \cref{supp-section:balanced-two-block-SBM}.}
    \label{supp-fig:SBM3}
\end{figure}

\end{document}